\documentclass[12pt]{amsart}
\usepackage[utf8]{inputenc}
\usepackage[utf8]{inputenc}
\usepackage[T1]{fontenc}
\usepackage[all]{xy}
\usepackage{lipsum}
\usepackage{url}
\usepackage{tikz}
\usepackage{stackrel}
\usepackage{color}
\def\azul{}

\usepackage{amsmath,amsthm,amssymb}

\newcommand{\aspas}[1]{``{#1}''}

\usepackage{xcolor}

\usepackage{graphicx}

\theoremstyle{definition}

\newtheorem{theorem}{Theorem}[section]

\newtheorem{example}[theorem]{Example}

\newtheorem{proposition}[theorem]{Proposition}
\newtheorem{definition}[theorem]{Definition}

\newtheorem{remark}[theorem]{Remark}
\newtheorem{corollary}[theorem]{Corollary}
\numberwithin{equation}{section}

\begin{document}

\renewcommand{\bf}{\bfseries}
\renewcommand{\sc}{\scshape}

\newcommand{\TC}{\text{TC}}

\title[Sectional number of a morphism]%
{Sectional number of a morphism}
\author{Cesar A. Ipanaque Zapata}
\address{Departamento de Matem\'atica, IME-Universidade de S\~ao Paulo, Rua do Mat\~ao, 1010 CEP: 05508-090, S\~ao Paulo, SP, Brazil}
\curraddr{School of Mathematics, University of Minnesota,  
127 Vincent Hall, 206 Church St. SE, Minneapolis MN 55455, USA}
\email{cesarzapata@usp.br}
\thanks{The first author would like to thank grant\#2023/16525-7, grant\#2022/16695-7 and grant\#2022/03270-8, S\~{a}o Paulo Research Foundation (FAPESP) for financial support.}




\begin{abstract} The genus of a fibration was introduced by Schwarz in 1962. Given a continuous map $g:A\to B$, the usual sectional number $\text{sec}_u(g)$ is the least integer~$m$ such that $B$ can be covered by $m$ open subsets, each of which admits a local section of~$g$. Likewise, the sectional category $\text{secat}(g)$ is the least integer~$m$ such that $B$ can be covered by $m$ open subsets,  each of which admits a local homotopy section of~$g$. In the case that $g$ is a fibration, the usual sectional number and the sectional category of $g$ coincide with the Schwarz's genus of $g$. In this paper, we introduce a notion of sectional number for a morphism in a category with covers, which extends the usual sectional number and  sectional category (and, of course, Schwarz's notion). We study the invariance property, the behaviour under weak pullbacks and continuous functors, and present upper bounds in terms of a notion of LS category and projective covering number for an object. In addition, we introduce a notion of multiplicative cohomology theory on a category with covers, and we use it to present a cohomological lower bound for the sectional number. Several examples of this invariant are presented to support this theory. For instance, our sectional number yields a family of new numerical invariants in Top; we also express the twin prime conjecture and Goldbach's conjecture in terms of sectional numbers in the category of sets. Our unified concept is relevant in this new setting because it provides a foundation for sectional theory in category theory and builds bridges between new areas in mathematics.  
\end{abstract}

\maketitle

\tableofcontents


\section{Introduction}\label{secintro}
In this article, a \aspas{space} refers to a topological space, and a \aspas{map} always denotes a continuous map. Fibrations are considered in the Hurewicz sense. Moreover, \aspas{$\hookrightarrow$} denotes the inclusion map.

\medskip Consider a map $f:X\to Y$. The fundamental problem in algebraic topology is to find a section of $f$, that is, a map $g:Y\to X$ such that $(f\circ g)(y)=y$ for any $y\in Y$.  In practice, such a map $g$ does not exist. In this context, we are interested in local sections. Given a subset $A$ of $Y$, we say that a map $s:A\to X$ is a local section of $f$ if $f\circ s=\mathrm{incl}_A$, and a local homotopy section of $f$ if  $f\circ s\simeq \mathrm{incl}_A$, where $\mathrm{incl}_A:A\hookrightarrow Y$ is the inclusion map. The usual \textit{sectional number} $\text{sec}_u(f)$ is the least integer~$m$ such that $Y$ can be covered by $m$ open subsets, each of which admits a local section of~$f$. We set $\text{sec}_u(f)=\infty$ if no such $m$ exists. 

\medskip Note that in the case where $f$ is a fibration, the usual sectional number of $f$ coincides with Schwarz's genus of $f$ (see \cite{schwarz1966}), that is, $\mathrm{sec}_u(f)=\mathrm{genus}(f)$, for any fibration $f$. Likewise, as noted in \cite{berstein1961}, \textit{the sectional category} $\text{secat}(f)$ is the least integer~$m$ such that $Y$ can be covered by $m$ open subsets, each of which admits a local homotopy section of~$f$. Again, we set $\text{secat}(f)=\infty$ if no such $m$ exists. 

\medskip The \textit{Lusternik-Schnirelmann category} of a space $X$, denoted by cat$(X)$, is the least integer $m$ such that $X$ can be covered by $m$ open sets, all of which are contractible within $X$ (cf. \cite{cornea2003}). On the other hand, consider the \textit{index of nilpotence}
$$\text{nil}(S)=\min\{n:~\text{every product of $n$ elements in $S$ vanishes}\}$$
defined for a subset $S$ of a ring $R$. 

\medskip To estimate the sectional number and sectional category of a map $f:X\to Y$, the following well-known properties (see \cite{schwarz1966}, \cite[Chapter~4]{zapata2022phd}, \cite{cornea2003}) are generally used:
\begin{enumerate}
\item $\text{secat}(f)\leq \text{sec}_u(f)$.
\item If $f:X\to Y$ and $g:Y\to Z$ are maps, then $\text{sec}_u(g\circ f)\geq\text{sec}_u(g)$ and $\text{secat}(g\circ f)\geq\text{secat}(g)$.
\item If $f,g:X\to Y$ are homotopic maps, then $$\text{secat}(f)=\text{secat}(g).$$
    \item If the following diagram \begin{eqnarray*}
\xymatrix{ \rule{3mm}{0mm}& X^\prime \ar[r]^{\azul{\varphi'}} \ar[d]_{f^\prime} & X \ar[d]^{f} & \\ &
       Y^\prime  \ar[r]_{\,\,\varphi} &  Y &}
\end{eqnarray*} is a weak pullback, then $\text{sec}_u\hspace{.1mm}(f^\prime)\leq \text{sec}_u\hspace{.1mm}(f).$ 
\item $\text{secat}(f)\leq \text{cat}(Y)$ whenever $Y$ is path connected.
\item Let $h^\ast$ be any multiplicative cohomology theory on the homotopy category of pairs of spaces. One has \[\mathrm{nil}\left(\mathrm{Ker}\left(f^\ast:h^\ast(Y)\to h^\ast(X)\right)\right)\leq\text{secat}(f).\]
\end{enumerate}

\medskip Our work generalizes the Schwarz's genus, developed in~\cite{schwarz1966}. The genus was first studied extensively by Schwarz for fibrations in~\cite{schwarz1966} and later by Berstein and Ganea for arbitrary maps in~\cite{berstein1961}, under the name of sectional category. Pave\v{s}i\'{c} in~\cite{pavesic2019} presents a variant of genus, which, under certain conditions, coincides with Schwarz's notion. Several examples of these invariants appear in the literature \cite{cornea2003}, \cite{farber2003}, \cite{rudyak2010higher}, \cite{colman2017}. \cite{cohen2021},  \cite{zapata2022}. Furthermore, several authors have shown that sectional theory has connections with various areas in pure and applied mathematics \cite{farber2003projective}, \cite{farber2008}, \cite{zapata2020}, \cite{zapata2022}, \cite{zapata-daciberg2023}.  For instance:
\begin{itemize}
    \item Topological robotics: Given a space $X$ and the path fibration $e_2^X:X^{[0,1]}\to X\times X$ defined by $e_2^X(\gamma)=(\gamma(0),\gamma(1))$, the \textit{topological complexity} of $X$ is defined as $\text{TC}(X)=\text{sec}_u(e_2^X)$. This notion was first studied by Michael Farber in \cite{farber2003} (see also \cite{farber2008}).

    In general, given $n\geq 2$, $1\leq s\leq n$, and a map $f:X\to Y$, consider the map $e_{n,s}^f:X^{[0,1]}\to X^{n-s}\times Y^s$ defined by \[e_{n,s}^f(\gamma)=\left(\gamma(0),\ldots,\gamma\left(\dfrac{n-s-1}{n-1}\right),f\left(\gamma\left(\dfrac{n-s}{n-1}\right)\right),\ldots, f(\gamma(1))\right). 
\] From \cite{zapata2022}, the \textit{$(n,s)$-higher topological complexity} of $f$ is defined as $\text{TC}_{n,s}(f)=\text{sec}_u\left(e_{n,s}^f\right)$. Note that, $\text{TC}_{2,1}(1_X)=\text{TC}_{2,2}(1_X)=\text{TC}(X)$. 
    \item Immersion problem of $\mathbb{RP}^m$: From \cite{farber2003projective}, the topological complexity $\text{TC}(\mathbb{RP}^m)$ for any $m\neq 1,3,7$ coincides with the smallest integer $k$ such that the projective space $\mathbb{RP}^m$ admits an immersion into $\mathbb{R}^{k-1}$. 
    \item Fixed point property: From \cite{zapata2020}, a Hausdorff space $X$ has the fixed point property (FPP) if and only if \[\text{sec}_u\hspace{.1mm}\left(\pi_{2,1}^X:F(X,2)\to X\right)=2.\]
    \item Borsuk-Ulam property: from \cite{zapata-daciberg2023}, given a Hausdorff paracompact space $X$ that admits a fixed-point free involution $\tau$, let $q:X\to X/\tau$ be the quotient map. The sectional number $\text{sec}_u(q)$ minus one coincides with the smallest integer $n$ such that $X\stackrel{\mathbb{Z}_2} {\to} S^n$.
\end{itemize}

\medskip This work aims to establish the foundations of sectional theory within category theory and to create new connections between areas of pure and applied mathematics. For instance, we generalize properties (1) to (6) into a categorical setting. This is achieved by introducing a numerical invariant that measures the minimal number of certain local sections of a morphism. 

\medskip In more detail, motivated by Grothendieck topologies, we introduce the notion of a category with covers (Definition~\ref{defn:category-with-covers}). We then define the $\mathcal{R}$-sectional number $\text{$\mathcal{R}$-sec}$ and the $\mathcal{R}$-mono sectional number $\text{$\mathcal{R}$-msec}$ of a morphism $f$ in a category $\mathcal{C}$ with covers, where $\mathcal{R}$ is a relation over $\text{Mor}_{\mathcal{C}}(c,c^\prime)$ for each $c,c^\prime\in \mathcal{C}$ (Definition~\ref{defn:sectional-number}). The usual sectional number and sectional category are recovered in the category of topological spaces with the standard notion of open coverings, in the case that $\mathcal{R}$ is equal to the trivial relation, and when $\mathcal{R}$ corresponds to the homotopy relation, respectively (Example~\ref{exam:sectional-number-Top}). 

\medskip We note that a previous version of sectional number of a morphism in a category does not appear in the literature. 

\medskip This paper is organized as follows: In Section~\ref{sec:category-covers}, we introduce the notion of a  category with covers (Definition~\ref{defn:category-with-covers}) in terms of sieves, as well as the concept of covering for an object in a category with covers (Definition~\ref{def:covering}). We recall two well-known construction of sieves (Example~\ref{exam:pullback-sieve-sieve-genrated}). A key property of sieves is the sieve generated by a given subcollection of morphisms (Proposition~\ref{prop:sub-collection}). In Definition~\ref{defn:quasi-pullback}, we present the notion of weak pullback, leading to the concept of a category with weak pullbacks in Definition~\ref{defn:category-with-weak-puullbacks}. Proposition~\ref{lem;pullback} provides an explicit construction of the pullback of a generated sieve. We also introduce the notion of a quasi Grothendieck topology (Definition~\ref{def:quasi-topology}).

\medskip Several examples of categories with covers are presented: Example~\ref{exam:discrete} discusses the discrete and trivial Grothendieck topology; Example~\ref{exam:category-equal-group} shows that  any topology on a group $G$ coincides with the trivial Grothendieck topology; Example~\ref{exam:open-cover-space-X} presents the usual Grothendieck topology on the category $O(X)$; Example~\ref{exam:cover-set} introduces a quasi Grothendieck topology in the category of sets; Example~\ref{exam:open-cover-space} describes the open, \'{e}tale, and image quasi Grothendieck topologies in the category of spaces; Example~\ref{exam:category-groups} discusses topologies on the category of groups; Example~\ref{exam:subring-cover} presents a quasi Grothendieck topology on the category of rings; Example~\ref{exam:topology-category-r-modules} discusses topologies on the category of $R$-modules; and Example~\ref{exam:subgraph-cover} introduces a quasi Grothendieck topology on the category of graphs. 

\medskip We then obtain several examples of coverings, including Example~\ref{exam:coverings-discret-trivial}, Example~\ref{covering-top-usual-open-covering}, Example~\ref{exam:coverings-group-rings-submodules}, and Example~\ref{covering-graph-covering}). 

\medskip In Section~\ref{sec:sectional-number}, we define the numerical invariants, the $\mathcal{R}$-sectional number and the $\mathcal{R}$-mono sectional number of a morphism $f:X\to Y$ in a category $\mathcal{C}$ with covers, where $\mathcal{R}$ is a relation over $\text{Mor}_{\mathcal{C}}(c,c^\prime)$ for each $c,c^\prime\in \mathcal{C}$ (Definition~\ref{defn:sectional-number}). Basic  properties are presented in Remark~\ref{rem:ine-top}.  Example~\ref{exam:sec-codomain-initial} demonstrates that $\mathcal{S}\text{-}\mathcal{R}\text{-sec}(f:X\to I)=1$ for any topology $\mathcal{S}$ in $\mathcal{C}$ satisfying Axiom (T1) and any reflexive relation $\mathcal{R}$. Example~\ref{exam:sec-discrete-topology} shows that the sectional number with the discrete Grothendieck topology is usually 1. Proposition~\ref{prop:mono-sec-coincides} states that if the covers are by monomorphisms and $\mathcal{R}$ is an equivalence relation preserving composition, the $\mathcal{R}$-sectional number coincides with the $\mathcal{R}$-mono sectional number. 

\medskip Several examples support this theory:  Example~\ref{exam:sectional-number-Top} presents the sectional number in (Top, Open) and shows that our sectional number theory recovers the usual sectional theory; Remark~\ref{rem:sec-induce-functor} indicates that our sectional number yields a family of new numerical invariants in Top; Example~\ref{exam:sectional-number-Set} expresses the twin prime conjecture and Goldbach's conjecture in terms of sectional numbers in the category of sets; Example~\ref{exam:sectional-number-top} compares the usual sectional number and sectional category with respect to the sectional numbers considering the Image and \'{e}tale topologies in Top; Example~\ref{exam:topological-complexity} introduces the notion of \'{e}tale topological complexity; Example~\ref{sec-homomorphism-groups} discusses sectional numbers in the category of groups (related to the notion of covering number of a group); Example~\ref{sec-homomorphism-rings} presents a notion of sectional number in the category of rings (which recovers the notion of covering number of a ring); Example~\ref{sec-homomorphism-modules} discusses sectional numbers in the category of modules (related to  covering numbers of a module); Example~\ref{exam:sectional-in-group} presents a notion of sectional number in a group and shows that we can express the existence of solutions to equations in terms of this numerical invariant; and Example~\ref{exam:sectional-number-graph} presents a notion of sectional number in the category of graphs. 

\medskip Proposition~\ref{rem:via-subobjects} states that the mono sectional number can be expressed in terms of coverings by subobjects. Proposition~\ref{prop:characterisation-sieves} presents a characterisation of (mono) sectional number in terms of generated sieves. Theorem~\ref{r-fe-invariance} establishes the $\mathcal{R}$-FE-invariance of $\text{$\mathcal{R}$-}\mathrm{sec}(-)$ and $\text{$\mathcal{R}$-}\mathrm{msec}(-)$. Theorem~\ref{quasi-pullback-sec} demonstrates the behavior of (mono) sectional number under weak pullbacks. In Theorem~\ref{thm:continuous-functor}, we study the sectional number under continuous functors. 

\medskip We also present upper bounds for the sectional number. For this purpose, we introduce the notions of LS category (Definition~\ref{defn:LS-category}) and projective covering number (Definition~\ref{defn:projective-covering-number}). Theorem~\ref{thm:upper-bound-LS-category} provides an upper bound for the $\mathcal{R}$-sectional number in terms of LS category, while   Theorem~\ref{thm:upper-bound-covering-number} presents an upper bound for the sectional number in terms of the projective covering number.

\medskip In Section~\ref{sec:cohomology}, Example~\ref{exam:lim-colim-construction} introduces a key construction to define a notion of cohomology theory. Definition~\ref{cohomology-theory} presents the notion of cohomology theory for a category with covers along with certain properties. Proposition~\ref{cohomology-retratopordefor} establishes the cohomology of retracts. Example~\ref{exam:lim-colim-topology} presents the lim-colim property. Definition~\ref{multiplicative-cohomology} introduces the notion of multiplicative cohomology theory. Proposition~\ref{diagonal-cup} establishes a key property of the cup product. Example~\ref{exam:multiplicative-generalized-cohomology-theory} shows that any multiplicative generalized cohomology theory is a multiplicative cohomology theory on $(\text{Top},\text{Open})$. Theorem~\ref{thm:cohomo-ic} provides a cohomological lower bound for the $\mathcal{R}$-sectional number.


\section{Categories with covers}\label{sec:category-covers}
 In this section, we introduce the concept of a category with covers (Definition~\ref{defn:category-with-covers}), defined in terms of sieves, as well as the notion of a covering for an object within a category with covers (Definition~\ref{def:covering}). These concepts are inspired by the notion of a Grothendieck topology (Definition~\ref{def:topology}). We adopt the notations used in  \cite[Appendix B.1]{lurieultra2018}. 

\subsection{Sieves} Let $\mathcal{C}$ be a category. A subcategory $\mathcal{A}$ of  $\mathcal{C}$ is called \textit{full} if, for any objects  $a,a^\prime\in \mathcal{A}$, we have $\text{Mor}_{\mathcal{A}}(a,a^\prime)=\text{Mor}_{\mathcal{C}}(a,a^\prime)$. A \textit{sieve} on $\mathcal{C}$ is a full subcategory $S\subset\mathcal{C}$ with the following property: for every morphism $f:C'\to C$ in $\mathcal{C}$, if $C$ belongs to $S$, then $C'$ also belongs to $S$. Not that any category is a sieve on itself. 

 \medskip Let $\mathcal{C}$ be a category and $C\in\mathcal{C}$ be an object. The \textit{overcategory} $\mathcal{C}_{|C}$ is the category whose objects are morphisms with codomain $C$, and whose morphisms are commutative triangles of the form:
\begin{eqnarray*}
\xymatrix{ U\ar[rd]^{}\ar@{-->}[rr]^{} & &U'\ar[ld]^{}  & \\
        &  C & &} 
\end{eqnarray*}

A \textit{sieve on} $C$ is a sieve on the overcategory $\mathcal{C}_{|C}$. Let $S$ be a full subcategory of $\mathcal{C}_{|C}$. Then $S$ is a sieve on $C$ if and only if $f\in S$ implies $f\circ g\in S$ for any morphism $g$ with $\text{codom}(g)=\text{dom}(f)$. Note that, $\mathcal{C}_{\mid C}$ is itself a sieve on $C$. 

\medskip Let $\mathcal{C}$ be a category. An \textit{initial object} in $\mathcal{C}$ is an object $I$ such that for every object $X$ in $\mathcal{C}$, there exists exactly one morphism $I\to X$. A \textit{terminal object} in $\mathcal{C}$ is an object $T$ such that for every object $X$ in $\mathcal{C}$, there exists exactly one morphism $X\to T$. If an object is both initial and terminal, it is call a \textit{zero object}. 

\begin{remark}\label{rem:initial-object}
Let $I$ be an initial object in $\mathcal{C}$. For each object $C$ in $\mathcal{C}$, the unique morphism $I\to C$ is contained in any sieve $S$ on $C$.    
\end{remark}

We recall two well known construction of sieves.

\begin{example}\label{exam:pullback-sieve-sieve-genrated}
    Let $\mathcal{C}$ be a category.
    \begin{enumerate}
        \item[(1)] \textit{Pullback of a sieve}: Let $f:D\to C$ be a morphism in $\mathcal{C}$. If $S\subset \mathcal{C}_{\mid C}$ is a sieve on the object $C$, then the full subcategory $f^\ast S\subset \mathcal{C}_{\mid D}$ defined by \[f^\ast S=\left\{(E\stackrel{g}{\to} D)\in \mathcal{C}_{\mid D}:~(E\stackrel{f\circ g}{\to} C)\in S\right\}\] is a sieve on $D$. This sieve $f^\ast S$ is called the \textit{pullback} of the sieve $S$ through $f$. Notably, we have $f^\ast \mathcal{C}_{\mid C}=\mathcal{C}_{\mid D}$ for any morphism $f:D\to C$, and $(1_C)^\ast S=S$ for any sieve $S$ on $C$. 
        \item[(2)] \textit{Sieve generated by morphisms}: Let $C\in\mathcal{C}$ be an object, and let $\{f_i:C_i\to C\}_{i\in I}$ be a collection of morphisms with codomain $C$. The full subcategory  $\mathcal{C}_{\mid C}^{(0)}\{f_i:C_i\to C\}_{i\in I}\subset \mathcal{C}_{\mid C}$ is defined as follows:
\begin{eqnarray*}
\mathcal{C}_{\mid C}^{(0)}\{f_i:C_i\to C\}_{i\in I} &=& \{(D\stackrel{g}{\to} C)\in \mathcal{C}_{\mid C}:~ g \text{ factors as }\\
& & D\to C_i\stackrel{f_i}{\to } C  \text{ for some } i\in I\}.
\end{eqnarray*}
This category represents the smallest sieve (with respect to  inclusion) that contains each morphism $f_i$ for $i\in I$. We will refer to $\mathcal{C}_{\mid C}^{(0)}\{f_i:C_i\to C\}_{i\in I}$ as the \textit{sieve generated by the morphisms} $f_i$. 

It is important to note that if $S$ is a sieve on $C$, then we can express it as $S=\mathcal{C}_{\mid C}^{(0)}\{f\colon E\to C\}_{f\in S}$. We say that a sieve $S$ on $C$ is \textit{finitely generated} if it is generated by a finite collection of morphisms having codomain $C$.  
    \end{enumerate}
\end{example}

For example, consider the following:

\begin{example}\label{exam:identity}
The sieve generated by the identity morphism $1_C:C\to C$ coincides with the entire sieve $\mathcal{C}_{\mid C}$. Additionally, the sieve generated by any isomorphism $D\to C$ also coincides with  $\mathcal{C}_{\mid C}$.
\end{example}

We have the following statement. 

\begin{proposition}[Sieve generated by a subcollection]\label{prop:sub-collection}
Let $C\in\mathcal{C}$ be an object and let $\{f_i:C_i\to C\}_{i\in I}$ be a collection of morphisms having codomain $C$.
\begin{enumerate}
    \item[(1)] Suppose that  $\{f_j:C_j\to C\}_{j\in J}$ is a subcollection of $\{f_i:C_i\to C\}_{i\in I}$ (i.e., $J\subset I$) such that for each $i\in I$, there exists $j\in J$ for which $f_i$ factors as a composition $C_i\to C_j\stackrel{f_j}{\to } C$. Then, we have \[\mathcal{C}_{\mid C}^{(0)}\{f_i:C_i\to C\}_{i\in I}=\mathcal{C}_{\mid C}^{(0)}\{f_j:C_j\to C\}_{j\in J}.\]   
    \item[(2)] Suppose there exists a finite collection  $\{g_j:D_j\to C\}_{j=1}^{m}$ of morphisms with codomain $C$ such that \[\mathcal{C}_{\mid C}^{(0)}\{f_i:C_i\to C\}_{i\in I}=\mathcal{C}_{\mid C}^{(0)}\{g_j:D_j\to C\}_{j=1}^{m}.\]  Then,  there exists a subcollection $\{f_{k_j}:C_{k_j}\to C\}_{j=1}^{m}$ of $\{f_i:C_i\to C\}_{i\in I}$ (i.e., $\{k_1,\ldots,k_m\}\subset I$) such that for each $i\in I$, there exists $\alpha_i\in \{k_1,\ldots,k_m\}$ such that $f_i$ factors as \[C_i\to C_{\alpha_i}\stackrel{f_{\alpha_i}}{\to } C.\] In particular, by Item (1), we have $\mathcal{C}_{\mid C}^{(0)}\{f_i:C_i\to C\}_{i\in I}=\mathcal{C}_{\mid C}^{(0)}\{f_{k_j}:C_{k_j}\to C\}_{j=1}^{m}$.
\end{enumerate} 
\end{proposition}
\begin{proof}
 \noindent\begin{enumerate}
     \item[(1)] Since $\{f_j:C_j\to C\}_{j\in J}$ is a subcollection of $\{f_i:C_i\to C\}_{i\in I}$, we have $\mathcal{C}_{\mid C}^{(0)}\{f_j:C_j\to C\}_{j\in J}\subset \mathcal{C}_{\mid C}^{(0)}\{f_i:C_i\to C\}_{i\in I}$. Conversely, for each $i\in I$, since $f_i$ factors through some $f_j$, we have $f_i\in \mathcal{C}_{\mid C}^{(0)}\{f_j:C_j\to C\}_{j\in J}$, yielding  $\mathcal{C}_{\mid C}^{(0)}\{f_i:C_i\to C\}_{i\in I}\subset \mathcal{C}_{\mid C}^{(0)}\{f_j:C_j\to C\}_{j\in J}$. Thus, the equality holds. 
     \item[(2)] For each $j=1,\ldots,m$, since $g_j\in \mathcal{C}_{\mid C}^{(0)}\{f_i:C_i\to C\}_{i\in I}$, there exists $k_j\in I$ such that $g_j$ factors as a composition \[D_j\to C_{k_j}\stackrel{f_{k_j}}{\to } C.\] Thus, we have a subcollection $\{f_{k_j}:C_{k_j}\to C\}_{j=1}^{m}$ of $\{f_i:C_i\to C\}_{i\in I}$ (i.e., $\{k_1,\ldots,k_m\}\subset I$).
     
     Furthermore, since $f_i\in \mathcal{C}_{\mid C}^{(0)}\{g_j:D_j\to C\}_{j=1}^{m}$, there exists $j_i\in \{1,\ldots,m\}$ such that $f_i$ factors as a composition \[C_i\to D_{j_i}\stackrel{g_{j_i}}{\to } C.\] For this $j_i$, we have $\alpha_i\in \{k_1,\ldots,k_m\}$ such that $g_{j_i}$ factors as a composition \[D_{j_i}\to C_{\alpha_i}\stackrel{f_{\alpha_i}}{\to } C.\] Therefore, $f_i$ factors as a composition \[C_i\to C_{\alpha_i}\stackrel{f_{\alpha_i}}{\to } C.\]
 \end{enumerate}   
\end{proof}

\subsection{Weak pullback} Motivated by \cite[p. 235]{maclane1978}, we introduce the concept of a weak pullback.

\begin{definition}[Weak pullback]\label{defn:quasi-pullback}
 Let $\mathcal{C}$ be a category. A \textit{weak pullback} in $\mathcal{C}$ is defined by a strictly commutative diagram in $\mathcal{C}$ of the form:
\begin{eqnarray}\label{xfy}
\xymatrix{ \rule{3mm}{0mm}& X^\prime \ar[r]^{\varphi'} \ar[d]_{f^\prime} & X \ar[d]^{f} & \\ &
       Y^\prime  \ar[r]_{\,\,\varphi} &  Y &}
\end{eqnarray} 
such that for any strictly commutative diagram as shown on the left side of~(\ref{diagramadoble}), there exists a (not necessarily unique) morphism $h:Z\to X^\prime$ that makes the diagram on the right side of~(\ref{diagramadoble}) commute:
\begin{eqnarray}\label{diagramadoble}
\xymatrix{
Z \ar@/_10pt/[dr]_{\alpha} \ar@/^30pt/[rr]^{\beta} & & X \ar[d]^{f}  & & &
Z\rule{-1mm}{0mm} \ar@/_10pt/[dr]_{\alpha} \ar@/^30pt/[rr]^{\beta}\ar[r]^{h} & 
X^\prime \ar[r]^{\azul{\varphi'}} \ar[d]_{f^\prime} & X \\
& Y^\prime  \ar[r]_{\,\,\varphi} &  Y & & & & Y^\prime &  \rule{3mm}{0mm}}
\end{eqnarray}   
\end{definition}

Note that if ~(\ref{xfy}) is a pullback in $\mathcal{C}$, then it is a weak pullback. The concept of a weak pullback in the category of topological spaces was introduced in~\cite{zapata2022} under the name of \aspas{quasi pullback}. 

\begin{definition}[Category with weak pullbacks]\label{defn:category-with-weak-puullbacks}
 Let $\mathcal{C}$ be a category. We say that $\mathcal{C}$ is a \textit{category with weak pullbacks} if, for any morphisms $f:X\to Y$ and $\varphi:Y'\to Y$ in $\mathcal{C}$, there exist an object $X'\in\mathcal{C}$ and morphisms $f':X'\to Y'$ and $\varphi':X'\to X$ in $\mathcal{C}$  such that the following diagram  \begin{eqnarray*}
\xymatrix{ \rule{3mm}{0mm}& X^\prime \ar[r]^{\varphi'} \ar[d]_{f^\prime} & X \ar[d]^{f} & \\ &
       Y^\prime  \ar[r]_{\,\,\varphi} &  Y &}
\end{eqnarray*}  is a weak pullback. Otherwise, we say that $\mathcal{C}$ is a \textit{category without weak pullbacks}.
\end{definition}

Note that any category with pullbacks is also a category with weak pullbacks. Additionally, there exist categories without weak pullbacks. 

\medskip Now we present the following statement, which provides an explicit construction of the pullback of a generated sieve. 

\begin{proposition}[Pullback of a generated sieve]\label{lem;pullback}
Let $\mathcal{C}$ be a category with weak pullbacks, $f:D\to C$ be a morphism in $\mathcal{C}$, and $\{f_i:C_i\to C\}_{i\in I}$ be a collection of morphisms with codomain $C$. Then, the following equality holds:
\[f^\ast \left(\mathcal{C}_{\mid C}^{(0)}\{f_i:C_i\to C\}_{i\in I}\right)=\mathcal{C}_{\mid D}^{(0)}\{\widetilde{f}_i:\widetilde{C}_i\to D\}_{i\in I},\] where each $\widetilde{f}_i:\widetilde{C}_i\to D$ is defined by a weak pullback as follows:  
\begin{eqnarray}\label{xfyy}
\xymatrix{ \rule{3mm}{0mm}& \widetilde{C}_i \ar[r]^{ } \ar[d]_{\widetilde{f}_i} & C_i \ar[d]^{f_i} & \\ &
       D  \ar[r]_{\,\,f} &  C &}
\end{eqnarray}  
\end{proposition}
\begin{proof}
By the commutativity of the diagram~(\ref{xfyy}), for each $i\in I$, $f\circ \widetilde{f}_i$ factors as a composition $\widetilde{C}_i\to C_i\stackrel{f_i}{\to } C$. Therefore, $f\circ \widetilde{f}_i\in \mathcal{C}_{\mid C}^{(0)}\{f_i:C_i\to C\}_{i\in I}$. This implies that $\widetilde{f}_i\in f^\ast \left(\mathcal{C}_{\mid C}^{(0)}\{f_i:C_i\to C\}_{i\in I}\right)$ for each $i\in I$, establishing the inclusion  \[\mathcal{C}_{\mid D}^{(0)}\{\widetilde{f}_i:\widetilde{C}_i\to D\}_{i\in I}\subset f^\ast \left(\mathcal{C}_{\mid C}^{(0)}\{f_i:C_i\to C\}_{i\in I}\right)\]. 

Next, we prove the other inclusion. Suppose $(E\stackrel{g}{\to} D)\in f^\ast \left(\mathcal{C}_{\mid C}^{(0)}\{f_i:C_i\to C\}_{i\in I}\right)$. This means that $(E\stackrel{g}{\to} D\stackrel{f}{\to} C)\in \mathcal{C}_{\mid C}^{(0)}\{f_i:C_i\to C\}_{i\in I}$, so it can be factored as a composition $E\stackrel{\varphi}{\to} C_i\stackrel{f_i}{\to} C$ for some $i\in I$. 

Consider the strictly commutative diagram on the left-hand side of~(\ref{diagramadoblee}). Since the diagram~(\ref{xfyy}) is a weak pullback, there exists a (not necessarily unique) morphism $h:E\to \widetilde{C}_i$ that makes the right-hand side of~(\ref{diagramadoblee}) strictly commutative:
\begin{eqnarray}\label{diagramadoblee}
\xymatrix{
E \ar@/_10pt/[dr]_{g} \ar@/^30pt/[rr]^{\varphi} & & C_i \ar[d]^{f_i}  & & &
E\rule{-1mm}{0mm} \ar@/_10pt/[dr]_{g} \ar@/^30pt/[rr]^{\varphi}\ar[r]^{h} & 
\widetilde{C}_i \ar[r]^{ } \ar[d]_{\widetilde{f}_i} & C_i \\
& D  \ar[r]_{\,\,f} &  C & & & & D &  \rule{3mm}{0mm}}
\end{eqnarray} Then, $(E\stackrel{g}{\to} D)$ factors as a composition $E\stackrel{h}{\to} \widetilde{C}_i\stackrel{\widetilde{f}_i}{\to} D$, which implies $(E\stackrel{g}{\to} D)\in \mathcal{C}_{\mid D}^{(0)}\{\widetilde{f}_i:\widetilde{C}_i\to D\}_{i\in I}$. 

Thus, we conclude that $f^\ast \left(\mathcal{C}_{\mid C}^{(0)}\{f_i:C_i\to C\}_{i\in I}\right)=\mathcal{C}_{\mid D}^{(0)}\{\widetilde{f}_i:\widetilde{C}_i\to D\}_{i\in I}$.
\end{proof}

\subsection{Covers} A central idea in this work is the concept of a category with covers (cf. \cite[p. 108]{maclane1978}).

\begin{definition}[Category with covers]\label{defn:category-with-covers}
A \textit{topology} on a category $\mathcal{C}$ is a procedure $\mathcal{S}$ that assigns to each object $C\in\mathcal{C}$ a collection of sieves on $C$, which we call \textit{$\mathcal{S}$-covering sieves} or simply \textit{covering sieves}. A \textit{category with covers} is defined as a pair $\left(\mathcal{C},\mathcal{S}\right)$, where $\mathcal{C}$ is a category and $\mathcal{S}$ is a topology on $\mathcal{C}$. 
\end{definition}

 A direct example of category with covers is a category equipped with a Grothendieck topology. Therefore, we will briefly review the concept of a Grothendieck topology.

\begin{definition}[Grothendieck topology]\label{def:topology}
Let $\mathcal{C}$ be a category. A \textit{Grothendieck topology} on  $\mathcal{C}$ is a procedure that assigns to each object $C\in\mathcal{C}$ a collection of sieves on $C$, referred to as \textit{covering sieves}. This assignment must satisfy the following properties: 
\begin{itemize}
    \item[(T1)] For each object $C\in \mathcal{C}$, the sieve $\mathcal{C}_{\mid C}$ is a covering sieve on $C$.
    \item[(T2)] For each morphism $f:D\to C$ in $\mathcal{C}$ and any covering sieve $S\subset\mathcal{C}_{\mid C}$ on $C$, the  pullback $f^\ast S$ is a covering sieve on $D$.
    \item[(T3)] Let $S\subset\mathcal{C}_{\mid C}$ be a covering sieve on an object $C\in\mathcal{C}$, and let $S'\subset\mathcal{C}_{\mid C}$ be another sieve such that for every morphism $f:D\to C$ in $S$, the pullback $f^\ast S'$ is a covering sieve on $D$. Then, $S'$ is a covering sieve on $C$.
\end{itemize}
\end{definition}

\begin{remark}
    Note that any category equipped with a Grothendieck topology is also a category with covers. In this context, the covering sieves (as defined in  Definition~\ref{defn:category-with-covers}) correspond to the covering sieves defined in the context of Grothendieck topology (see Definition~\ref{def:topology}).
\end{remark}

We now introduce the following Grothendieck topologies.

\begin{example}\label{exam:discrete}  
   Let $\mathcal{C}$ be a category.
 \begin{enumerate}
     \item[(1)] [Discrete Grothendieck topology] The procedure that assigns to each object $C\in\mathcal{C}$ the collection of all sieves on $C$--meaning any sieve is considered a covering sieve--is a Grothendieck topology known as the \textit{discrete Grothendieck topology} on $\mathcal{C}$. This topology is the biggest (finest) among all topologies on $\mathcal{C}$.
     \item[(2)] [Trivial Grothendieck topology](cf. \cite[p. 113]{maclane1992}) The procedure that assigns to each object $C\in\mathcal{C}$ the collection consisting solely of the sieve $\mathcal{C}_{\mid C}$--meaning that $\mathcal{C}_{\mid C}$ is the only covering sieve on $C$--defines a Grothendieck topology known as the \textit{trivial Grothendieck topology} on $\mathcal{C}$. It is important to note that this topology is the smallest (coarsest) among all topologies on $\mathcal{C}$.
 \end{enumerate}   
\end{example}

The following example illustrates that the discrete and trivial Grothendieck topologies can coincide. 

\begin{example}[A group as a category]\label{exam:category-equal-group}
    Let $G$ be a group. This group can be viewed as a category with a single object, denoted $\ast$, where the morphisms correspond to the elements of the group $G$. The group multiplication serves as the composition operation for the morphisms in this category. Consequently, the identity element $1$ of the group corresponds to the identity morphism $1_\ast$ of the unique object $\ast$. This demonstrates that a group is equivalent to a category with only one object in which each morphism is an isomorphism \cite[p. 11]{maclane1992}. 
    
    Since any sieve $S$ on the unique object $\ast$ corresponds to $G_{\mid \ast}$, and the identity morphism $1$ can be expressed as a composition $g^{-1}g$ for any morphism $g\in S$, it follows that any topology on $G$ coincides with the trivial Grothendieck topology. In particular, the discrete and trivial Grothendieck topologies on $G$ are the same.  
\end{example}

Now, we will present the usual Grothendieck topology on the category $O(X)$.

\begin{example}[Category $O(X)$](cf. \cite[p. 111]{maclane1992})\label{exam:open-cover-space-X}
    Let $X$ be a space, and let $O(X)$ be the category whose objects are the open subsets of $X$. There is exactly one morphism $U\to V$ if and only if $U\subset V$, with this morphism being the inclusion map $U\hookrightarrow V$). Consequently, the overcategory $O(X)_{| U}$ consists of all open subsets of $U$. A sieve on $U$ is simply a collection $S$ of open subsets of $U$ such that if $V'\subset V$ and $V\in S$, then $V'\in S$.
    
    For each $U\in O(X)$, we define a sieve $S$ to be a covering sieve on $U$ if $U$ is contained in (and thus equal to) the union of the open sets in $S$. This assignment forms a Grothendieck topology on $O(X)$.
    \begin{enumerate}
        \item[(1)] Condition (T1) follows directly from the fact that $U$ is equal to the union of all open subsets of $U$, which means that the sieve $O(X)_{| U}$ is indeed a covering sieve on $U$. 
        \item[(2)] To verify condition (T2), let $i:V\subset U$ be a morphism and suppose $S$ is a covering sieve on $U$. This means  $S$ is a sieve on $U$ such that $U$ is equal to the union of the open sets in $S$. 
        
        For each $A\in S$, we consider the following diagram: 
    \begin{eqnarray*}
\xymatrix{ \rule{3mm}{0mm}& V\cap A \ar@{^{(}->}[r]^{} \ar@{^{(}->}[d]_{} &A \ar@{^{(}->}[d]^{} & \\ &
       V  \ar@{^{(}->}[r]_{i} &  U &}
\end{eqnarray*} It is well known that this diagram is a pullback. Since $ V\cap A$ is an open subset of $V$, the collection $S^\ast=\{V\cap A\}_{A\in S}$ forms a sieve on $V$ (utilizing the property of $S$ being a sieve). Moreover, since $V=\bigcup_{A\in S}  V\cap A$, the collection $S^\ast$ is a covering sieve on $V$. Notable, we have $S=O(X)_{| U}^{(0)}\{A\subset U\}_{A\in S}$ and $S^\ast=O(X)_{| V}^{(0)}\{V\cap A\subset V\}_{A\in S}$. By Proposition~\ref{lem;pullback}, we obtain the pullback: \[i^\ast\left(S\right)=S^\ast,\] which conforms that $i^\ast\left(S\right)$ is a covering sieve on $V$. 
     \item[(3)] To see condition (T3), let $S$ be a covering sieve on $U$, and suppose $S'$ is another sieve on $U$ with the property that for every morphism $i:V\to U$ belonging to $S$, the pullback $i^\ast\left(S'\right)$ is a covering sieve on $V$. We aim to show that $S'$ is a covering sieve on $U$, meaning that $U=\bigcup_{A\in S'}A$. 
     
     For each morphism $i:W\subset U$ in $S$, consider the following diagram:
      \begin{eqnarray*}
     \xymatrix{ \rule{3mm}{0mm}& W\cap A \ar@{^{(}->}[r]^{} \ar@{^{(}->}[d]_{} &A \ar@{^{(}->}[d]^{}& \\ &
       W  \ar@{^{(}->}[r]_{i} &  U &}
\end{eqnarray*} This diagram is a pullback for each $A\in S'$. By Proposition~\ref{lem;pullback}, we have: \[(i)^\ast\left(S'\right)=O(X)_{\mid W}^{(0)}\{W\cap A\subset W\}_{A\in S'}.\] Since $S'$ is a sieve on $U$, the collection $\{W\cap A\subset W\}_{A\in S'}$ is also a sieve on $W$.  Therefore, $\{W\cap A\subset W\}_{A\in S'}=O(X)_{\mid W}^{(0)}\{W\cap A\subset W\}_{A\in S'}$ is a covering sieve on $W$, which implies  $W=\bigcup_{A\in S'}W\cap A$. 

Now, we have: \[U=\bigcup_{W\in S}W.\] Next, we can express this as: \[U=\bigcup_{W\in S}\left(W\cap\bigcup_{A\in S'} A\right).\] 
This give us: \[U=\left(\bigcup_{A\in S'} A\right)\cap\left(\bigcup_{W\in S} W\right).\] Since $\bigcup_{W\in S}W=U$, it follows that:\[U=\left(\bigcup_{A\in S'} A\right)\cap U.\] Consequently, we find that:\[U=\bigcup_{A\in S'} A.\]
Thus, we conclude that $S'$ is indeed a covering sieve on $U$.
    \end{enumerate}   
\end{example}

\subsection{Quasi Grothendieck topology} Sometimes, the axiom (T3) in the definition of a Grothendieck topology does not hold. To address this, we introduce the concept of a quasi Grothendieck topology.

\begin{definition}[Quasi Grothendieck topology]\label{def:quasi-topology}
Let $\mathcal{C}$ be a category. A \textit{quasi Grothendieck topology} on  $\mathcal{C}$ is a procedure that assigns to each object $C\in\mathcal{C}$ a collection of sieves on $C$, which are also referred to as \textit{covering sieves}. This assignment must satisfy the following properties: 
\begin{itemize}
    \item[(T1)] For every object $C\in \mathcal{C}$, the sieve $\mathcal{C}_{\mid C}$ is a covering sieve on $C$.
    \item[(T2)] For each morphism $f:D\to C$ in $\mathcal{C}$ and any covering sieve $S\subset\mathcal{C}_{\mid C}$ on $C$, the  pullback $f^\ast S$ is a covering sieve on $D$.
\end{itemize}
\end{definition}

Note that any Grothendieck topology is also a quasi Grothendieck topology. Furthermore, any category equipped with a quasi Grothendieck topology is a category with covers, where the covering sieves (as defined in Definition~\ref{defn:category-with-covers}) correspond to the covering sieves in the context of a quasi Grothendieck topology (see Definition~\ref{def:quasi-topology}). 

\medskip We now introduce the following quasi Grothendieck topology in the category of sets. 

\begin{example}[Category of sets]\label{exam:cover-set}
    Let $\text{Set}$ be the category of sets and functions. In this category, the overcategory $\text{Set}_{| X}$ consists of all functions with codomain $X$. A sieve on $X$ is a family $S$ of functions with codomain $X$ such that if $f:Y\to X$ belongs to $S$ and $g:Z\to Y$ is any function, then $f\circ g\in S$. 
    
    We can define a quasi Grothendieck topology in Set. For each set $X$, a sieve $S$ is considered a covering sieve on $X$ if $S=\text{Set}_{\mid X}^{(0)}\{U_\lambda\hookrightarrow X\}_{\lambda\in \Lambda}$, where each $U_\lambda$ is a  subset of $X$ and $X=\bigcup_{\lambda\in\Lambda} U_\lambda$. This assignment defines a quasi Grothendieck topology on $\text{Set}$ what we refer to as the \textit{subset quasi Grothendieck topology}.
    \begin{enumerate}
        \item[(1)] Condition (T1) follows directly because $X$ is a subset of itself, and we have $\text{Set}_{\mid X}=\text{Set}_{\mid X}^{(0)}\{1_X:X\to X\}$. 
        \item[(2)] To verify condition (T2), let $f:Y\to X$ be a function and assume that $S$ is a covering sieve on $X$. Specifically, we have \[S=\text{Set}_{\mid X}^{(0)}\{U_\lambda\hookrightarrow X\}_{\lambda\in \Lambda},\] where each $U_\lambda$ is a subset of $X$ such that  $X=\bigcup_{\lambda\in\Lambda} U_\lambda$. 
        
        For each $\lambda\in\Lambda$, consider the following diagram: 
    \begin{eqnarray*}
\xymatrix{ \rule{3mm}{0mm}& V_\lambda \ar[r]^{f_|} \ar@{^{(}->}[d]_{} &U_\lambda \ar@{^{(}->}[d]^{} & \\ &
       Y  \ar[r]_{\,\,f} &  X &}
\end{eqnarray*} where $V_\lambda=f^{-1}(U_\lambda)$. This diagram is a pullback. Note that, each $V_\lambda$ is a subset of $Y$. Furthermore, since: \[Y=\bigcup_{\lambda\in\Lambda} V_\lambda,\] it follows that the collection $\text{Set}_{\mid Y}^{(0)}\{V_\lambda\hookrightarrow Y\}_{\lambda\in \Lambda}$ forms a covering sieve on $Y$.

    By Proposition~\ref{lem;pullback}, we have: \[f^\ast S=\text{Set}_{\mid Y}^{(0)}\{V_\lambda\hookrightarrow Y\}_{\lambda\in \Lambda},\] which confirms that $f^\ast S$ is indeed a covering sieve on $Y$.  Thus, condition (T2) holds.
\end{enumerate} 
\end{example} 

We now introduce the following quasi Grothendieck topologies in the category of spaces:

\begin{example}[Category of spaces]\label{exam:open-cover-space}
    Let $\text{Top}$ be the category of spaces and maps. The overcategory $\text{Top}_{| X}$ consists of maps with codomain $X$, and a sieve on $X$ is a family of maps with codomain $X$ such that if $f:Y\to X$ is in the sieve, then $f\circ g$ is in the sieve for any map $g:Z\to Y$. We present several quasi Grothendieck topologies:
    \begin{enumerate}
        \item[(i)] [Open quasi Grothendieck topology] A sieve $S$ is a covering sieve on $X$ if $S=\text{Top}_{\mid X}^{(0)}\{U_\lambda\hookrightarrow X\}_{\lambda\in \Lambda}$, where each $U_\lambda$ is an open subset of $X$ and $X=\bigcup_{\lambda\in\Lambda} U_\lambda$. This assignment defines a quasi Grothendieck topology on $\text{Top}$, which we refer to as the \textit{open quasi Grothendieck topology}.
    \begin{enumerate}
        \item[(1)] Condition (T1) holds directly since $X$ is an open subset of itself, and we have $\text{Top}_{\mid X}=\text{Top}_{\mid X}^{(0)}\{1_X:X\to X\}$. 
        \item[(2)] To verify condition (T2), let $f:Y\to X$ be a map and $S$ a covering sieve on $X$. Specifically, we have $S=\text{Top}_{\mid X}^{(0)}\{U_\lambda\hookrightarrow X\}_{\lambda\in \Lambda}$, where each $U_\lambda$ is an open subset of $X$ such that  $X=\bigcup_{\lambda\in\Lambda} U_\lambda$. 
        
        For each $\lambda\in\Lambda$, the following diagram is a pullback: 
    \begin{eqnarray*}
\xymatrix{ \rule{3mm}{0mm}& V_\lambda \ar[r]^{f_|} \ar@{^{(}->}[d]_{} &U_\lambda \ar@{^{(}->}[d]^{} & \\ &
       Y  \ar[r]_{\,\,f} &  X &}
\end{eqnarray*} Here, $V_\lambda=f^{-1}(U_\lambda)$. Each $V_\lambda$ is an open subset of $Y$, and we have $Y=\bigcup_{\lambda\in\Lambda} V_\lambda$. Thus, the collection $\text{Top}_{\mid Y}^{(0)}\{V_\lambda\hookrightarrow Y\}_{\lambda\in \Lambda}$ forms a covering sieve on $Y$.

    By Proposition~\ref{lem;pullback}, the pullback is given by \[f^\ast S=\text{Top}_{\mid Y}^{(0)}\{V_\lambda\hookrightarrow Y\}_{\lambda\in \Lambda},\] which confirms that $f^\ast S$ is indeed a covering sieve. 
\end{enumerate} 
\item[(ii)] [\'{E}tale quasi Grothendieck topology](cf. \cite[p. 88]{maclane1992}) A map $f:Y\to X$ is called \textit{\'{e}tale} if it is a local homeomorphism. Specifically, for each point $y\in Y$, there exists an open subset $V\subset Y$ containing $y$ such that $f(V)$ is an open subset of $X$ and the restriction $f_{\mid V}:V\to f(V)$ is a homeomorphism. 

For each space $X$, we define a sieve $S$ to be a covering sieve on $X$ if \[S=\text{Top}_{\mid X}^{(0)}\{Y_\lambda\stackrel{f_\lambda}{\to} X\}_{\lambda\in \Lambda},\] where each $f_\lambda:Y_\lambda\to X$ is an \'{e}tale map and $X=\bigcup_{\lambda\in\Lambda} f_\lambda(Y_\lambda)$. This assignment establishes a quasi Grothendieck topology on the category $\text{Top}$, which we refer to as the \textit{\'{e}tale quasi Grothendieck topology}.
    \begin{enumerate}
        \item[(1)] Condition (T1) holds directly because we have $\text{Top}_{\mid X}=\text{Top}_{\mid X}^{(0)}\{1_X:X\to X\}$. 
        \item[(2)] To verify condition (T2), consider a map $f:Y\to X$ and a covering sieve $S$ on $X$. Specifically, we have \[S=\text{Top}_{\mid X}^{(0)}\{Y_\lambda\stackrel{f_\lambda}{\to} X\}_{\lambda\in \Lambda},\] where each $f_\lambda:Y_\lambda\to X$ is an \'{e}tale map, and $X=\bigcup_{\lambda\in\Lambda} f_\lambda(Y_\lambda)$. 
        
        For each $\lambda\in\Lambda$, consider the canonical pullback: 
    \begin{eqnarray*}
\xymatrix{ \rule{3mm}{0mm}& f^\ast Y_\lambda \ar[r]^{\pi_2} \ar[d]_{\pi_1} &Y_\lambda \ar[d]^{f_\lambda} & \\ &
       Y  \ar[r]_{\,\,f} &  X &}
\end{eqnarray*} Here, $f^\ast Y_\lambda=\{(y_1,y_2)\in Y\times Y_\lambda:~f(y_1)=f_\lambda(y_2)\}$ and $\pi_j(y_1,y_2)=y_j$ for $j=1,2$. The  projection $\pi_1:f^\ast Y_\lambda\to Y$ is an \'{e}tale map. Furthermore, since $Y=\bigcup_{\lambda\in\Lambda} \pi_1\left(f^\ast Y_\lambda\right)$, we conclude that $\text{Top}_{\mid Y}^{(0)}\{f^\ast Y_\lambda\stackrel{\pi_1}{\to}Y\}_{\lambda\in \Lambda}$ is a covering sieve on $Y$.

    By Proposition~\ref{lem;pullback}, we then have \[f^\ast S=\text{Top}_{\mid Y}^{(0)}\{f^\ast Y_\lambda\stackrel{\pi_1}{\to}Y\}_{\lambda\in \Lambda},\] which confirms that $f^\ast S$ is indeed a covering sieve on $Y$. 
    \end{enumerate} 
\item[(iii)] [Image quasi Grothendieck topology] For each space $X$, a sieve $S$ is considered a covering sieve on $X$ if it can be expressed in the form \[S=\text{Top}_{\mid X}^{(0)}\{Y_\lambda\stackrel{f_\lambda}{\to} X\}_{\lambda\in \Lambda},\] where each $f_\lambda:Y_\lambda\to X$ is a map and $X=\bigcup_{\lambda\in\Lambda} f_\lambda(Y_\lambda)$. This assignment is a quasi Grothendieck topology on $\text{Top}$, which we refer to as the \textit{image quasi Grothendieck topology}.
    \begin{enumerate}
        \item[(1)] Condition (T1) is satisfied straightforwardly because $X$ is an open subset of itself and we have $\text{Top}_{\mid X}=\text{Top}_{\mid X}^{(0)}\{1_X:X\to X\}$. 
        \item[(2)] To verify condition (T2), suppose that $f:Y\to X$ is a map and $S$ is a covering sieve on $X$. Specifically, $S=\text{Top}_{\mid X}^{(0)}\{Y_\lambda\stackrel{f_\lambda}{\to} X\}_{\lambda\in \Lambda}$, where each $f_\lambda:Y_\lambda\to X$ is a map and $X=\bigcup_{\lambda\in\Lambda} f_\lambda(Y_\lambda)$. For each $\lambda\in\Lambda$, consider the canonical pullback: 
    \begin{eqnarray*}
\xymatrix{ \rule{3mm}{0mm}& f^\ast Y_\lambda \ar[r]^{\pi_2} \ar[d]_{\pi_1} &Y_\lambda \ar[d]^{f_\lambda} & \\ &
       Y  \ar[r]_{\,\,f} &  X &}
\end{eqnarray*} where $f^\ast Y_\lambda=\{(y_1,y_2)\in Y\times Y_\lambda:~f(y_1)=f_\lambda(y_2)\}$ and $\pi_j(y_1,y_2)=y_j$ for $j=1,2$. Note that  $Y=\bigcup_{\lambda\in\Lambda} \pi_1\left(f^\ast Y_\lambda\right)$ and thus $\text{Top}_{\mid Y}^{(0)}\{f^\ast Y_\lambda\stackrel{\pi_1}{\to}Y\}_{\lambda\in \Lambda}$ is a covering sieve on $Y$.

    By Proposition~\ref{lem;pullback}, we conclude that the pullback is given by \[f^\ast S=\text{Top}_{\mid Y}^{(0)}\{f^\ast Y_\lambda\stackrel{\pi_1}{\to}Y\}_{\lambda\in \Lambda},\] which confirms that it is a covering sieve on $Y$. 
    \end{enumerate}    
    \end{enumerate}
    Thus, we have the following inclusions:
\[\text{Open}\subset \text{\'{E}tale}\subset \text{Image}.\]
\end{example}

\subsection{More examples of topologies} Next, we present topologies on the category of groups. 

\begin{example}[Category of groups]\label{exam:category-groups}
    Let $\text{Grp}$ be the category of groups and group homomorphisms. In this category, the overcategory $\text{Grp}_{| G}$ consists of all group homomorphisms with codomain $G$. A sieve on $G$ is defined as a family $S$ of group homomorphisms with codomain $G$ that satisfies the property: if $(H\stackrel{f}{\to} G)\in S$ and $g:K\to H$ is any group homomorphism, then $f\circ g$ is also in $S$. 
   \begin{enumerate}
       \item[(I)] \label{exam:subgroup-cover} [Quasi Grothendieck topology by subgroups] For each group $G$, we define a sieve $S$ to be a covering sieve on $G$ if it can be expressed as $S=\text{Grp}_{\mid G}^{(0)}\{G_\lambda\hookrightarrow G\}_{\lambda\in \Lambda}$, where each $G_\lambda$ is a subgroup of $G$ and $G=\bigcup_{\lambda\in\Lambda} G_\lambda$. This assignment defines a quasi Grothendieck topology on the category $\text{Grp}$.
    \begin{enumerate}
        \item[(1)] Condition (T1) holds trivially since $G$ is a subgroup of itself, and we have $\text{Grp}_{\mid G}=\text{Grp}_{\mid G}^{(0)}\{1_G:G\to G\}$.
        \item[(2)] To verify condition (T2), suppose that $h:H\to G$ is a group homomorphism and let $S$ be a covering sieve on $G$. Specifically, we have $S=\text{Grp}_{\mid G}^{(0)}\{G_\lambda\hookrightarrow G\}_{\lambda\in \Lambda}$, where each $G_\lambda$ is a subgroup of $G$ and $G=\bigcup_{\lambda\in\Lambda} G_\lambda$. 
        
        For each $\lambda\in\Lambda$, consider the diagram: 
    \begin{eqnarray*}
\xymatrix{ \rule{3mm}{0mm}& H_\lambda \ar[r]^{h_|} \ar@{^{(}->}[d]_{} &G_\lambda \ar@{^{(}->}[d]^{} & \\ &
       H  \ar[r]_{\,\,h} &  G &}
\end{eqnarray*} where $H_\lambda=h^{-1}(G_\lambda)$. This diagram is a pullback. Each $H_\lambda$ is a subgroup of $H$, and we can see that  $H=\bigcup_{\lambda\in\Lambda} H_\lambda$. Thus, the collection $\text{Grp}_{\mid H}^{(0)}\{H_\lambda\hookrightarrow H\}_{\lambda\in \Lambda}$ forms a covering sieve on $H$.  

        By Proposition~\ref{lem;pullback}, we conclude that the pullback \[h^\ast S=\text{Grp}_{\mid H}^{(0)}\{H_\lambda\hookrightarrow H\}_{\lambda\in \Lambda}\] is indeed a covering sieve on $H$.  
    \end{enumerate}   
   \item[(II)] \label{exam:proper-subgroup-cover} [Cover with proper subgroups] For each group $G$, we define a sieve $S$ to be a covering sieve on $G$ whenever $S=\text{Grp}_{\mid G}^{(0)}\{G_\lambda\hookrightarrow G\}_{\lambda\in \Lambda}$, where each $G_\lambda$ is a proper subgroup of $G$ and $G=\bigcup_{\lambda\in\Lambda} G_\lambda$. While the category $\text{Grp}$ together with this collection of covering sieves forms a category with covers, this assignment does not constitute a quasi Grothendieck topology on $\text{Grp}$ because Condition (T1) does not hold. 
   
   Additionally, Condition (T2) is also violated. Notably, $\text{Grp}_{\mid G}$ is not a covering sieve for any group $G$. The identity homomorphism $1_G$ cannot be factored as a composition $G\to H\hookrightarrow G$ for some proper subgroup $H\subset G$. 
   \item[(III)] \label{exam:cyclic-subgroup-cover} [Cover with cyclic subgroups] For each group $G$, we define a sieve $S$ to be a covering sieve on $G$ whenever $S=\text{Grp}_{\mid G}^{(0)}\{C_\lambda\hookrightarrow G\}_{\lambda\in \Lambda}$, where each $C_\lambda$ is a cyclic subgroup of $G$ and $G=\bigcup_{\lambda\in\Lambda} C_\lambda$. Similar to the previous case, the category $\text{Grp}$ along with this collection of covering sieves constitutes a category with covers, but this assignment does not qualify as a quasi Grothendieck topology on $\text{Grp}$ due to the failure of Condition (T1). 
   
   Condition (T2) is also not satisfied. Specifically, $\text{Grp}_{\mid G}$ is not a covering sieve for any non-cyclic group $G$ because the identity homomorphism $1_G$ does not factor as a composition $G\to C\hookrightarrow G$ for some cyclic subgroup $C\subset G$. 
   \end{enumerate} 
\end{example}

Now, we present a quasi Grothendieck topology on the category of rings. Here, rings are assumed to be associative, but they need not be commutative or possess a multiplicative identity. Additionally, we define a subring in the weakest sense: $S\subset R$ is considered a subring if $S$ is an additive subgroup that is closed under multiplication. In particular, $S$ need not contain a multiplicative identity, even if $R$ itself is unital. Furthermore, we define a ring homomorphism in the weakest sense: $f:R\to S$ is a ring homomorphism if $f$ preserves addition and multiplication. Again, $f$ need not preserve the multiplicative identity, even if both $R$ and $S$ are unital.   

\begin{example}[Category of rings]\label{exam:subring-cover}
    Let $\text{Ring}$ be the category of rings and ring homomorphisms. The overcategory $\text{Ring}_{| R}$ consists of all ring homomorphisms with codomain $R$. A sieve on $R$ is simply a family $S$ of ring homomorphisms with codomain $R$ such that $f\circ g\in S$ for any $(T\stackrel{f}{\to} R)\in S$ and any ring homomorphism $g:L\to T$. For each ring $R$, we consider a sieve $S$ to be a covering sieve on $R$ whenever $S=\text{Ring}_{\mid R}^{(0)}\{R_\lambda\hookrightarrow R\}_{\lambda\in \Lambda}$, where each $R_\lambda$ is a subring of $R$ and $R=\bigcup_{\lambda\in\Lambda} R_\lambda$. This assignment forms a quasi Grothendieck topology on $\text{Ring}$.
    \begin{enumerate}
        \item[(1)] Condition (T1) follows directly because $R$ is a subring of itself, and $\text{Ring}_{\mid R}=\text{Ring}_{\mid R}^{(0)}\{1_R:R\to R\}$.
        \item[(2)] To see condition (T2), suppose that $h:T\to R$ is a ring homomorphism and $S$ is a covering sieve on $R$, that is, $S=\text{Ring}_{\mid R}^{(0)}\{R_\lambda\hookrightarrow R\}_{\lambda\in \Lambda}$, where each $R_\lambda$ is a subring of $R$ and $R=\bigcup_{\lambda\in\Lambda} R_\lambda$. 
        For each $\lambda\in\Lambda$, it is well known that the following diagram 
    \begin{eqnarray*}
\xymatrix{ \rule{3mm}{0mm}& T_\lambda \ar[r]^{h_|} \ar@{^{(}->}[d]_{} &R_\lambda \ar@{^{(}->}[d]^{} & \\ &
       T  \ar[r]_{\,\,h} &  R &}
\end{eqnarray*} is a pullback, where $T_\lambda=h^{-1}(R_\lambda)$. Note that each $T_\lambda$ is a subring of $T$ and $T=\bigcup_{\lambda\in\Lambda} T_\lambda$; thus, $\text{Ring}_{\mid T}^{(0)}\{T_\lambda\hookrightarrow T\}_{\lambda\in \Lambda}$ is a covering sieve.  
        By Proposition~\ref{lem;pullback}, we have that the pullback \[h^\ast S=\text{Ring}_{\mid T}^{(0)}\{T_\lambda\hookrightarrow T\}_{\lambda\in \Lambda},\] and thus it is a covering sieve on $T$.  
    \end{enumerate}   
\end{example}

Next, we present topologies on the category of $R$-modules. Given a ring $R$, a $R$-module $P$ is called \textit{projective} if and only if for every surjective $R$-homomorphism $f:N\twoheadrightarrow M$ and every $R$-homomorphism $g:P\to M$, there exists an $R$-homomorphism $h:P\to N$ such that $f\circ h=g$.
 \begin{eqnarray*}
\xymatrix@C=3cm{& N \ar@{->>}[d]^{f}   \\  
       P  \ar@{-->}[ru]^{h}\ar[r]_{g} &  M}
\end{eqnarray*}

\begin{example}[Category of $R$-modules]\label{exam:topology-category-r-modules}
  Given a ring $R$, let $R\text{-Mod}$ be the category of left modules over $R$ (or $R$-modules) and $R$-homomorphisms. The overcategory $R\text{-Mod}_{| M}$, for an $R$-module $M$, consists of all $R$-homomorphisms with codomain $M$. A sieve on $M$ is defined as a family $S$ of $R$-homomorphisms with codomain $M$ with the property that $f\circ g\in S$ for any $(N\stackrel{f}{\to} M)\in S$ and any $R$-homomorphism $g:P\to N$.
  \begin{enumerate}
      \item[(I)]\label{exam:submodule-cover}[Quasi Grothendieck topolgy with submodules] For each $R$-module $M$, we consider a sieve $S$ to be a covering sieve on $M$ whenever $S=R\text{-Mod}_{\mid M}^{(0)}\{M_\lambda\hookrightarrow M\}_{\lambda\in \Lambda}$, where each $M_\lambda$ is a submodule of $M$ and $M=\bigcup_{\lambda\in\Lambda} M_\lambda$. This assignment defines a quasi Grothendieck topology on $R\text{-Mod}$.
    \begin{enumerate}
        \item[(1)] Condition (T1) follows directly because $M$ is a submodule of itself, and we have $R\text{-Mod}_{\mid M}=R\text{-Mod}_{\mid M}^{(0)}\{1_M:M\to M\}$.
        \item[(2)] To verify condition (T2), suppose that $h:N\to M$ is an $R$-homomorphism and that $S$ is a covering sieve on $M$. Specifically, $S=R\text{-Mod}_{\mid M}^{(0)}\{M_\lambda\hookrightarrow M\}_{\lambda\in \Lambda}$, where each $M_\lambda$ is a submodule of $M$ and $M=\bigcup_{\lambda\in\Lambda} M_\lambda$. 
        For each $\lambda\in\Lambda$, it is well known that the following diagram 
    \begin{eqnarray*}
\xymatrix{ \rule{3mm}{0mm}& N_\lambda \ar[r]^{h_|} \ar@{^{(}->}[d]_{} &M_\lambda \ar@{^{(}->}[d]^{} & \\ &
       N  \ar[r]_{\,\,h} &  M &}
\end{eqnarray*} is a pullback, where $N_\lambda=h^{-1}(M_\lambda)$. Note that each $N_\lambda$ is a submodule of $N$ and $N=\bigcup_{\lambda\in\Lambda} N_\lambda$; thus, $R\text{-Mod}_{\mid N}^{(0)}\{N_\lambda\hookrightarrow N\}_{\lambda\in \Lambda}$ is a covering sieve on $N$.  
        By Proposition~\ref{lem;pullback}, we have that the pullback \[h^\ast S=R\text{-Mod}_{\mid N}^{(0)}\{N_\lambda\hookrightarrow N\}_{\lambda\in \Lambda},\] and thus it is a covering sieve on $N$.  
    \end{enumerate}
    \item[(II)][Covers with projective submodules]\label{exam:projectivesubmodules} For each $R$-module $M$, we consider a sieve $S$ to be a covering sieve on $M$ whenever $S=R\text{-Mod}_{\mid M}^{(0)}\{P_\lambda\hookrightarrow M\}_{\lambda\in \Lambda}$, where each $P_\lambda$ is a projective submodule of $M$ and $M=\bigcup_{\lambda\in\Lambda} P_\lambda$. The category $R\text{-Mod}$, together with this collection of covering sieves, forms a category with covers. However, this assignment is not a quasi Grothendieck topology on $R\text{-Mod}$ because Condition (T1) does not hold. 
    
    Note that $R\text{-Mod}_{| M}$ is not a covering sieve for any non-projective module $M$ because the identity homomorphism $1_M$ does not factor as a composition $M\to P\hookrightarrow M$ for some projective submodule $P\subset M$.     
  \end{enumerate}
\end{example}

The category of right modules is defined in a similar manner. In the case that $R$ is a field $K$, we have that $K$-Mod is the category of $K$-vector spaces and $K$-linear maps, which we denote  by $K$-Vect. 

Similarly, we can define the following categories with covers: the category of groups with proper cyclic subgroups, the category of rings with proper subrings, the category of $R$-modules with proper submodules, the category of $R$-modules with proper projective submodules, and the category of $K$-vector spaces with proper vector subspaces. 


\medskip Now, we present a quasi Grothendieck topology on the category of graphs.

\begin{example}[Category of graphs]\label{exam:subgraph-cover} 
 Let $\text{Graph}$ be the category of graphs and graph homomorphisms. The overcategory $\text{Graph}_{| G}$ consists of all graph homomorphisms with codomain $G$. A sieve on $G$ is defined as a family $S$ of graph homomorphisms with codomain $G$ such that $f\circ g\in S$ for any $(H\stackrel{f}{\to} G)\in S$ and any graph homomorphism $g:K\to H$. For each graph $G$, we consider a sieve $S$ to be a covering sieve on $G$ whenever $S=\text{Graph}_{\mid G}^{(0)}\{G_\lambda\hookrightarrow G\}_{\lambda\in \Lambda}$, where each $G_\lambda$ is a subgraph of $G$ and $G=\bigcup_{\lambda\in\Lambda} G_\lambda$. This assignment defines a quasi Grothendieck topology on $\text{Graph}$.
    \begin{enumerate}
        \item[(1)] Condition (T1) follows directly because $G$ is a subgraph of itself, and we have $\text{Grp}_{\mid G}=\text{Grp}_{\mid G}^{(0)}\{1_G:G\to G\}$.
        \item[(2)] To verify condition (T2), suppose that $h:H\to G$ is a graph homomorphism and that $S$ is a covering sieve on $G$. Specifically, $S=\text{Grp}_{\mid G}^{(0)}\{G_\lambda\hookrightarrow G\}_{\lambda\in \Lambda}$, where each $G_\lambda$ is a subgraph of $G$ and $G=\bigcup_{\lambda\in\Lambda} G_\lambda$. 
        
        For each $\lambda\in\Lambda$, it is well known that the following diagram 
    \begin{eqnarray*}
\xymatrix{ \rule{3mm}{0mm}& H_\lambda \ar[r]^{h_|} \ar@{^{(}->}[d]_{} &G_\lambda \ar@{^{(}->}[d]^{} & \\ &
       H  \ar[r]_{\,\,h} &  G &}
\end{eqnarray*} is a pullback, where $H_\lambda=h^{-1}(G_\lambda)$. Note that each $H_\lambda$ is a subgraph of $H$ and $H=\bigcup_{\lambda\in\Lambda} H_\lambda$; thus $\text{Graph}_{\mid H}^{(0)}\{H_\lambda\hookrightarrow H\}_{\lambda\in \Lambda}$ is a covering sieve.  
        By Proposition~\ref{lem;pullback}, we have that the pullback \[h^\ast S=\text{Graph}_{\mid H}^{(0)}\{H_\lambda\hookrightarrow H\}_{\lambda\in \Lambda},\] and thus it is a covering sieve on $H$.  
    \end{enumerate}   
\end{example}

Now, we introduce the notion of covering for objects.

\begin{definition}[Covering for objects]\label{def:covering}
Let $\left(\mathcal{C},\mathcal{S}\right)$ be a category with covers and let $C$ be an object in $\mathcal{C}$. We will say that a collection of morphisms $\{f_i:C_i\to C\}_{i\in I}$ is a \textit{covering} for $C$ if it generates a covering sieve on $C$;  that is, $\mathcal{C}_C^{(0)}\{f_i:C_i\to C\}_{i\in I}\in\mathcal{S}$. 
\end{definition}

Hence, we have the following examples.

\begin{example}\label{exam:coverings-discret-trivial}
    Let $\mathcal{C}$ be a category and let $C$ be an object in $\mathcal{C}$.
    \begin{enumerate}
        \item[(1)] \label{exam:discrete-gro} With the discrete Grothendieck topology on $\mathcal{C}$, any collection of morphisms  $\{f_i:C_i\to C\}_{i\in I}$ is a covering for $C$. 
        \item[(2)]\label{exam:identity-covering} With any quasi Grothendieck topology on $\mathcal{C}$, Example~\ref{exam:identity}, together with axiom $(T1)$, implies that $\{1_C:C\to C\}$ is a covering for $C$. 
    \end{enumerate} 
\end{example}

\begin{example}\label{exam:covering-subsets}
    Consider the quasi Grothendieck topology on $\text{Set}$ presented in Example~\ref{exam:cover-set}. Let $X$ be a set and $\{U_\lambda\}_{\lambda\in\Lambda}$ be a collection of subsets of $X$ such that $X=\bigcup_{\lambda\in\Lambda} U_\lambda$. The collection $\{U_\lambda\hookrightarrow X\}_{\lambda\in\Lambda}$ of inclusions is a covering for $X$. In this case, we say that the collection $\{U_\lambda\}_{\lambda\in\Lambda}$ is a \textit{covering by subsets} for the set $X$. 
\end{example}

\begin{example}\label{covering-top-usual-open-covering}
    Consider the quasi Grothendieck topology on $\text{Top}$ presented in Example~\ref{exam:open-cover-space}. Let $X$ be a space and $\{U_\lambda\}_{\lambda\in\Lambda}$ be a collection of open subsets of $X$ such that $X=\bigcup_{\lambda\in\Lambda} U_\lambda$. The collection $\{U_\lambda\hookrightarrow X\}_{\lambda\in\Lambda}$ of inclusions is a covering for $X$. Hence, in this case, the notion of covering in the category setting coincides with the usual notion of open covering in topological setting.  
\end{example}

\textit{\'{E}tale covering} and \textit{image covering} are defined in a similar manner. 

\begin{example}\label{exam:coverings-group-rings-submodules}
\noindent\begin{enumerate}
    \item[(1)] \label{group-category-covers}
Recall that $\text{Grp}$ denotes the category of groups and group homomorphisms.
\begin{enumerate}
    \item[(i)]  Consider the quasi Grothendieck topology on $\text{Grp}$ presented in Example~\ref{exam:subgroup-cover}(I). Let $G$ be a group and $\{G_\lambda\}_{\lambda\in\Lambda}$ be a collection of subgroups of $G$ such that $G=\bigcup_{\lambda\in\Lambda} G_\lambda$. The collection of inclusions $\{G_\lambda\hookrightarrow G\}_{\lambda\in\Lambda}$ is a covering for $G$. In this case, we say that the collection $\{G_\lambda\}_{\lambda\in\Lambda}$ is a \textit{covering} for the group $G$.  
    \item[(ii)]  Consider the covering sieves in $\text{Grp}$ presented in Example~\ref{exam:proper-subgroup-cover}(II). Let $G$ be a group and $\{G_\lambda\}_{\lambda\in\Lambda}$ be a collection of proper subgroups of $G$ such that $G=\bigcup_{\lambda\in\Lambda} G_\lambda$. The collection of inclusions $\{G_\lambda\hookrightarrow G\}_{\lambda\in\Lambda}$ is a covering for $G$. In this case, we say that the collection $\{G_\lambda\}_{\lambda\in\Lambda}$ is a \textit{covering by proper subgroups} for the group $G$. 
    \item[(iii)]  Consider the covering sieves in $\text{Grp}$ presented in Example~\ref{exam:proper-subgroup-cover}(III). Let $G$ be a group and $\{C_\lambda\}_{\lambda\in\Lambda}$ be a collection of cyclic subgroups of $G$ such that $G=\bigcup_{\lambda\in\Lambda} C_\lambda$. The collection of inclusions $\{C_\lambda\hookrightarrow G\}_{\lambda\in\Lambda}$ is a covering for $G$. In this case, we say that the collection $\{C_\lambda\}_{\lambda\in\Lambda}$ is a \textit{covering by cyclic subgroups} for the group $G$. 
\end{enumerate}
\item[(2)] \label{ring-category-covers}  Consider the quasi Grothendieck topology on $\text{Ring}$ presented in Example~\ref{exam:subring-cover}. Let $R$ be a ring and $\{R_\lambda\}_{\lambda\in\Lambda}$ be a collection of subrings of $R$ such that $R=\bigcup_{\lambda\in\Lambda} R_\lambda$. The collection of inclusions $\{R_\lambda\hookrightarrow R\}_{\lambda\in\Lambda}$ is a covering for $R$. In this case, we say that the collection $\{R_\lambda\}_{\lambda\in\Lambda}$ is a \textit{covering} for the ring $R$. 
\item[(3)] \label{module-category-covers}  Recall that $R$-\text{Mod} denotes the category of $R$-modules and $R$-homomorphisms.
\begin{enumerate}
    \item[(i)] Consider the quasi Grothendieck topology on $R$-\text{Mod} presented in Example~\ref{exam:submodule-cover}(I). Let $M$ be a $R$-module and $\{M_\lambda\}_{\lambda\in\Lambda}$ be a collection of submodules of $M$ such that $M=\bigcup_{\lambda\in\Lambda} M_\lambda$. The collection of inclusions $\{M_\lambda\hookrightarrow M\}_{\lambda\in\Lambda}$ is a covering for $M$. In this case, we say that the collection $\{M_\lambda\}_{\lambda\in\Lambda}$ is a \textit{covering} for the $R$-module $M$. 
   \item[(ii)]  Consider the topology on $R$-\text{Mod} presented in Example~\ref{exam:projectivesubmodules}(II). Let $M$ be a $R$-module and $\{P_\lambda\}_{\lambda\in\Lambda}$ be a collection of projective submodules of $M$ such that $M=\bigcup_{\lambda\in\Lambda} P_\lambda$. The collection of inclusions $\{P_\lambda\hookrightarrow M\}_{\lambda\in\Lambda}$ is a covering for $M$. In this case, we say that the collection $\{P_\lambda\}_{\lambda\in\Lambda}$ is a \textit{covering by projective submodules} for the $R$-module $M$. 
\end{enumerate} 
\end{enumerate} 
\end{example}

\textit{Covering by proper subrings}, \textit{covering by proper cyclic subgroups}, \textit{covering by proper submodules}, \textit{covering by proper vector subspaces} and \textit{covering by proper projective submodules} are defined similarly.

\begin{example}\label{covering-graph-covering}
    Consider the quasi Grothendieck topology on $\text{Graph}$ presented in Example~\ref{exam:subgraph-cover}. Let $G$ be a graph and $\{G_\lambda\}_{\lambda\in\Lambda}$ be a collection of subgraphs of $G$ such that $G=\bigcup_{\lambda\in\Lambda} G_\lambda$. The collection of inclusions $\{G_\lambda\hookrightarrow G\}_{\lambda\in\Lambda}$ is a covering for $G$. In this case, we say that the collection $\{G_\lambda\}_{\lambda\in\Lambda}$ is a \textit{covering} for the graph $G$.
\end{example}


\section{Sectional number}\label{sec:sectional-number}
In this section, for an arbitrary morphism $f:X\to Y$ in a category $(\mathcal{C},\mathcal{S})$ with covers, we define a numerical invariant, the sectional number of $f$, in terms of covers of the target $Y$. Indeed, if $\mathcal{R}$ is a relation over $\text{Mor}_{\mathcal{C}}(c,c^\prime)$ for each $c,c^\prime\in \mathcal{C}$, we define the $\mathcal{R}$-sectional number of $f$ (Definition~\ref{defn:sectional-number}). In the case that $\mathcal{R}$ is the \textit{trivial relation}, that is, $f\mathcal{R}f'$ if and only if $f=f'$, we will refer to this as the  \aspas{sectional number} instead of the \aspas{$\mathcal{R}$-sectional number}.

\subsection{Definition and basic results} Let $\mathcal{C}$ be a category, and let $f:X\to Y$ and $g:Z\to Y$ be morphisms in $\mathcal{C}$. Let $\mathcal{R}$ be a relation over $\text{Mor}_{\mathcal{C}}(c,c^\prime)$ for each $c,c^\prime\in \mathcal{C}$. We say that 
a morphism $s:Z\to X$ in $\mathcal{C}$ is a \textit{$\mathcal{R}$-$g$-section} of $f$ if $(f\circ s)\mathcal{R}g$. That is, the following diagram commutes up to the relation $\mathcal{R}$:
\begin{eqnarray*}
\xymatrix{ X \ar[rr]^{f} & &Y  & \\
        &  Z\ar@{-->}[lu]^{s}\ar[ru]_{g} & &} 
\end{eqnarray*}  In the case that $\mathcal{R}$ is the trivial relation, instead of saying \aspas{$\mathcal{R}$-$g$-section}, we will say \aspas{$g$-section}. That is, a morphism $s:Z\to X$ in $\mathcal{C}$ is a \textit{$g$-section} of $f$ if $f\circ s=g$. 

\begin{remark}\label{rem:section-sieve}
Let $\mathcal{C}$ be a category. 
\noindent\begin{enumerate}
    \item[(1)] Note that $f:X\to Y$ admits a $1_Y$-section $Y\to X$ if and only if $\mathcal{C}^{(0)}_{|Y}\{f:X\to Y\}=\mathcal{C}_{|Y}$.
    \item[(2)]  For $g:Z\to Y$, we have that $f:X\to Y$ admits a $g$-section $Z\to X$ if and only if $g$  factors as a composition $Z\to X\stackrel{f}{\to } Y$, which is equivalent to  $\mathcal{C}^{(0)}_{|Y}\{g:Z\to Y\}\subset \mathcal{C}^{(0)}_{|Y}\{f:X\to Y\}$.
\end{enumerate}  
\end{remark}

A morphism $f:Y\to X$ in a category $\mathcal{C}$ is called a \textit{monomorphism}, denoted by $f:Y\rightarrowtail X$, if for any object $Z\in\mathcal{C}$ and any morphisms $g_1,g_2:Z\to Y$, the condition $f\circ g_1=f\circ g_2$ implies that $g_1=g_2$. For example, the identity morphism $1_X:X\to X$ is a monomorphism.

\medskip Now, we introduce the notion of sectional number. 

 \begin{definition}[Sectional number]\label{defn:sectional-number}
Let $(\mathcal{C},\mathcal{S})$ be a category with covers, and suppose that $\mathcal{R}$ is a relation over $\text{Mor}_{\mathcal{C}}(c,c^\prime)$ for each $c,c^\prime\in \mathcal{C}$. For a morphism $f:X\to Y$ in $\mathcal{C}$:
\begin{enumerate}
    \item[(1)] The \textit{$\mathcal{R}$-sectional number} of $f$, denoted by $\mathcal{R}$-sec$(f)$, is the least integer $m$ such that there exists a covering $\{g_j:Y_j\to Y\}_{j=1}^m$ for $Y$ such that for each $j=1,\ldots,m$, $f$ admits a $\mathcal{R}$-$g_j$-section. Such a covering $\{g_j:Y_j\to Y\}_{j=1}^m$ is called \textit{$\mathcal{R}$-categorical}. If no such $m$ exists, we set $\mathcal{R}$-sec$(f)=\infty$. 
    \item[(2)] The \textit{$\mathcal{R}$-mono sectional number} of $f$, denoted by $\mathcal{R}$-msec$(f)$, is the least integer $m$ such that there exists a covering $\{g_j:Y_j\to Y\}_{j=1}^m$ for $Y$ with the property that for each $j=1,\ldots,m$, $g_j$ is a monomorphism and $f$ admits a $\mathcal{R}$-$g_j$-section. Such a covering $\{g_j:Y_j\rightarrowtail Y\}_{j=1}^m$ is called \textit{$\mathcal{R}$-mono categorical}. If no such $m$ exists, we set $\mathcal{R}$-msec$(f)=\infty$.
\end{enumerate}
\begin{eqnarray*}
\xymatrix{ X \ar[rr]^{f} & &Y  & \\
        &  Y_j\ar@{-->}[lu]^{}\ar[ru]_{g_j} & &} 
\end{eqnarray*}
\end{definition}

\medskip In the case that $\mathcal{R}$ is the trivial relation,  instead of saying \aspas{$\mathcal{R}$-sectional number}, we will say \aspas{sectional number} and write $\mathrm{sec}(f)$ instead of $\mathcal{R}$-sec$(f)$. Similarly, instead of saying  \aspas{$\mathcal{R}$-categorical}, we will say \aspas{categorical}. Likewise, instead of saying \aspas{$\mathcal{R}$-mono sectional number}, we will say \aspas{mono sectional number} and write $\mathrm{msec}(f)$ instead of  $\mathcal{R}$-msec$(f)$. Finally, instead of saying \aspas{$\mathcal{R}$-mono categorical}, we will say \aspas{mono categorical}. 

\medskip To emphasise the collection of covering sieves $\mathcal{S}$, we write $\mathcal{S}$-$\mathcal{R}$-sec$(f)$, $\mathcal{S}$-$\mathcal{R}$-msec$(f)$, $\mathcal{S}$-sec$(f)$, and $\mathcal{S}$-msec$(f)$ for the $\mathcal{R}$-sectional number, $\mathcal{R}$-mono sectional number, sectional number, and mono sectional number, respectively.     

\begin{remark}\label{rem:ine-top} 
Let $\mathcal{C}$ be a category with covers, and let $\mathcal{R}$ be a relation over $\text{Mor}_{\mathcal{C}}(c,c^\prime)$ for each $c,c^\prime\in \mathcal{C}$.
\begin{enumerate}
    \item[(1)] Note that $\text{$\mathcal{R}$-sec}(f)\leq\min\{\text{$\mathcal{R}$-msec}(f),\text{sec}(f)\}\leq\text{msec}(f)$, for any reflexive relation $\mathcal{R}$. The inequalities $\text{$\mathcal{R}$-sec}(f)\leq\text{sec}(f)$ and $\text{$\mathcal{R}$-msec}(f)\leq\text{msec}(f)$ follow because any section is also a $\mathcal{R}$-section (here, we use the fact that $\mathcal{R}$ is reflexive).
    \item[(2)]\label{rem:two-topologies} Suppose that $\mathcal{S}$ and $\mathcal{T}$ are two topologies on $\mathcal{C}$. If $\mathcal{T}\subset\mathcal{S}$ (that is, any $\mathcal{T}$-covering sieve is also a $\mathcal{S}$-covering sieve), then \[\text{$\mathcal{S}$-$\mathcal{R}$-sec$(f)$}\leq \text{$\mathcal{T}$-$\mathcal{R}$-sec$(f)$}.\]  Similarly, one has \[\text{$\mathcal{S}$-$\mathcal{R}$-msec$(f)\leq \mathcal{T}$-$\mathcal{R}$-msec$(f)$}.\]
\end{enumerate}
\end{remark}

Remark~\ref{rem:ine-top}(1) generalizes property (1) mentioned in the introduction. 

\medskip We have the following example. 

\begin{example}[$\mathcal{S}\text{-}\mathcal{R}\text{-sec}(f:X\to I)=1$]\label{exam:sec-codomain-initial}
 Let $\mathcal{C}$ be a category with an initial object $I$. Let $\mathcal{S}$ be a topology on $\mathcal{C}$ satisfying Axiom (T1). Given any morphism $f:X\to I$ in $\mathcal{C}$ whose codomain is the initial object $I$, we have $\mathcal{S}\text{-}\mathcal{R}\text{-msec}(f)=1$ (and, of course, $\mathcal{S}\text{-}\mathcal{R}\text{-sec}(f)=1$) for any reflexive relation $\mathcal{R}$ over $\text{Mor}_{\mathcal{C}}(c,c^\prime)$ for each $c,c^\prime\in \mathcal{C}$. This follows because we have the following commutative diagram:
 \begin{eqnarray*}
\xymatrix{ X \ar[rr]^{f} & &I   \\
        &  I\ar@{-->}[lu]^{}\ar[ru]_{1_I} & &} 
\end{eqnarray*}
\end{example}

Let $\mathcal{C}$ be any category. We denote the discrete Grothendieck topology on $\mathcal{C}$ as \aspas{Discrete}. For any topology $\mathcal{S}$ on $\mathcal{C}$, we have \[\mathcal{S}\subset \mathrm{Discrete},\] and thus, by Remark~\ref{rem:ine-top}(2), the inequality  \[\text{Discrete-$\mathcal{R}$-sec$(f)$}\leq \text{$\mathcal{S}$-$\mathcal{R}$-sec$(f)$}\] holds for any morphism $f$ in $\mathcal{C}$ and any relation $\mathcal{R}$. Similarly, the inequality \[\text{Discrete-$\mathcal{R}$-msec$(f)$}\leq \text{$\mathcal{S}$-$\mathcal{R}$-msec$(f)$}\] also holds. 

\medskip The following example demonstrates that the sectional number with the discrete Grothendieck topology is typically 1. 

\begin{example}[Sectional number with the discrete topology]\label{exam:sec-discrete-topology}
 Let $\mathcal{C}$ be a category with the discrete Grothendieck topology. We have $\text{$\mathcal{R}$-sec}(f)=1$ for any morphism $f:X\to Y$ in $\mathcal{C}$ and any reflexive relation $\mathcal{R}$. This follows because $\{f:X\to Y\}$ is a covering for $Y$ and the identity morphism $1_X:X\to X$ serves as a $\mathcal{R}$-$f$-section of $f$ (here we use the fact that $R$ is reflexive). 
 
 Likewise, note that $\text{$\mathcal{R}$-msec}(h)=1$ for any monomorphism $h:X\to Y$ in $\mathcal{C}$ and any reflexive relation $\mathcal{R}$.  
 \begin{eqnarray*}
\xymatrix{ X \ar[rr]^{f} & &Y  & \\
        &  X\ar@{-->}[lu]^{1_X}\ar[ru]_{f} & &} & \xymatrix{ X \ar@{>->}[rr]^{h} & &Y  & \\
        &  X\ar@{-->}[lu]^{1_X}\ar@{>->}[ru]_{h} & &}
\end{eqnarray*}
\end{example}

Before presenting the following result, we introduce a new notion. We will say that a covering $\{f_j:D_j\to C\}_{j\in J}$ for $C$ is \textit{through monomorphism} if there exists a collection of monomorphisms $\{g_i:C_i\rightarrowtail C\}_{i\in I}$ such that \[\mathcal{C}_C^{(0)}\{f_j:D_j\to C\}_{j\in J}=\mathcal{C}_C^{(0)}\{g_i:C_i\rightarrowtail C\}_{i\in I}.\] In particular, the  collection of monomorphisms $\{g_i:C_i\rightarrowtail C\}_{i\in I}$ forms a covering for $C$. 

The following statement asserts that when the covers are through monomorphisms, the sectional number coincides with the mono sectional number. Let $\mathcal{R}$ be a relation over $\text{Mor}_{\mathcal{C}}(c,c^\prime)$ for each $c,c^\prime\in \mathcal{C}$. We say that $\mathcal{R}$ \textit{preserves composition} if $(f\circ f')\mathcal{R}(h\circ h')$ whenever $f\mathcal{R}h$ in $\text{Mor}_{\mathcal{C}}(c,c^\prime)$ and $f'\mathcal{R}h'$ in $\text{Mor}_{\mathcal{C}}(c'',c)$.

\begin{proposition}[Sectional equal mono sectional]\label{prop:mono-sec-coincides}
   Let $\mathcal{C}$ be a category with covers, $f:X\to Y$ be a morphism, and $\mathcal{R}$ be an equivalence relation over $\text{Mor}_{\mathcal{C}}(c,c^\prime)$ for each $c,c^\prime\in \mathcal{C}$ that preserves composition. Suppose that each covering for an object in $\mathcal{C}$ is through monomorphisms. Then, we have \[\text{$\mathcal{R}$-sec}(f)=\text{$\mathcal{R}$-msec}(f).\]
\end{proposition}
\begin{proof}
Suppose that $\text{$\mathcal{R}$-sec}(f)=m$ and consider a covering $\{g_j:Y_j\to Y\}_{j=1}^m$ for $Y$ such that for each $j=1,\ldots,m$, $f$ admits a $\mathcal{R}$-$g_j$-section. By hypothesis, the covering $\{g_j:Y_j\to Y\}_{j=1}^m$ for $Y$ is through monomorphisms, meaning  there exists a collection of monomorphisms $\{f_i:U_i\rightarrowtail Y\}_{i\in I}$ such that \[\mathcal{C}_Y^{(0)}\{g_j:Y_j\to Y\}_{j=1}^m=\mathcal{C}_Y^{(0)}\{f_i:U_i\rightarrowtail Y\}_{i\in I}.\] 

By Proposition~\ref{prop:sub-collection}, there exists a subcollection $\{f_{k_j}:U_{k_j}\rightarrowtail Y\}_{j=1}^{m}$ of $\{f_i:U_i\rightarrowtail Y\}_{i\in I}$ (where $\{k_1,\ldots,k_m\}\subset I$) such that for each $i\in I$ there exists $\alpha_i\in \{k_1,\ldots,k_m\}$ with $f_i$ factors as a composition \[U_i\to U_{\alpha_i}\stackrel{f_{\alpha_i}}{\rightarrowtail } Y.\] 

Furthermore, we have \[\mathcal{C}_{\mid C}^{(0)}\{f_i:U_i\rightarrowtail Y\}_{i\in I}=\mathcal{C}_{\mid C}^{(0)}\{f_{k_j}:U_{k_j}\rightarrowtail Y\}_{j=1}^{m},\] and thus \[\mathcal{C}_{\mid C}^{(0)}\{g_j:Y_j\to Y\}_{j=1}^m=\mathcal{C}_{\mid C}^{(0)}\{f_{k_j}:U_{k_j}\rightarrowtail Y\}_{j=1}^{m}.\] 

For each $\ell=1,\ldots,m$, $f_{k_\ell}$ factors as a composition \[U_{k_\ell}\stackrel{\psi_\ell}{\to} Y_{j_\ell}\stackrel{g_{j_\ell}}{\to } Y\] for some $j_\ell\in\{1,\ldots,m\}$. Moreover, for each $j_\ell$, there exists a morphism $s_{j_\ell}:Y_{j_\ell}\to X$ such that \[(f\circ s_{j_\ell})\mathcal{R}g_{j_\ell}.\] 

Thus, we have \[\left((f\circ s_{j_\ell})\circ\psi_\ell\right)\mathcal{R}\left(g_{j_\ell}\circ\psi_\ell\right),\] using the fact that $\mathcal{R}$ preserves composition. Consequently, we obtain \[\left(f\circ \left(s_{j_\ell}\circ\psi_\ell\right)\right)\mathcal{R}f_{k_\ell}\] because $\mathcal{R}$ is an equivalence relation. Therefore, we have constructed a covering $\{f_{k_j}:U_{k_j}\rightarrowtail Y\}_{j=1}^{m}$ for $Y$ such that for each $j=1,\ldots,m$, $f_{k_j}$ is a monomorphism and $f$ admits a $\mathcal{R}$-$f_{k_j}$-section. 

This leads us to conclude that \[\text{$\mathcal{R}$-msec}(f)\leq m=\text{$\mathcal{R}$-sec}(f).\]
\end{proof}

We will apply Proposition~\ref{prop:mono-sec-coincides} in the following examples to illustrate its implications.

\begin{example}[Sectional number in (Top, Open)]\label{exam:sectional-number-Top} Consider the open quasi Grothendieck topology on $\text{Top}$ as presented in Example~\ref{exam:open-cover-space}. Recall from Example~\ref{covering-top-usual-open-covering} that,  given a space $W$, a collection of inclusions $\{U_\lambda\hookrightarrow W\}_{\lambda\in\Lambda}$ constitutes a covering for $W$ whenever $\{U_\lambda\}_{\lambda\in\Lambda}$ consists of open subsets of $W$ such that   $W=\bigcup_{\lambda\in\Lambda} U_\lambda$. 

Let $f:X\to Y$ be a map and $\mathcal{R}$ be an equivalence relation over $\text{Mor}_{\text{Top}}(W,W^\prime)$ for each pair of spaces $W$ and $W^\prime$, which preserves composition. By Proposition~\ref{prop:mono-sec-coincides}, the $\mathcal{R}$-sectional number of $f$ coincides with the $\mathcal{R}$-mono sectional number of $f$. 

Moreover, the following statements hold: 
\begin{enumerate}
        \item[(1)] Our sectional number of $f$ coincides with the usual sectional number of $f$, denoted $\text{sec}_u(f)$. 
        \item[(2)] Let $\mathcal{R}$ be the homotopy relation $\simeq$ between maps. In this case, the $\mathcal{R}$-sectional number of $f$ coincides with the sectional category of $f$, denoted $\text{secat}(f)$. 
    \end{enumerate} From the above, we see that our sectional number theory recovers the usual sectional theory. Thus, we can conclude that this sectional number serves as a generalization--in a categorical setting (that is, an unification or categorification)--of the usual sectional number and the sectional category.  
\end{example}

The following remark highlights that our sectional number provides a family of new numerical invariants. 

\begin{remark}\label{rem:sec-induce-functor} 
    The $\mathcal{R}$-sectional number of $f$ provides new numerical invariants. For instance, consider the functor $\mathcal{F}:\text{Top}\to\text{Grp}$ (which can be covariant or contravariant). We define the equivalence relation $\mathcal{R}_{\mathcal{F}}$ by setting  $f\mathcal{R}_{\mathcal{F}} g$ if and only if $\mathcal{F}(f)=\mathcal{F}(g)$ for any maps $f,g:X\to Y$. Note that $\mathcal{R}_{\mathcal{F}}$ preserves composition. Therefore, we can define the sectional number $\mathcal{R}_{\mathcal{F}}\text{-sec}(-)$ in $\text{Top}$ with respect to any topology.  
\end{remark}

Computing the sectional number presents a significant challenge. For instance, consider the following example.

\begin{example}[Sectional number in Set]\label{exam:sectional-number-Set} 
We will work in the category Set of sets, equipped with the quasi Grothendieck topology presented in Example~\ref{exam:cover-set}. According to this topology, a collection $\{U_\lambda\hookrightarrow X\}_{\lambda\in\Lambda}$ of inclusions is a covering for a set $X$ if and only if $X=\bigcup_{\lambda\in\Lambda} U_\lambda$ (see Example~\ref{exam:covering-subsets}). For any function $f$, we have that $\mathrm{sec}(f)<\infty$ if and only if $f$ is surjective if and only if $\mathrm{sec}(f)=1$. For instance, the following statements hold:
\begin{enumerate}
\item[(1)] Given a set $X\subset\mathbb{N}$, consider the set $F_{<}(X)=\{(x,y)\in X\times X:~x<y\}$ and the projection $\pi_1^X:F_{<}(X)\to X$ defined by $\pi_1(x,y)=x$.

Note that $\mathrm{sec}(\pi_1^X)<\infty$ if and only if $\pi_1^X$ is surjective, which is equivalent to $\mathrm{sec}(\pi_1^X)=1$, and this occurs if and only if $X$ is infinite. For example, we have  $\mathrm{sec}(\pi_1^P)=1$ in the case that $P$ is the set of prime numbers. 

Furthermore, let $P^{\text{twin}}=\{p:~\text{$p$ is prime and $p+2$ is prime}\}$. Then, $\mathrm{sec}(\pi_1^{P^{\text{twin}}})=1$ if and only if the twin prime conjecture holds, which posits that there are infinitely many twin primes. 
    \item[(2)] Let $P_{> 2}=\{p:~ \text{$p$ is a prime greater than $2$}\}$ and \[\mathbb{N}_{> 2}^{\text{even}}=\{n:~\text{$n$ is an even natural number greater than $2$}\}.\] Consider the map $\varphi:P_{> 2}\times P_{> 2}\to \mathbb{N}_{> 4}^{\text{even}}$ defined by $\varphi(p,q)=p+p$. We have that $\mathrm{sec}(\varphi)<\infty$ if and only if $\varphi$ is surjective, which is equivalent to $\mathrm{sec}(\varphi)=1$. This condition holds if and only if the Goldbach's conjecture is true. Goldbach's conjecture posits that every even natural number greater than 2 can be expressed as the sum of two prime numbers.  
\end{enumerate}
\end{example}

Recall that in $\text{Top}$, we have the following quasi Grothendieck topologies: the open topology (denoted as \textit{Open}), the \'{e}tale  topology (denoted as \text{\'{E}tale}), and the image 
 topology (denotes as \text{Image}). In the next example, we will  compare the sectional numbers with respect to these different  topologies. 

\begin{example}[Sectional number in Top]\label{exam:sectional-number-top} 
  In the category $\text{Top}$, let $f:X\to Y$ be a map, and let $\mathcal{R}$ be a relation over $\text{Mor}_{\text{Top}}(W,W^\prime)$ for each pair of spaces $W$ and $W^\prime$. Recall that we have the inclusions $\text{Open}\subset \text{\'{E}tale}\subset \text{Image}$. By Remark~\ref{rem:two-topologies}, it follows that: \[\text{Image-$\mathcal{R}$-sec$(f)$}\leq \text{\'{E}tale-$\mathcal{R}$-sec$(f)$}\leq \text{Open-$\mathcal{R}$-sec$(f)$}.\] In particular, we obtain the inequalities: \[\text{Image-sec$(f)$}\leq \text{\'{E}tale-sec$(f)$}\leq \text{sec$_{\text{u}}(f)$}\] and \[\text{Image-$\simeq$-sec$(f)$}\leq \text{\'{E}tale-$\simeq$-sec$(f)$}\leq \text{secat$(f)$}.\] 
  
  Additionally, note that:
  \begin{enumerate}
      \item[(1)] $\text{Image-sec$(f)$}=1$ for any surjective map $f$. Furthermore, if $\text{Image-sec$(f)$}<\infty$, then $f$ is surjective, and thus $\text{Image-sec$(f)$}=1$.  
      \item[(2)] $\text{\'{E}tale-sec$(f)$}=1$ for any locally homeomorphism $f$. Furthermore, if $\text{\'{E}tale-sec$(f)$}<\infty$ then $f$ is surjective. 
      \item[(3)] By the Homotopy Lifting Property, observe that $\text{secat$(f)$}=\text{sec$_{\text{u}}(f)$}$ whenever $f$ is a fibration. Likewise, $\text{\'{E}tale-$\simeq$-sec$(f)$}=\text{\'{E}tale-sec$(f)$}$ and $ \text{Image-$\simeq$-sec$(f)$}=\text{Image-sec$(f)$}$ whenever $f$ is a fibration.   \end{enumerate} In the case that $f$ is the quotient map $q:X\to X/G$, with $G$ acting on a space $X$, we have that $\text{Image-sec$(q)$}=1$ (because $q$ is surjective). Moreover, if $X$ is Hausdorff and $G$ is finite, acting freely on $X$, then $\text{\'{E}tale-sec$(q)$}=1$ (since $q$ is a covering map and thus a locally homeomorphism). In addition,  sec$_{\text{u}}(f)$ is related to the Borsuk-Ulam property (see \cite{zapata-daciberg2023}, \cite{zapata-daciberg-arxiv}, \cite{zapata-daciberg-param}).
\end{example}

Example~\ref{exam:sectional-number-top} motivates the following example.

\begin{example}[Usual and \'{e}tale topological complexity]\label{exam:topological-complexity} 
 Let $X$ be a path-connected space. Consider the fibration $e_2^X:X^{[0,1]}\to X\times X$ given by $e_2^X(\gamma)=(\gamma(0),\gamma(1))$. From \cite{farber2003}, the \textit{topological complexity} of $X$ is defined as \[\text{TC}(X)=\text{sec$_{\text{u}}(e_2^X)$}.\]  We introduce the \textit{\'{e}tale topological complexity} of $X$ given by \[\text{TC}_{\text{\'{e}tale}}(X)=\text{\'{E}tale-sec$(e_2^X)$}.\] Note that $\text{Image-sec$(e_2^X)$}=1$ because $e_2^X$ is surjective (this follows from the path-connectedness of $X$). 
 
 On the other hand, from Example~\ref{exam:sectional-number-top}, we have:
 \[\text{TC}_{\text{\'{e}tale}}(X)\leq \text{TC}(X).\] In particular, $\text{TC}_{\text{\'{e}tale}}(X)= \text{TC}(X)=1$ whenever $X$ is contractible. 
 
 Moreover, we note that $\text{TC}(S^1)=2$ and $\text{TC}_{\text{\'{e}tale}}(S^1)=1$ (here we use the universal covering $\mathbb{R}\to S^1, r\mapsto e^{2\pi i r}$, together with the map $\sigma:\mathbb{R}\times\mathbb{R}\to \left(S^1\right)^{[0,1]}$ given by $\sigma(r,r')(t)=e^{2\pi i ((1-t)r+tr')}$). 
\end{example}

Now, in the next example, we will explore the sectional number in the category of groups. In particular, this new numerical invariant recovers the notion of covering number of groups, as discussed in \cite[p. 491]{haber1959} and \cite[p. 44]{cohn1994}. 

\begin{example}[Sectional number in the category of groups]\label{sec-homomorphism-groups}
Recall that $\text{Grp}$ denotes the category of groups and group homomorphisms.  Let $h:H\to G$ be a group homomorphism, and $\mathcal{R}$ be an equivalence relation over $\text{Mor}_{\text{Grp}}(G,G^\prime)$  for each pair of groups $G,G^\prime$, preserving composition.
\begin{enumerate}
    \item[(i)]  Consider the quasi Grothendieck topology on $\text{Grp}$ presented in Example~\ref{exam:subgroup-cover}(I). From Example~\ref{group-category-covers}(1)(i), a collection of inclusions $\{G_\lambda\hookrightarrow G\}_{\lambda\in\Lambda}$ is a covering for a group $G$ whenever $\{G_\lambda\}_{\lambda\in\Lambda}$ consists of subgroups of $G$ such that $G=\bigcup_{\lambda\in\Lambda} G_\lambda$. 
    
    By Proposition~\ref{prop:mono-sec-coincides}, the $\mathcal{R}$-sectional number of $h$ coincides with the $\mathcal{R}$-mono sectional number of $h$. Thus, we can define a notion of sectional number of a group homomorphism, denoted as $\mathrm{sec}(h)$. Similarly, we can define a sectional number for a group homomorphism based on proper subgroups, denoted as  $\mathrm{sec}'(h)$. 

    \item[(ii)]  Consider the topology on $\text{Grp}$ presented in Example~\ref{exam:subgroup-cover}(III). From Example~\ref{group-category-covers}(1)(iii), a collection of inclusions $\{C_\lambda\hookrightarrow G\}_{\lambda\in\Lambda}$ forms a covering for a group $G$ whenever $\{C_\lambda\}_{\lambda\in\Lambda}$ consists of cyclic subgroups of $G$ such that $G=\bigcup_{\lambda\in\Lambda} C_\lambda$. 
    
     By Proposition~\ref{prop:mono-sec-coincides}, the $\mathcal{R}$-sectional number of $h$ coincides with the $\mathcal{R}$-mono sectional number of $h$. Thus, we can define a notion of sectional number of a group homomorphism based on  cyclic subgroups, denoted as $\mathrm{sec}_{\text{cyclic}}(h)$. Similarly, we can define a sectional number for a group homomorphism based on proper cyclic subgroups, denoted as  $\mathrm{sec}'_{\text{cyclic}}(h)$.
\end{enumerate}
Moreover, the following statements hold:
\begin{enumerate}
        \item[(1)] It follows that if $\mathrm{sec}(h)<\infty$, then $h$ is surjective. Moreover, $\mathrm{sec}(h)=1$ if and only if there exists a group homomorphism  $s:G\to H$ such that $h\circ s=1_G$. In contrast, the inequality $\mathrm{sec}'(h)\geq 3$ always holds (see the next item,  Item (3)). 
        \item[(2)] The inequality $\mathrm{sec}(h)\leq \min\{\mathrm{sec}'(h),\mathrm{sec}_{\text{cyclic}}(h)\}\leq \mathrm{sec}'_{\text{cyclic}}(h)$ always holds. Furthermore, if $\mathrm{sec}(h)\geq 2$, then $\mathrm{sec}(h)=\mathrm{sec}'(h)$ and $\mathrm{sec}_{\text{cyclic}}(h)=\mathrm{sec}'_{\text{cyclic}}(h)$. 
        \item[(3)] From~\cite[p. 491]{haber1959} and \cite[p. 44]{cohn1994}, the minimal number of proper subgroups needed to cover $G$ is called the \textit{covering number} of $G$, denoted by $\sigma(G)$. It is easy to see that $\sigma(G)\geq 3$ for any group $G$. Furthermore, if $h:H\to G$ is surjective, then $\sigma(H)\leq \sigma(G)$. By the definition of sec$'$, we have $\sigma(G)\leq\mathrm{sec}'(h)$. Additionally, if $\mathrm{sec}(h)=1$, then $\mathrm{sec}'(h)=\sigma(G)$. In the case that $h:H\to G$ is an isomorphism, we find that $\mathrm{sec}'(h)=\sigma(G)=\sigma(H)$. Thus, our $\mathrm{sec}'(-)$ serves as a generalization of the covering number. 
        \item[(4)] Motivated by Item (3) above, the minimal number of proper cyclic subgroups needed to cover $G$ is called the \textit{cyclic covering number} of $G$, denoted by $\sigma_{\text{cyclic}}(G)$. It is important to note that $\sigma(G)\leq \sigma(G)_{\text{cyclic}}$ for any group $G$. Furthermore, by definition of $\mathrm{sec}'_{\text{cyclic}}$, we have $\mathrm{sec}'_{\text{cyclic}}(h)=\sigma_{\text{cyclic}}(G)$ for any surjective homomorphism $h:H\to G$. 
    \end{enumerate}
 For instance, consider any homomorphism $h:H\to \mathbb{Q}$, where the codomain is the group of rational numbers $(\mathbb{Q},+)$. In this case, we have $\mathrm{sec}'(h)=\infty$. To see why, note that $\sigma(\mathbb{Q})=\infty$. If we suppose that $\mathbb{Q}=Q_1\cup\cdots\cup Q_n$ for some proper subgroups $Q_i$ of $(\mathbb{Q},+)$, then for each $i=1,\ldots,n$, there exists a rational number $q_i=\dfrac{a_i}{b_i}$ such that $q_i\not\in Q_i$. Since each $Q_i$ is a $\mathbb{Z}$-module, we have $\dfrac{1}{b_i}\not\in Q_i$ for all $i=1,\ldots,n$. Consequently, $\dfrac{1}{b_1\cdots b_n}\not\in Q_1\cup\cdots\cup Q_n$ (here we use that $\dfrac{1}{b_i}=b_1\cdots b_{i-1}b_{i+1}\cdots b_n\cdot\dfrac{1}{b_1\cdots b_n}$) which leads to a contradiction.    
\end{example}

Now, in the next example, we will explore the sectional number in the category of rings. This new numerical invariant, in particular,  recovers the notion of covering number of rings. The covering number of rings was studied for several authors (see \cite{swartz-werner2024} and the references of therein). 

\begin{example}[Sectional number in the category of Rings]\label{sec-homomorphism-rings} 
Recall that $\text{Ring}$ denotes the category of rings and ring homomorphisms. Let $h:R\to S$ be a ring  homomorphism, and let $\mathcal{R}$ be an equivalence relation over $\text{Mor}_{\text{Ring}}(R,S)$ for each pair of rings $R,S$, preserving composition.
\begin{enumerate}
    \item[(i)]  Consider the quasi Grothendieck topology on $\text{Ring}$ presented in Example~\ref{exam:subring-cover}. From Example~\ref{ring-category-covers}(2), given a ring $R$, a collection of inclusions $\{R_\lambda\hookrightarrow R\}_{\lambda\in\Lambda}$ is a covering for $R$ whenever $\{R_\lambda\}_{\lambda\in\Lambda}$ is a collection of subrings of $R$ such that $R=\bigcup_{\lambda\in\Lambda} R_\lambda$. 
    
    By Proposition~\ref{prop:mono-sec-coincides}, the $\mathcal{R}$-sectional number of $h$ coincides with the $\mathcal{R}$-mono sectional number of $h$. Thus, we obtain a notion of sectional number for a ring homomorphism, denoted $\mathrm{sec}(h)$. 
    \item[(ii)]  Consider the covering sieves in $\text{Ring}$ formed by proper subrings. Specifically, given a ring $R$, a collection of inclusions $\{R_\lambda\hookrightarrow R\}_{\lambda\in\Lambda}$ is called a \textit{covering by proper subrings} for $R$ whenever $\{R_\lambda\}_{\lambda\in\Lambda}$ consists of proper subrings of $R$ such that  $R=\bigcup_{\lambda\in\Lambda} R_\lambda$.  
    
    By Proposition~\ref{prop:mono-sec-coincides}, the $\mathcal{R}$-sectional number of $h$ coincides with the $\mathcal{R}$-mono sectional number of $h$. Thus, we obtain a new version of the sectional number for a ring homomorphism, denoted  $\mathrm{sec}'(h)$. 
\end{enumerate}
Moreover, the following statements are true:
\begin{enumerate}
        \item[(1)] If $\mathrm{sec}(h)<\infty$, then $h$ is surjective. Moreover, $\mathrm{sec}(h)=1$ if and only if there exists a ring homomorphism  $s:S\to R$ such that $h\circ s=1_S$. In contrast, the inequality $\mathrm{sec}'(h)\geq 3$ always holds (see the next Item (3)). 
        \item[(2)] The inequality $\mathrm{sec}(h)\leq \mathrm{sec}'(h)$ always holds. Furthermore, if $\mathrm{sec}(h)\geq 2$, then $\mathrm{sec}(h)=\mathrm{sec}'(h)$. 
        \item[(3)] From~\cite{swartz-werner2024}, the minimal number of proper subrings needed to cover $R$ is called the \textit{covering number} of $R$, denoted by $\sigma(R)$. It is easy to see that $\sigma(R)\geq\sigma(R,+)$, and thus $\sigma(R)\geq 3$ for any ring $R$. Furthermore, $\sigma(R)\leq \sigma(S)$ whenever $h$ is surjective. By the  definition of sec$'$, we have  $\sigma(S)\leq\mathrm{sec}'(h)$. Additionally, if $\mathrm{sec}(h)=1$, then $\mathrm{sec}'(h)=\sigma(S)$. Hence, in the case that $h:R\to S$ is a ring isomorphism, we obtain  $\mathrm{sec}'(h)=\sigma(R)=\sigma(S)$. Thus, our $\mathrm{sec}'(-)$ serves as a generalization of the covering number for rings. 
    \end{enumerate}
For instance, from the last part of Example~\ref{sec-homomorphism-groups}, it follows that $\mathrm{sec}'(h)=\infty$ for any ring homomorphism $h:R\to\mathbb{Q}$.    
\end{example}

In the following example, we will explore the sectional number in the category of modules.  

\begin{example}[Sectional number in the category of modules]\label{sec-homomorphism-modules} 
Recall that $R\text{-Mod}$ denotes the category of $R$-modules and $R$-homomorphisms.  Let $h:M\to N$ be an $R$-homomorphism, and let $\mathcal{R}$ be an equivalence relation over $\text{Mor}_{R\text{-Mod}}(M,N)$ for each pair of modules $M$ and $N$, preserving composition.
\begin{enumerate}
    \item[(i)]  Consider the quasi Grothendieck topology on $R\text{-Mod}$ presented in Example~\ref{exam:submodule-cover}. From Example~\ref{module-category-covers}(3)(i), given a $R$-module $M$, a collection of inclusions $\{M_\lambda\hookrightarrow M\}_{\lambda\in\Lambda}$ forms a covering for $M$ if $\{M_\lambda\}_{\lambda\in\Lambda}$ consists of submodules of $M$ and $M=\bigcup_{\lambda\in\Lambda} M_\lambda$. 
    
    By Proposition~\ref{prop:mono-sec-coincides}, the $\mathcal{R}$-sectional number of $h$ coincides with the $\mathcal{R}$-mono sectional number of $h$. For instance, we obtain a notion of sectional number for an $R$-homomorphism, denoted $\mathrm{sec}(h)$. Similarly, we obtain a notion of sectional number for an $R$-homomorphism based on proper submodules, denoted  $\mathrm{sec}'(h)$. 
    
    \item[(ii)]  Consider the topology on $R\text{-Mod}$ presented in Example~\ref{exam:projectivesubmodules}. From   Example~\ref{module-category-covers}(3)(ii), given a $R$-module $M$, a collection of inclusions $\{M_\lambda\hookrightarrow M\}_{\lambda\in\Lambda}$ forms a covering for $M$ whenever $\{M_\lambda\}_{\lambda\in\Lambda}$ consists of projective submodules of $M$ and $M=\bigcup_{\lambda\in\Lambda} M_\lambda$. 
    
    By Proposition~\ref{prop:mono-sec-coincides}, the $\mathcal{R}$-sectional number of $h$ coincides with the $\mathcal{R}$-mono sectional number of $h$. For instance, we obtain a new notion of sectional number for an $R$-homomorphism based on projective submodules, which we will denote $\mathrm{sec}_{proj}(h)$. Similarly,  we obtain a new notion of sectional number for an $R$-homomorphism based on proper projective submodules, denoted  $\mathrm{sec}'_{proj}(h)$. 
\end{enumerate}
Moreover, the following statements hold:
\begin{enumerate}
        \item[(1)] If $\mathrm{sec}(h)<\infty$, then $h$ is surjective. Moreover, $\mathrm{sec}(h)=1$ if and only if there exists an $R$-homomorphism  $s:N\to M$ such that $h\circ s=1_N$. In contrast, the inequality $\mathrm{sec}'(h)\geq 3$ always holds (see the next Item (3)). 
        \item[(2)] The inequality $\mathrm{sec}(h)\leq \min\{\mathrm{sec}'(h),\mathrm{sec}_{proj}(h)\}\leq \mathrm{sec}'_{proj}(h)$ always holds. In particular, we have $\mathrm{sec}_{proj}(h)=\mathrm{sec}(h)=1$ whenever $h:M\to N$ is surjective and $N$ is projective. Furthermore, if $\mathrm{sec}(h)\geq 2$, then $\mathrm{sec}(h)=\mathrm{sec}'(h)$ and $\mathrm{sec}_{proj}(h)=\mathrm{sec}'_{proj}(h)$. 
        \item[(3)] Given a $R$-module $M$, the minimal number of proper submodules needed to cover $M$ is called the \textit{covering number} of $M$, denoted by $\sigma(M)$. This number was studied in \cite{khare-tikaradze2022} in the case where $R$ is a commutative ring. It is easy to see that 
$\sigma(M)\geq\sigma(M,+)$, and thus $\sigma(M)\geq 3$. Furthermore, $\sigma(M)\leq \sigma(N)$ whenever $h:M\to N$ is surjective. By the definition of sec$'$, we have $\sigma(N)\leq\mathrm{sec}'(h)$. In addition, if $\mathrm{sec}(h)=1$, then $\mathrm{sec}'(h)=\sigma(N)$. Hence, in the case that $h:M\to N$ is an $R$-isomorphism, we have  $\mathrm{sec}'(h)=\sigma(M)=\sigma(N)$. Thus, $\mathrm{sec}'(-)$ serves as a generalization of the covering number for modules.
    
         \item[(4)] Given a $R$-module $M$, the minimal number of proper projective submodules needed to cover $M$ is called the \textit{projective covering number} of $M$, denoted by $\sigma_{proj}(M)$. It is easy to see that $\sigma_{proj}(M)\geq \sigma(M)$. Furthermore, we emphasise that the equality $\sigma_{proj}(N)=\mathrm{sec}'_{proj}(h)$ holds whenever $h$ is surjective (here, we use the fact that any surjective $R$-homomomorphism $M\twoheadrightarrow N$ admits a local section on each projective submodule $P\subset N$).  
    \end{enumerate}
\end{example}

Next, we will explore the sectional number in the category given by  a group.

\begin{example}[Sectional number in a group]\label{exam:sectional-in-group} 
    Let $G$ be a group. From Example~\ref{exam:category-equal-group}, recall that any sieve $S$ on the unique object $\ast$ coincides with $G_{\mid \ast}$. Therefore, the discrete and trivial Grothendieck topology on $G$ are the same, and thus for any $g\in G$, $\mathcal{R}\text{-sec}(g)=1$ for any reflexive relation $\mathcal{R}$. 
    
    On the other hand, consider the relation $\mathcal{R}'$ defined as follows: for $a,b\in G$, $a\mathcal{R}'b$ if and only if $aba^{-1}b^{-1}ab=1$. It follows that $\mathcal{R}'\text{-sec}(g)=1$ if and only if there is an element $h\in G$ such that the equation \begin{align*}
       gxhx^{-1}g^{-1}h^{-1}gxh&=1. 
    \end{align*} admits a solution. Furthermore, $\mathcal{R}'\text{-sec}(g)<\infty$ if and only if $\mathcal{R}'\text{-sec}(g)=1$.  
\end{example}

Now, we will explore the sectional number in the category of graphs. We note that a previous version of the sectional number for a graph homomorphism appeared in the paper \cite{zapata2023}. 

\begin{example}[Sectional number in Graph]\label{exam:sectional-number-graph}  Consider the quasi Grothendieck topology on $\text{Graph}$ presented in Example~\ref{exam:subgraph-cover}. From Example~\ref{covering-graph-covering}, recall that, given a graph $G$, a collection of inclusions $\{G_\lambda\hookrightarrow G\}_{\lambda\in\Lambda}$ is a covering for $G$ whenever $\{G_\lambda\}_{\lambda\in\Lambda}$ consists of subgraphs of $G$ and  $G=\bigcup_{\lambda\in\Lambda} G_\lambda$. 

Let $f:G\to H$ be a graph homomorphism and $\mathcal{R}$ be an equivalence relation on $\text{Mor}_{\text{Graph}}(G,H)$ for each pair of graphs $G$ and $H$, preserving composition. By Proposition~\ref{prop:mono-sec-coincides}, the $\mathcal{R}$-sectional number of $f$ coincides with the $\mathcal{R}$-mono sectional number of $f$. 

For instance, the sectional number of $f$, denoted $\mathrm{sec}(f)$, coincides with the least interger $n$ such that there exist subgraphs $H_1,\ldots,H_n$ of $H$ satisfying $H=H_1\cup\cdots\cup H_n$ and such that for each $H_j$, there exists a section $s_j:H_j\to G$ of $f$. Thus, we recover the notion of sectional number for a graph homomorphism introduced in \cite{zapata2023}. 
\end{example}

\subsection{Coverings by subobjects} A \textit{subobject} of an object $Y\in\mathcal{C}$ is an equivalence class  $[h:U\rightarrowtail Y]$ of monomorphisms with target $Y$ under the equivalence relation $\sim$ defined as follows: for monomorphisms $u:T\rightarrowtail Y$ and $v:S\rightarrowtail Y$, \[
u\sim v \text{ if and only if there exists an isomorphism } \varphi:T\to S \text{ such that } u=v\circ \varphi.\]

The equivalence class of the identity morphism $[1_Y]$ is a subobject of $Y$. We will denote this class by the same letter $Y$, that is, $Y=[1_Y]$, and we say that $Y$ is a subobject of itself.  

\begin{remark}\label{rem:relation-sieve}
\noindent\begin{enumerate}
    \item[(1)] Let $h\colon E\rightarrowtail Y$ be a monomorphism. Note that $[h]=[1_Y]$, which means $h\sim 1_Y$ if and only if $h$ is an isomorphism.
    \item[(2)] For monomorphisms $u:T\rightarrowtail Y$ and $v:S\rightarrowtail Y$, we have that $u\sim v$ if and only if $\mathcal{C}^{(0)}_{|Y}\{u:T\rightarrowtail Y\}=\mathcal{C}^{(0)}_{|Y}\{v:S\rightarrowtail Y\}$. To see why this is true, the necessary condition follows directly. For the sufficient condition, if $\mathcal{C}^{(0)}_{|Y}\{u:T\rightarrowtail Y\}=\mathcal{C}^{(0)}_{|Y}\{v:S\rightarrowtail Y\}$, then there exist morphisms $\varphi:T\to S$ and $\psi:S\to T$ such that $u=v\circ \varphi$ and $v=u\circ\psi$. Noting that $u\circ 1_T=u=u\circ \left(\psi\circ\varphi\right)$, we conclude that  $\psi\circ\varphi=1_T$ (here, we use the fact that $u$ is a monomorphism). Similarly, we obtain $\varphi\circ\psi=1_S$. Thus, $\varphi$ is an isomorphism, with $\psi$ as its inverse, which implies that $u=v\circ \varphi$. Hence, $u\sim v$. 
    \end{enumerate}  
\end{remark}

Let $f:X\to Y$ be a morphism, and let $u:T\rightarrowtail Y$ and $v:S\rightarrowtail Y$ be monomorphisms such that $u\sim v$. In addition, suppose that $\mathcal{R}$ is an equivalence relation over $\text{Mor}_{\mathcal{C}}(c,c^\prime)$ for each $c,c^\prime\in \mathcal{C}$ that preserves composition; that is, if $f\mathcal{R}h$ in $\text{Mor}_{\mathcal{C}}(c,c^\prime)$ and $f'\mathcal{R}h'$ in $\text{Mor}_{\mathcal{C}}(c'',c)$, then $(f\circ f')\mathcal{R}(h\circ h')$. 

\medskip Note that $f$ admits a $\mathcal{R}$-$u$-section if and only if it  admits a $\mathcal{R}$-$v$-section. This follows from the condition that $\mathcal{R}$ preserves composition, together with the fact that $\mathcal{R}$ is an equivalence relation. In particular, if there is a $\mathcal{R}$-section to $f$ over the domain of a monomorphism in a class, then there is also a $\mathcal{R}$-section to $f$ over the domains of all other monomorphisms in that class. We say that such a class \textit{admits a $\mathcal{R}$-section} to $f$. In the case that $\mathcal{R}$ is the trivial relation, we refer to a \aspas{$\mathcal{R}$-section} simply as a \aspas{section}. 

\medskip Let $\mathcal{C}$ be a category with covers. Consider a collection of subobjects $\{U_i\}_{i\in I}$ of $Y$. If a collection of monomorphisms $\{(F_i\stackrel{f_i}{\rightarrowtail} Y)\in U_i\}_{i\in I}$ forms a covering for $Y$, then any other collection of monomorphisms $\{(G_i\stackrel{g_i^\prime}{\rightarrowtail} Y)\in U_i\}_{i\in I}$ is also a covering for $Y$. We say that such a collection of subobjects $\{U_i\}_{i\in I}$ is a \textit{covering} for $Y$. Indeed, from Remark~\ref{rem:relation-sieve}(2), we have \[\mathcal{C}_{\mid Y}^{(0)}\{f_i:F_i\rightarrowtail Y\}_{i\in I}=\mathcal{C}_{\mid Y}^{(0)}\{g_i:G_i\rightarrowtail Y\}_{i\in I}.\]

We present direct examples of coverings. 

 \begin{example}
\noindent\begin{enumerate}
    \item[(1)] Example~\ref{exam:identity-covering} states that $\{Y\}$ is a covering for $Y$. Recall that $Y=[1_Y]$ is a subobject of itself.  
    \item[(2)] In a category $\mathcal{C}$ equipped with the discrete Grothendieck topology, any collection $\{U_\lambda\}_{\lambda\in\Lambda}$ of subobjects of $Y\in\mathcal{C}$ constitutes a covering for $Y$ (see Example~\ref{exam:discrete-gro}).  
\end{enumerate}  
 \end{example}

The following statement indicates that the mono sectional number can be expressed in terms of coverings by subobjects. 

\begin{proposition}\label{rem:via-subobjects} 
  Let $\mathcal{C}$ be a category with covers, and suppose that $\mathcal{R}$ is an equivalence relation on $\text{Mor}_{\mathcal{C}}(c,c^\prime)$ for each $c,c^\prime\in \mathcal{C}$ that preserves composition. 
  
  For a  morphism $f:X\to Y$ in $\mathcal{C}$, the $\mathcal{R}$-mono sectional number $\mathcal{R}$-msec$(f)$ coincides with the least integer $m$ such that the target $Y$ can be covered by $m$ subobjects $U_i$, where each $U_i$ admits a $\mathcal{R}$-section to $f$. Such a covering $\{U_i\}_{i=1}^m$ is called \textit{categorical by subobjects}.
\end{proposition}

For instance, the mono sectional number msec$(f)$ coincides with the least integer $m$ such that the target $Y$ can be covered by $m$ subobjects $U_i$, with each $U_i$ admitting a section to $f$.

\subsection{Characterization by sieves} We have the following characterization of the mono-sectional and sectional numbers in terms of generated sieves, respectively.

\begin{proposition}[Generated sieves]\label{prop:characterisation-sieves}
    Let $\mathcal{C}$ be a category with covers and $f:X\to Y$ be a morphism in $\mathcal{C}$. Then
    \begin{enumerate}
        \item[(1)] $\text{msec}(f)$ coincides with the least integer $m$ such that there are monomorphisms $h_1\colon Y_1\rightarrowtail Y,\ldots,h_m\colon Y_m\rightarrowtail Y$ for which $\mathcal{C}^{(0)}_{| Y}\{h_j\colon Y_j\rightarrowtail Y\}_{j=1}^{m}$ is a covering sieve on $Y$ and $\mathcal{C}^{(0)}_{| Y}\{f,h_1,\ldots,h_m\}=\mathcal{C}^{(0)}_{| Y}\{f\}$. 
         \item[(2)] $\text{sec}(f)$ coincides with the least integer $m$ such that there are morphisms $h_1\colon Y_1\to Y,\ldots,h_m\colon Y_m\to Y$ such that $\mathcal{C}^{(0)}_{| Y}\{h_j\colon Y_j\to Y\}_{j=1}^{m}$ forms a covering sieve on $Y$ and $\mathcal{C}^{(0)}_{| Y}\{f,h_1,\ldots,h_m\}=\mathcal{C}^{(0)}_{| Y}\{f\}$. 
    \end{enumerate}
\end{proposition}
\begin{proof}
Given morphisms $h_1\colon Y_1\to Y,\ldots,h_m\colon Y_m\to Y$, note that $\mathcal{C}^{(0)}_{| Y}\{f,h_1,\ldots,h_m\}=\mathcal{C}^{(0)}_{| Y}\{f\}$ if and only if for each $j=1,\ldots,m$, the morphism $f$ admits a $h_j$-section. 
\end{proof}

Note that $\mathcal{C}^{(0)}_{| Y}\{f\colon X\to Y\}$ is a sieve, but it is not necessarily a covering sieve. 

\begin{example} Let $\mathcal{C}$ be a category with a topology satisfying axiom (T1). 
    If $f:X\to Y$ admits a $1_Y$-section $Y\to X$, then $\text{msec}(f)=1$ and $\mathcal{C}^{(0)}_{| Y}\{f\colon X\to Y\}=\mathcal{C}_{| Y}$. In particular, $\mathcal{C}^{(0)}_{| Y}\{f\colon X\to Y\}$ is a covering sieve.
\end{example}

\begin{example}
  If $\mathcal{C}^{(0)}_{| Y}\{f\colon X\to Y\}$ is a covering sieve, then $\text{sec}(f)=1$. In addition, if $f$ is a monomorphism then $\text{msec}(f)=1$.   
\end{example}

\subsection{Invariance}

Let $\mathcal{C}$ be a category with covers, and suppose that $\mathcal{R}$ is an equivalence relation over $\text{Mor}_{\mathcal{C}}(c,c^\prime)$ for each $c,c^\prime\in \mathcal{C}$, preserving composition. Two morphisms $f:X\to Y$ and $f^\prime:X^\prime\to Y$ are said to be \textit{$\mathcal{R}$-fibrewise   
equivalent} (or $\mathcal{R}$-FE) if there are morphisms $\psi:X\to X'$, $\varphi:X'\to X$ such that $(f'\circ\psi)\mathcal{R}f$ and $(f\circ\varphi)\mathcal{R}f'$. This can be represented by the following commutative diagrams: 
\begin{eqnarray*}
\xymatrix{ X \ar[rr]^{\psi} \ar[rd]_{f} & & X^{\prime} \ar[ld]^{f^\prime} & \\
        &  Y & &} & \xymatrix{ X^{\prime} \ar[rd]_{f^\prime}  \ar[rr]^{\varphi}  & & X \ar[ld]^{f}& \\
        &  Y & &}
\end{eqnarray*} Such maps $\psi$ and $\varphi$ are called \textit{$\mathcal{R}$-fibrewise morphism}. 

\medskip The following statement establishes the $\mathcal{R}$-FE-invariance of $\text{$\mathcal{R}$-}\mathrm{sec}(-)$ and $\text{$\mathcal{R}$-}\mathrm{msec}(-)$. 

\begin{theorem}[Invariance]\label{r-fe-invariance}
  Let $\mathcal{C}$ be a category with covers, and suppose that $\mathcal{R}$ is an equivalence relation over $\text{Mor}_{\mathcal{C}}(c,c^\prime)$ for each $c,c^\prime\in \mathcal{C}$, preserving composition.  Given morphisms $f:X\to Y$ and $f^\prime:X^\prime\to Y$, assume there exists a $\mathcal{R}$-fibrewise morphism $\psi:X\to X^\prime$; that is, $(f^\prime\circ \psi)\mathcal{R}f$. Then we have \[\text{$\mathcal{R}$-}\mathrm{sec}(f)\geq \text{$\mathcal{R}$-}\mathrm{sec}(f') \text{\quad and \quad} \text{$\mathcal{R}$-}\mathrm{msec}(f)\geq \text{$\mathcal{R}$-}\mathrm{msec}(f').\] In particular, if $f$ and $f'$ are $\mathcal{R}$-FE, then \[\text{$\mathcal{R}$-}\mathrm{sec}(f)=\text{$\mathcal{R}$-}\mathrm{sec}(f') \text{\quad and \quad} \text{$\mathcal{R}$-}\mathrm{msec}(f)=\text{$\mathcal{R}$-}\mathrm{msec}(f').\]
\end{theorem}
\begin{proof}
 Suppose there is a covering $\{g_j:Y_j\to Y\}_{j=1}^m$ for $Y$ such that for each $j=1,\ldots,m$, $f$ admits a $\mathcal{R}$-$g_j$-section, that is, a morphism $s:Y_j\to X$ in $\mathcal{C}$ with $(f\circ s)\mathcal{R}g_j$. Since $\mathcal{R}$ preserves composition, we have that $(f'\circ \psi\circ s)\mathcal{R}(f\circ s)$ because $(f'\circ \psi)\mathcal{R}f$. Thus, we obtain $(f'\circ \psi\circ s)\mathcal{R}g_j$, indicating that $\psi\circ s$ is a $\mathcal{R}$-$g_j$-section of $f'$. Therefore, we conclude that  $\text{$\mathcal{R}$-}\mathrm{sec}(f)\geq \text{$\mathcal{R}$-}\mathrm{sec}(f')$.

 Similarly, suppose there is a covering $\{g_j:Y_j\rightarrowtail Y\}_{j=1}^m$ for $Y$ such that for each $j=1,\ldots,m$, $g_j$ is a monomorphism and $f$ admits a $\mathcal{R}$-$g_j$-section, that is, a morphism $s:Y_j\to X$ in $\mathcal{C}$ with $(f\circ s)\mathcal{R}g_j$. Since $\mathcal{R}$ preserves composition, we have $(f'\circ \psi\circ s)\mathcal{R}(f\circ s)$ because $(f'\circ \psi)\mathcal{R}f$. Therefore, we obtain $(f'\circ \psi\circ s)\mathcal{R}g_j$, indicating that $\psi\circ s$ is a $\mathcal{R}$-$g_j$-section of $f'$. Consequently, we conclude that $\text{$\mathcal{R}$-}\mathrm{msec}(f)\geq \text{$\mathcal{R}$-}\mathrm{msec}(f')$. 
\end{proof}

Theorem~\ref{r-fe-invariance} generalizes property (2) mentioned in  the introduction. 

\medskip Given $f,h\in\text{Mor}_{\mathcal{C}}(X,Y)$, note that if $f\mathcal{R}h$, then $(f\circ 1_X)\mathcal{R}h$ and $f\mathcal{R}(h\circ 1_X)$, which implies that $f$ and $h$ are $\mathcal{R}$-FE. Therefore, Theorem~\ref{r-fe-invariance} leads to the following statement. 

\begin{corollary}\label{cor:fRg-equal-sec}
Let $\mathcal{C}$ be a category with covers, and suppose that $\mathcal{R}$ is an equivalence relation over $\text{Mor}_{\mathcal{C}}(c,c^\prime)$ for each $c,c^\prime\in \mathcal{C}$, preserving composition. Let $f,h\in\text{Mor}_{\mathcal{C}}(X,Y)$. If $f\mathcal{R}h$, then we have \[\text{$\mathcal{R}$-sec}(f)=\text{$\mathcal{R}$-sec}(h) \text{\quad and \quad} \text{$\mathcal{R}$-msec}(f)=\text{$\mathcal{R}$-msec}(h).\]
\end{corollary}

Corollary~\ref{cor:fRg-equal-sec} generalizes property (3) mentioned in the introduction.

\subsection{Weak pullback}
Let $\mathcal{C}$ be a category. From Definition~\ref{defn:quasi-pullback}, a weak pullback in $\mathcal{C}$ is defined by a strictly commutative diagram: 
\begin{eqnarray}\label{xfyzz}
\xymatrix{ \rule{3mm}{0mm}& X^\prime \ar[r]^{\varphi'} \ar[d]_{f^\prime} & X \ar[d]^{f} & \\ &
       Y^\prime  \ar[r]_{\,\,\varphi} &  Y &}
\end{eqnarray} 
such that for any strictly commutative diagram as the one on the left hand-side of~(\ref{ddiagramadoble}), there exists a (not necessarily unique) morphism $h:Z\to X^\prime$ that renders a strictly commutative diagram as the one on the right hand-side of~(\ref{ddiagramadoble}).
\begin{eqnarray}\label{ddiagramadoble}
\xymatrix{
Z \ar@/_10pt/[dr]_{\alpha} \ar@/^30pt/[rr]^{\beta} & & X \ar[d]^{f}  & & &
Z\rule{-1mm}{0mm} \ar@/_10pt/[dr]_{\alpha} \ar@/^30pt/[rr]^{\beta}\ar[r]^{h} & 
X^\prime \ar[r]^{\azul{\varphi'}} \ar[d]_{f^\prime} & X \\
& Y^\prime  \ar[r]_{\,\,\varphi} &  Y & & & & Y^\prime &  \rule{3mm}{0mm}}
\end{eqnarray}

We have the following statement that describes the behaviour of the (mono) sectional number under weak pullbacks.

\begin{theorem}[Under weak pullback]\label{quasi-pullback-sec}
Let $\mathcal{C}$ be a category equipped with a quasi Grothendieck topology and weak pullbacks. If~(\ref{xfyzz}) is a weak pullback in $\mathcal{C}$, then:
\begin{itemize}
    \item[(1)] $\text{sec}\hspace{.1mm}(f^\prime)\leq \text{sec}\hspace{.1mm}(f)$.
    \item[(2)] Suppose that $\mathcal{C}$ is a category with pullbacks. Then $\text{msec}\hspace{.1mm}(f^\prime)\leq \text{msec}\hspace{.1mm}(f)$. 
\end{itemize}
\end{theorem}
\begin{proof}
\noindent\begin{enumerate}
    \item[(1)] Suppose that $E\stackrel{h}{\to} Y$ is a morphism having codomain $Y$ such that there exists a morphism $s:E\to X$ in $\mathcal{C}$ with $f\circ s=h$. We consider the morphism $\widetilde{h}:\widetilde{E}\to Y'$ given by a weak pullback: 
    \begin{eqnarray*}
\xymatrix{ \rule{3mm}{0mm}& \widetilde{E} \ar[r]^{\,\,\varphi_{|}} \ar[d]_{\widetilde{h}} & E \ar[d]^{h} & \\ &
       Y^\prime  \ar[r]_{\,\,\varphi} &  Y &}
\end{eqnarray*} Next, take the following strictly commutative diagram as the one on the left hand-side of~(\ref{1-diagramadobleee}). Since the diagram~(\ref{xfyzz}) is a weak pullback, there exists a morphism $\widetilde{s}:\widetilde{E}\to X'$ that renders a strictly commutative diagram as the one on the right hand-side of~(\ref{1-diagramadobleee}).
\begin{eqnarray}\label{1-diagramadobleee}
\xymatrix{
\widetilde{E} \ar@/_10pt/[dr]_{\widetilde{h}} \ar@/^30pt/[rr]^{s\circ \varphi_{|}} & & X \ar[d]^{f}  & & &
\widetilde{E}\rule{-1mm}{0mm} \ar@/_10pt/[dr]_{\widetilde{h}} \ar@/^30pt/[rr]^{s\circ \varphi_{|}}\ar[r]^{\widetilde{s}} & 
X' \ar[r]^{\varphi'} \ar[d]_{f'} & X \\
& Y'  \ar[r]_{\,\,\varphi} &  Y & & & & Y' &  \rule{3mm}{0mm}}
\end{eqnarray} Then, $\widetilde{s}$ is a $\widetilde{h}$-section of $f'$.

Now, if $\text{sec}\hspace{.1mm}(f)=m$ and $\{E_i\stackrel{h_i}{\to} Y\}_{i=1}^m$ is a categorical covering for $f$, then by the construction above, together with Proposition~\ref{lem;pullback} and axiom $(T2)$, we obtain that $\{\widetilde{h_i}\colon\widetilde{E_i}\to Y'\}_{i=1}^m$ is a categorical covering for $f'$. Therefore, the inequality $\text{sec}\hspace{.1mm}(f^\prime)\leq \text{sec}\hspace{.1mm}(f)$ holds.
    \item[(2)] Here, we use Proposition~\ref{rem:via-subobjects}. Suppose that $U=[E\stackrel{h}{\rightarrowtail} Y]$ is a subobject of $Y$ such that there exists a morphism $s:E\to X$ in $\mathcal{C}$ with $f\circ s=h$. Recall that in any category, a pullback of a monomorphism along any arrow is itself a monomorphsm \cite[p. 30]{maclane1992}. We consider the monomorphism $\widetilde{h}:\varphi^{-1}(E)\rightarrowtail Y'$ given by the pullback: 
    \begin{eqnarray*}
\xymatrix{ \rule{3mm}{0mm}& \varphi^{-1}(E) \ar[r]^{\,\,\varphi_{|}} \ar@{^{(}->}[d]_{\widetilde{h}} & E \ar@{^{(}->}[d]^{h} & \\ &
       Y^\prime  \ar[r]_{\,\,\varphi} &  Y &}
\end{eqnarray*} Next, take the following strictly commutative diagram as the one on the left hand-side of~(\ref{diagramadobleee}). Since the diagram~(\ref{xfyzz}) is a weak pullback, there exists a morphism $\widetilde{s}:\varphi^{-1}(E)\to X'$ that renders a strictly commutative diagram as the one on the right hand-side of~(\ref{diagramadobleee}):
\begin{eqnarray}\label{diagramadobleee}
\xymatrix{
\varphi^{-1}(E) \ar@{^{(}->}@/_10pt/[dr]_{\widetilde{h}} \ar@/^30pt/[rr]^{s\circ \varphi_{|}} & & X \ar[d]^{f}  & & &
\varphi^{-1}(E)\rule{-1mm}{0mm} \ar@{^{(}->}@/_10pt/[dr]_{\widetilde{h}} \ar@/^30pt/[rr]^{s\circ \varphi_{|}}\ar[r]^{\widetilde{s}} & 
X' \ar[r]^{\varphi'} \ar[d]_{f'} & X \\
& Y'  \ar[r]_{\,\,\varphi} &  Y & & & & Y' &  \rule{3mm}{0mm}}
\end{eqnarray} Then, the subobject $V=[\widetilde{h}\colon\varphi^{-1}(E)\rightarrowtail Y']$ of $Y'$ admits a section of $f'$.

Now, if $\text{msec}\hspace{.1mm}(f)=m$ and $\{U_i=[E_i\stackrel{h_i}{\rightarrowtail} Y]\}_{i=1}^m$ is a categorical covering by subobjects for $Y$, then by the construction above, together with Proposition~\ref{lem;pullback} and axiom $(T2)$, we obtain that $\{V_i=[\widetilde{h_i}\colon\varphi^{-1}(E_i)\rightarrowtail Y']\}_{i=1}^m$ is a covering for $Y'$ such that each subobject $V_i$ of $Y'$ admits a section of $f'$. Therefore, the inequality $\text{msec}\hspace{.1mm}(f^\prime)\leq \text{msec}\hspace{.1mm}(f)$ holds.
\end{enumerate}
\end{proof}

Theorem~\ref{quasi-pullback-sec} generalizes property (4) mentioned it the introduction.

\begin{remark}
The corresponding assertion of Theorem~\ref{quasi-pullback-sec} 
for $\text{$\mathcal{R}$-}\mathrm{sec}\hspace{.1mm}$ may fail. This can be illustrated by considering the category of spaces together with homotopy relation. For example, consider the constant map $\overline{1}:X\to\mathbb{R}^2$ and the inclusion $i:S^1\hookrightarrow \mathbb{R}^2$: \begin{eqnarray*}
\xymatrix{  & X \ar[d]^{\overline{1}} & \\
       S^1  \ar@{^{(}->}[r]_{\,\, i} &  \mathbb{R}^2 &}
\end{eqnarray*} Note that its canonical pullback is given by: 
\begin{eqnarray*}
\xymatrix{ X\ar[r]^{\,\,} \ar[d]_{\overline{1}} & X \ar[d]^{\overline{1}} & \\
       S^1 \ar@{^{(}->}[r]_{\,\, i} &  \mathbb{R}^2 &}
\end{eqnarray*} Thus, we have $\mathrm{secat}\hspace{.1mm}(\overline{1}:X\to S^1)=\mathrm{cat}(S^1)=2$ and $\mathrm{secat}\hspace{.1mm}(\overline{1}:X\to \mathbb{R}^2)=\mathrm{cat}(\mathbb{R}^2)=1$.
\end{remark}

\subsection{Continuous functors} Let $\mathcal{C}$ be a category with covers, and suppose that $\mathcal{R}$ is an equivalence relation over $\text{Mor}_{\mathcal{C}}(c,c^\prime)$ for each $c,c^\prime\in \mathcal{C}$. Similarly, let $\mathcal{D}$ be a category with covers,  and suppose that $\mathcal{R}'$ is an equivalence relation over $\text{Mor}_{\mathcal{D}}(d,d^\prime)$ for each $d,d^\prime\in \mathcal{D}$. 

\medskip Let $\mathfrak{F}:\mathcal{C}\to\mathcal{D}$ be a \textit{continuous functor} (that is, if $\{A_\lambda\stackrel{i_\lambda}{\to} X\}_{\lambda\in\Lambda}$ is a covering for $X$, then $\{\mathfrak{F}(A_\lambda)\stackrel{\mathfrak{F}(i_\lambda)}{\to} \mathfrak{F}(X)\}_{\lambda\in\Lambda}$ is a covering for $\mathfrak{F}(X)$) such that $\mathfrak{F}(f)\mathcal{R}'\mathfrak{F}(g)$ whenever $f\mathcal{R}g$. 

\medskip Under these conditions, the following statement holds:

\begin{theorem}[Continuous functors]\label{thm:continuous-functor}
Let $f:X\to Y$ be a morphism in $\mathcal{C}$. Under the conditions stated above, we have \[\text{$\mathcal{R}'$-sec}(\mathfrak{F}(f))\leq \text{$\mathcal{R}$-sec}(f).\] Furthermore, if  $\mathfrak{F}$ preserves monomorphisms, then \[\text{$\mathcal{R}'$-msec}(\mathfrak{F}(f))\leq \text{$\mathcal{R}$-msec}(f).\]  
\end{theorem}
\begin{proof}
 Set $\mathcal{R}\text{-}\mathrm{sec}(f)=m<\infty$, and consider a covering $\{j_\ell:B_\ell\to Y\}_{\ell=1}^{m}$ of $Y$ such that for each $\ell=1,\ldots,m$, $f$ admits a $\mathcal{R}$-$j_\ell$-section:
  \begin{eqnarray*}
\xymatrix{ X \ar[rr]^{f} & &Y   \\
        &  B_\ell\ar@{-->}[lu]^{}\ar[ru]_{j_\ell} & } 
\end{eqnarray*}  Then, the collection $\{\mathfrak{F}(j_\ell):\mathfrak{F}(B_\ell)\to \mathfrak{F}(Y)\}_{\ell=1}^{m}$ forms a covering of $\mathfrak{F}(Y)$ such that for each $\ell=1,\ldots,m$, $\mathfrak{F}(f)$ admits a $\mathcal{R}'$-$\mathfrak{F}(j_\ell)$-section:  
 \begin{eqnarray*}
\xymatrix{ \mathfrak{F}(X) \ar[rr]^{\mathfrak{F}(f)} & &\mathfrak{F}(Y)   \\
        &  \mathfrak{F}(B_\ell)\ar@{-->}[lu]^{}\ar[ru]_{\mathfrak{F}(j_\ell)} & } 
\end{eqnarray*} Therefore, we have $\text{$\mathcal{R}'$-sec}(\mathfrak{F}(f))\leq m=\text{$\mathcal{R}$-sec}(f)$.

Additionally, if each $j_\ell:B_\ell\to Y$ is a monomorphism, then each $\mathfrak{F}(j_\ell):\mathfrak{F}(B_\ell)\to \mathfrak{F}(Y)$ is also a monomorphism (here we use the fact that $\mathfrak{F}$ preserves monomorphisms). Thus, we also obtain the inequality $\text{$\mathcal{R}'$-msec}(\mathfrak{F}(f))\leq \text{$\mathcal{R}$-msec}(f)$.
\end{proof}

One has the following example.

\begin{example}\label{exam:forgetful-functor}
 Note that the usual forgetful functors $\mathfrak{F}:(R\text{-Mod},\text{Proper submodules})\to (\text{Grp},\text{Proper subgroups})$ and $\mathfrak{G}:(\text{Ring},\text{Proper subrings})\to (\text{Grp},\text{Proper subgroups})$ are continuous. By Theorem~\ref{thm:continuous-functor}, it follows that  $\text{sec}'(f)\geq \text{sec}'(\mathfrak{F}(f))$ for any $R$-homomorphism $f$ and $\text{sec}'(h)\geq \text{sec}'(\mathfrak{G}(h))$ for any ring homomorphism $h$.   
\end{example}

\subsection{Upper bound} In this section, we present upper bounds for the sectional number. For this purpose, we introduce the notion of LS category (Definition~\ref{defn:LS-category}) and projective covering number (Definition~\ref{defn:projective-covering-number}) for an object. 

\medskip Let $\mathcal{C}$ be a category with covers. A \textit{weak initial object} of $\mathcal{C}$ is an object $I$ in $\mathcal{C}$ such that for every object $X$ in $\mathcal{C}$, there exists a (not necessarily unique) morphism $I\to X$. Suppose that $\mathcal{R}$ is an equivalence relation over $\text{Mor}_{\mathcal{C}}(c,c^\prime)$ for each $c,c^\prime\in \mathcal{C}$. 

\medskip Fixed a weak initial object $I$ in $\mathcal{C}$; we say that an object $Y$ is \textit{$\mathcal{R}$-connected} if $i\mathcal{R} i'$ (and, of course, $\text{$\mathcal{R}$-sec}(i)=\text{$\mathcal{R}$-sec}(i')$, see Corollary~\ref{cor:fRg-equal-sec}) for any morphisms $i,i':I\to Y$. Hence, we have the following definition.

\begin{definition}[LS category]\label{defn:LS-category}
The \textit{LS category} of a $\mathcal{R}$-connected object $Y$ is given by \[\text{cat}_{\mathcal{R}}(Y)=\text{$\mathcal{R}$-sec}(i)\] for some morphism $i:I\to Y$.   
\end{definition}

Note that in the category $(\text{Top},\text{Open})$, where $I$ is any point and we consider the homotopy relation, our LS category coincides with the well-known Lusternik-Schnirelmann category of path connected spaces (see \cite{cornea2003}). 

\medskip Now, we present the following upper bound for the $\mathcal{R}$-sectional number.

\begin{theorem}[Upper bound]\label{thm:upper-bound-LS-category} 
 Let $\mathcal{C}$ be a category with covers and a weak initial object $I$. Suppose that $\mathcal{R}$ is an equivalence relation over $\text{Mor}_{\mathcal{C}}(c,c^\prime)$ for each $c,c^\prime\in \mathcal{C}$, preserving composition. Let $f:X\to Y$ be any morphism. If $Y$ is $\mathcal{R}$-connected, then \[\text{$\mathcal{R}$-sec}(f)\leq \text{cat}_{\mathcal{R}}(Y).\]    
\end{theorem}
\begin{proof}
    Set $i:I\to X$ to be a morphism and consider $\text{cat}_{\mathcal{R}}(Y)=\text{$\mathcal{R}$-sec}(f\circ i)$. By Theorem~\ref{r-fe-invariance}, we have $\text{$\mathcal{R}$-sec}(f\circ i)\geq \text{$\mathcal{R}$-sec}(f)$. Thus, it follows that $\text{cat}_{\mathcal{R}}(Y)\geq \text{$\mathcal{R}$-sec}(f)$.
\end{proof}

Theorem~\ref{thm:upper-bound-LS-category} generalizes the inequality $\text{secat}(f)\leq\text{cat}(Y)$, which holds for any map $f:X\to Y$ and any path connected space $Y$ (see property (5) mentioned in the introduction). 

\medskip On the other hand, note that Theorem~\ref{thm:upper-bound-LS-category} is most useful when $\mathcal{R}$ is not the trivial relation. For this reason, we present an upper bound for sectional number. Let $\mathcal{C}$ be a category. An object $P$ is called \textit{projective} if and only if, for every epimorphism $f:N\twoheadrightarrow M$ and every morphism $g:P\to M$, there exists a morphism $h:P\to N$ such that $f\circ h=g$.
 \begin{eqnarray*}
\xymatrix@C=3cm{& N \ar@{->>}[d]^{f}   \\  
       P  \ar@{-->}[ru]^{h}\ar[r]_{g} &  M}
\end{eqnarray*}

Motivated by Example~\ref{sec-homomorphism-modules}(4), we introduce the following definition.

\begin{definition}[Projective covering number]\label{defn:projective-covering-number}
 Let $\mathcal{C}$ be a category with covers. Given an object $M$, the \textit{projective covering number} of $M$, denoted by $\sigma_{proj}(M)$, is the least integer $m$ such that there exists a covering $\{g_j:P_j\to M\}_{j=1}^m$ for $M$ with each $P_j$ projective for $j=1,\ldots,m$. Such a covering $\{g_j:P_j\to M\}_{j=1}^m$ is called \textit{projective categorical}. If no such $m$ exists, we set $\sigma_{proj}(M)=\infty$.   
\end{definition}

Thus, we obtain the following upper bound for the sectional number.

\begin{theorem}[Upper bound]\label{thm:upper-bound-covering-number}
 Let $\mathcal{C}$ be a category with covers, and let $f:N\twoheadrightarrow M$ be any epimorphism. Then, we have \[\text{sec}(f)\leq \sigma_{proj}(M).\]   
\end{theorem}
\begin{proof}
This is because any epimorphism $N\twoheadrightarrow M$ admits a $g$-section for each morphism $g:P\to M$ whenever $P$ is projective.
\end{proof}


\section{Cohomology theory over a category}\label{sec:cohomology}
In this section, we introduce a notion of multiplicative cohomology theory for a category with covers and distinguished morphisms (Definition~\ref{multiplicative-cohomology}). We then use this theory to present a cohomological lower bound for the $\mathcal{R}$-sectional number (Theorem~\ref{thm:cohomo-ic}). 

\subsection{Cohomology theory} For our purposes, we present the following universal construction.

\begin{example}\label{exam:lim-colim-construction} Let $\mathcal{C}$ be a category.
   \begin{enumerate}
       \item[(1)][Limit-colimit]\label{exam:limit-colimit} Let $\mathfrak{S}=\{f_i:C_i\to C\}_{i\in I}$ be a family of morphisms in $\mathcal{C}$ with codomain $C$. Consider the following partially ordered set $P=I\sqcup \{\ast\}$ with $i\leq \ast$ for each $i\in I$; and let $\mathcal{C}_P$ be its associated category. Define the diagram $D:\mathcal{C}_P\to \mathcal{C}$ by $D(i)=C_i$ for each $i\in I$ and $D(\ast)=C$; also, let $D(i\leq\ast)=f_i$ for each $i\in I$. Suppose we can take the limit $L=\text{lim}(D)$ of the diagram $D$, together with its morphisms $\{\lambda_p:L\to D(p)\}_{p\in P}$ such that $\lambda_\ast=D(i\leq \ast)\circ\lambda_i$ for each morphism $i\leq \ast$ in $\mathcal{C}_P$.

   Now, consider the following diagram $\widetilde{D}:\mathcal{C}_P^{\text{op}}\to \mathcal{C}$ defined  by $\widetilde{D}(i)=C_i$ for each $i\in I$ and  $\widetilde{D}(\ast)=L$; set $\widetilde{D}(\ast\leq i)=\lambda_i$ for each $i\in I$. Assume that we can take the colimit $\widetilde{L}=\text{colim}(\widetilde{D})$ of the diagram $\widetilde{D}$, together with its morphisms $\{\rho_p:\widetilde{D}(p)\to \widetilde{L}\}_{p\in P}$, such that $\rho_\ast=\rho_i\circ\widetilde{D}(\ast\leq i)$ for each morphism $\ast\leq i$ in $\mathcal{C}_P^{\text{op}}$.

   Note that the family $\mathfrak{S}\cup\{\lambda_\ast:L\to C\}$ of morphisms, whose codomain is $C$, satisfies the equality  $\lambda_\ast=f_i\circ \widetilde{D}(\ast\leq i)$ for each $i\in I$. By the co-universal property of the colimit $\widetilde{L}=\text{colim}(\widetilde{D})$, there exists a unique morphism $\phi:\widetilde{L}\to C$ such that $\phi\circ \rho_i=f_i$ for each $i\in I$ and $\phi\circ\rho_\ast=\lambda_\ast$.

   We denote such a limit $L$ (in the case that it exists) by $\displaystyle{\lim_{{\overset{\to}{i\in I}}} f_i}$ and the colimit $\widetilde{L}$ (in the case that it exists) by $\displaystyle{\widetilde{\lim_{{\overset{\to}{i\in I}}}} f_i}$. 
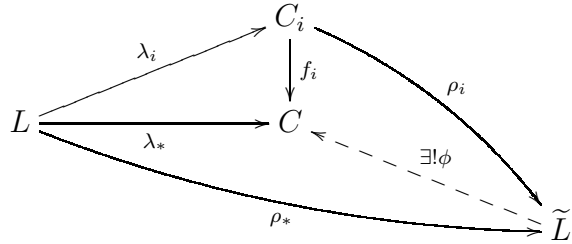
   \begin{eqnarray*}
\xymatrix@C=3cm{
 & C_i \ar[d]^{f_i}\ar@/^10pt/[ddr]^{\rho_i} &   \\  
L \ar[ur]^{\lambda_i} \ar@/_10pt/[drr]_{\rho_\ast} \ar[r]_{\lambda_\ast} & C &  \\ 
 & & \widetilde{L}\ar@{-->}[lu]_{\exists !\phi} } 
\end{eqnarray*}
   For instance:
   \begin{enumerate}
       \item[(i)] In the case that $\mathfrak{S}=\{f:A\to C,g:B\to C\}$, observe that $L=A\times_C B$ is the pullback of $f:A\to C$ and $g:B\to C$. Furthermore, $\widetilde{L}=A\sqcup_{A\times_C B} B$ is the pushout of the morphisms $A\times_C B\stackrel{f^\ast(g)}{\to}A$ and $A\times_C B\stackrel{\widetilde{f}(g)}{\to} B$.
\item[(ii)] Let $f:C'\rightarrowtail C$ be a monomorphism. Note that if $\left(C_i\stackrel{f_i}{\to} C\right)=\left(C'\stackrel{f}{\rightarrowtail} C\right)$ for any $i\in I$, then $\left(\widetilde{L}\stackrel{\phi}{\to} C\right)=\left(C'\stackrel{f}{\rightarrowtail} C\right)$.
   \end{enumerate} 
\item[(2)][Product]\label{exam:product-} Suppose that $\mathcal{C}$ admits finite products. Given morphisms $f_\ell:X_\ell\to X_\ell'$, $i_\ell:A_\ell\to X_\ell$, $i_\ell':A_\ell'\to X_\ell'$, and ${f_\ell}_\mid:A_\ell\to A_\ell'$ such that $f_\ell\circ i_\ell=i_\ell'\circ {f_\ell}_\mid$ for each $\ell=1,\ldots,k$, we define: \begin{align*}
    A^\ell&=X_1\times\cdots\times X_{\ell-1}\times A_\ell\times X_{\ell+1}\times\cdots\times X_k,\\
    i^\ell&=1_{X_1}\times\cdots\times 1_{X_{\ell-1}}\times i_{\ell}\times 1_{X_{\ell+1}}\times\cdots\times 1_{X_k},\\
    f^\ell&=f_{1}\times\cdots\times f_{\ell-1}\times {f_{\ell}}_|\times f_{\ell+1}\times\cdots\times f_{k}.\\
\end{align*} We obtain the following commutative diagram:
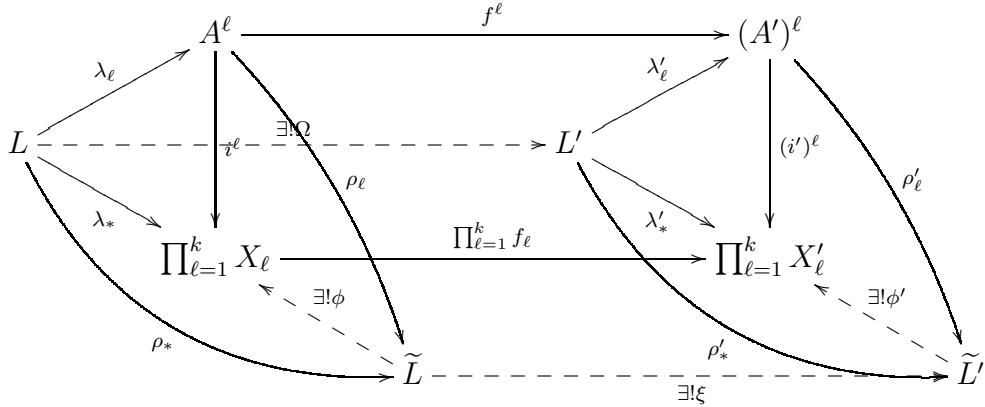
\begin{eqnarray*}
\xymatrix@C=1.5cm{
 & A^\ell \ar[rrr]^{f^\ell} \ar[dd]^{i^\ell}\ar@/^10pt/[dddr]^{\rho_\ell} &    & & (A')^\ell \ar[dd]^{(i')^\ell}\ar@/^10pt/[dddr]^{\rho'_\ell} &\\ 
 L\ar@{-->}[rrr]^{\exists !\Omega} \ar[ur]^{\lambda_\ell} \ar@/_30pt/[ddrr]_{\rho_\ast} \ar[rd]_{\lambda_\ast} & & & L' \ar[ur]^{\lambda'_\ell} \ar@/_30pt/[ddrr]_{\rho'_\ast} \ar[rd]_{\lambda'_\ast}& & \\
 & \prod_{\ell=1}^{k}X_\ell \ar[rrr]^{\prod_{\ell=1}^{k}f_\ell} &  &  & \prod_{\ell=1}^{k}X'_\ell & \\  
 & & \widetilde{L}\ar@{-->}[lu]_{\exists !\phi} \ar@{-->}[rrr]_{\exists !\xi} &  & & \widetilde{L'}\ar@{-->}[lu]_{\exists !\phi'}} 
\end{eqnarray*}
\item[(3)][Diagonal]\label{exam:diagonal-} Suppose that $\mathcal{C}$ admits finite products. Let $A_1\stackrel{i_1}{\to} X,\ldots,A_k\stackrel{i_k}{\to} X$ be morphisms with codomain $X$. Set \begin{align*}
    A^\ell&=X\times\cdots\times X\times A_\ell\times X\times\cdots\times X,\\
    i^\ell&=1_{X}\times\cdots\times 1_{X}\times i_{\ell}\times 1_{X}\times\cdots\times 1_{X},\\
\end{align*} and let $\Delta_\ell:A_\ell\to A^\ell$ be the usual diagonal morphism. We have the following commutative diagram:
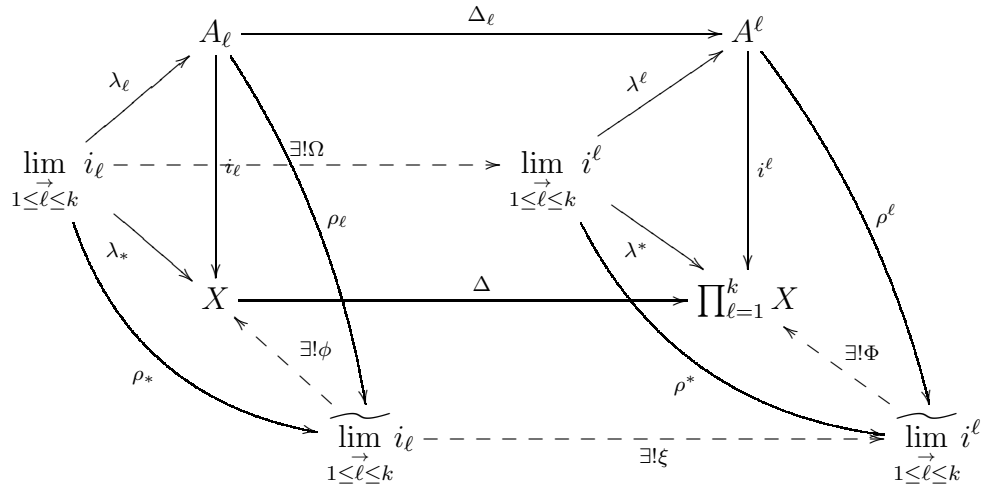
\begin{eqnarray*}
\xymatrix@C=1cm{
 & A_\ell \ar[rrr]^{\Delta_\ell} \ar[dd]^{i_\ell}\ar@/^10pt/[dddr]^{\rho_\ell} &    & & A^\ell \ar[dd]^{i^\ell}\ar@/^10pt/[dddr]^{\rho^\ell} &\\ 
 \displaystyle{\lim_{{\overset{\to}{1\leq \ell\leq k}}} i_\ell}\ar@{-->}[rrr]^{\exists !\Omega} \ar[ur]^{\lambda_\ell} \ar@/_30pt/[ddrr]_{\rho_\ast} \ar[rd]_{\lambda_\ast} & & & \displaystyle{\lim_{{\overset{\to}{1\leq \ell\leq k}}} i^\ell} \ar[ur]^{\lambda^\ell} \ar@/_30pt/[ddrr]_{\rho^\ast} \ar[rd]_{\lambda^\ast}& & \\
 & X \ar[rrr]^{\Delta} &  &  & \prod_{\ell=1}^{k}X  & \\  
 & & \displaystyle{\widetilde{\lim_{{\overset{\to}{1\leq \ell\leq k}}}} i_\ell}\ar@{-->}[lu]_{\exists !\phi} \ar@{-->}[rrr]_{\exists !\xi} &  & & \displaystyle{\widetilde{\lim_{{\overset{\to}{1\leq \ell\leq k}}}} i^\ell}\ar@{-->}[lu]_{\exists !\Phi}} 
\end{eqnarray*} The morphism of pairs \[\Delta=(\Delta,\xi):\left(X,\displaystyle{\widetilde{\lim_{{\overset{\to}{1\leq \ell\leq k}}}} i_\ell}\right)\to \left(\prod_{\ell=1}^{k}X,\displaystyle{\widetilde{\lim_{{\overset{\to}{1\leq \ell\leq k}}}} i^\ell}\right)\] is called the \textit{diagonal morphism}.
   \end{enumerate} 
\end{example}

We now introduce the notion of distinguished morphisms.

\begin{definition}[Distinguished morphisms]\label{defn:distinguished-morphisms}
Let $\mathcal{C}$ be a category with an initial object $I$ and covers. A collection of morphisms in $\mathcal{C}$ is called \textit{distinguished}--refer to as \textit{distinguished morphisms}--if it satisfies the following conditions:
\begin{enumerate}
\item[(1)] For any object $X$, the unique morphism $I\to X$ is distinguished.
\item[(2)] Any morphism $A\to X$ that belongs to a covering of $X$ is distinguished.  
 \item[(3)] The finite product of distinguished morphisms is also distinguished.
    \item[(4)] If $\left(A_1\stackrel{i_1}{\to} X\right),\ldots,\left(A_k\stackrel{i_k}{\to} X\right)$ is any finite collection of distinguished  morphisms with the same codomain, then the morphism $\displaystyle{\widetilde{\lim_{{\overset{\to}{1\leq \ell\leq k}}}} i_\ell}\stackrel{\phi}{\to} X$ (presented in Example~\ref{exam:limit-colimit}(1)) is also distinguished.
\end{enumerate}    
\end{definition}

By Definition~\ref{defn:distinguished-morphisms}(3) and (4), we can make the following remark.

\begin{remark}
   Let $A_1\stackrel{i_1}{\to} X_1,\ldots,A_k\stackrel{i_k}{\to} X_k$ be any finite collection of distinguished morphisms. For each $\ell=1,\ldots,k$, by Definition~\ref{defn:distinguished-morphisms}(3), each morphism \[i^\ell=1_{X}\times\cdots\times 1_{X}\times i_{\ell}\times 1_{X}\times\cdots\times 1_{X}\] is distinguished. Hence, by Definition~\ref{defn:distinguished-morphisms}(4), the morphism \[\displaystyle{\widetilde{\lim_{{\overset{\to}{1\leq \ell\leq k}}}} i^\ell}\stackrel{\phi}{\to} \prod_{\ell=1}^{k} X_\ell\] (presented in  Example~\ref{exam:limit-colimit}(1)) is also distinguished. 
\end{remark}

\medskip 
Let $\mathcal{C}$ be a category with an initial object $I$ and a distinguished collection of morphisms. A \textit{pair} $\left(X,A\stackrel{i}{\to}X\right)$ consists of an object $X\in\mathcal{C}$ together with a distinguished morphism $i:A\to X$ in $\mathcal{C}$ with codomain $X$. Any object $X$ can be considered as the pair $\left(X,I\to X\right)$. 

\medskip A \textit{morphism of pairs} $(f,f_|):\left(X,A\stackrel{i}{\to} X\right)\to \left(Y,B\stackrel{j}{\to} Y\right)$ consists of a morphism $f:X\to Y$ together with a morphism $f_|:A\to B$ such that \begin{align}\label{pair-morphism-eq}
    f\circ i& =j\circ f_|
\end{align}  When the morphism $f_|$ is understood (for example, when it is the only one satisfying equation (\ref{pair-morphism-eq})) we will write $f:\left(X,A\stackrel{i}{\to} X\right)\to \left(Y,B\stackrel{j}{\to} Y\right)$ instead of $(f,f_|):\left(X,A\stackrel{i}{\to} X\right)\to \left(Y,B\stackrel{j}{\to} Y\right)$. 

\medskip The morphism of pairs $(1_X,1_A):\left(X,A\stackrel{i}{\to} X\right)\to \left(X,A\stackrel{i}{\to} X\right)$ will be denoted by $1_{\left(X,A\stackrel{i}{\to} X\right)}$. A morphism of pairs of the form  $f:\left(X,I\to X\right)\to \left(Y,B\stackrel{j}{\to} Y\right)$ will be denoted by $f:X\to \left(Y,B\stackrel{j}{\to} Y\right)$. 

\medskip Given $(f,f_|):\left(X,A\stackrel{i}{\to} X\right)\to \left(Y,B\stackrel{j}{\to} Y\right)$ and $(g,g_|):\left(Y,B\stackrel{j}{\to} Y\right)\to \left(Z,C\stackrel{\kappa}{\to} Z\right)$, 
the composition \[(g,g_|)\circ (f,f_|):\left(X,A\stackrel{i}{\to} X\right)\to \left(Z,C\stackrel{\kappa}{\to} Z\right)\] is defined by   \[(g,g_|)\circ (f,f_|)=(g\circ f,g_|\circ f_|).\] 
Thus, we obtain the category $\mathcal{C}^2$ whose objects are the pairs $\left(X,A\stackrel{i}{\to} X\right)$ and whose morphisms are the morphisms of pairs $f:\left(X,A\stackrel{i}{\to} X\right)\to \left(Y,B\stackrel{j}{\to} Y\right)$. The morphism $1_{\left(X,A\stackrel{i}{\to} X\right)}$ serves as the identity morphism of the pair $\left(X,A\stackrel{i}{\to} X\right)$. When the morphism $i:A\to X$ is understood in the pair $\left(X,A\stackrel{i}{\to} X\right)$, we will simply write the pair as $(X,A)$. 

\medskip Set $R$ be the \textit{restriction functor} defined by:
\begin{eqnarray*}
R: \mathcal{C}^2 &\to& \mathcal{C}^2\\
\left(X,A\right) &\mapsto& A\\
f &\mapsto& f_|.
\end{eqnarray*} 

Let $\mathcal{R}$ be an equivalence relation preserving composition over the set of morphisms of pairs \[\text{Mor}_{\mathcal{C}^2}\left(\left(X,A\right),\left(Y,B\right)\right)\] for each pair $\left(X,A\right),\left(Y,B\right)$. Let $\mathcal{R}\mathcal{C}^2$ be the \textit{$\mathcal{R}$-category} of $\mathcal{C}^2$, which is the category whose objects are the same as those in $\mathcal{C}^2$, but whose morphisms are the $\mathcal{R}$-equivalence classes of morphisms of pairs. Given $[f]\in \text{Mor}_{\mathcal{R}\mathcal{C}^2}\left(\left(X,A\right),\left(Y,B\right)\right)$ and $[g]\in \text{Mor}_{\mathcal{R}\mathcal{C}^2}\left(\left(Y,B\right),\left(Z,C\right)\right)$, we define the composition $[g]\circ [f]$ to be $[g\circ f]$. One can readily verify that $[g]\circ [f]$ is well-defined and constitutes a composition (here we use the fact that $\mathcal{R}$ is an equivalence relation preserving composition). 

\medskip In addition, suppose that for $f,g\in \text{Mor}_{\mathcal{C}^2}\left(\left(X,A\right),\left(Y,B\right)\right)$: $f\mathcal{R}g$ implies $f_|\mathcal{R}g_|$. Then we can still consider the restriction functor \begin{eqnarray*}
R: \mathcal{R}\mathcal{C}^2 &\to& \mathcal{R}\mathcal{C}^2\\
\left(X,A\right) &\mapsto& A\\
\text{$[f]$} &\mapsto& [f_|].
\end{eqnarray*}

Motivated by \cite[Definition 12.1.1, pg. 384]{aguilar2002}, we present the following definition. 

\begin{definition}[Cohomology theory]\label{cohomology-theory} 
Let $\mathcal{C}$ be a category with an initial object, finite limits, finite colimits and covers, along with a distinguished collection of morphisms. Let $\mathcal{A}$ be the category of abelian groups or modules. Let $\mathcal{R}$ be an equivalence relation that preserves composition over the set of morphisms of pairs \[\text{Mor}_{\mathcal{C}^2}\left(\left(X,A\right),\left(Y,B\right)\right)\] for each pair $\left(X,A\right),\left(Y,B\right)$ such that, for $f,g\in \text{Mor}_{\mathcal{C}^2}\left(\left(X,A\right),\left(Y,B\right)\right)$, if $f\mathcal{R}g$, then $f_|\mathcal{R}g_|$. 

A \textit{cohomology theory} $h^\ast$ over $\mathcal{R}\mathcal{C}^2$ consists of a sequence of contravariant functors $h^n:\mathcal{R}\mathcal{C}^2\to \mathcal{A}$ for each $n\in\mathbb{Z}$, along with natural transformations $\delta^n:h^{n}\circ R\to h^{n+1}$ for $n\in\mathbb{Z}$, satisfying the following axioms:
\begin{itemize}
    \item[(C1)] Exactness: For each pair $\left(X,A\stackrel{i}{\to} X\right)\in\mathcal{R}\mathcal{C}^2$, the long sequence   
    \[\cdots\to h^{n-1}(A)\stackrel{\delta^{n-1}\left(X,A\right)}{\to}
    h^n\left(X,A\right)\stackrel{h^n[\iota]}{\to}h^n(X)
    \stackrel{h^n[i]}{\to} h^n(A)
    \stackrel{\delta^n\left(X,A\right)}{\to}\cdots\]
 is exact, where $\iota:X\to \left(X,A\right)$ is the morphism of pairs defined by the identity $1_X$. 
\item[(C2)] Finite covering: If $\{A_\ell\stackrel{i_\ell}{\to} X\}_{\ell=1}^k$ is a finite covering of $X$, then \[h^n\left(X,\displaystyle{\widetilde{\lim_{{\overset{\to}{1\leq \ell\leq k}}}} i_\ell}\right)=0, \text{ for any $n\in\mathbb{Z}$}.\] 
\end{itemize} In this case, we denote the collection $h^\ast=\{h^n,\delta^n\}_{n\in\mathbb{Z}}:\mathcal{R}\mathcal{C}^2\to\mathcal{A}$ as a \textit{$\mathcal{A}$-cohomology theory} or a \textit{cohomology theory with coefficients in $\mathcal{A}$} on $\mathcal{C}$. We will write $f^\ast$ instead of $h^n[f]$ and $\delta$ instead of $\delta^n\left(X,A\right)$. 
\end{definition}

We have the following remark. 

\begin{remark}
 Let $h^\ast=\{h^n:\mathcal{R}\mathcal{C}^2\to\mathcal{A}\}_{n\in\mathbb{Z}}$ be a family of contravariant functors. If $f:\left(X,A\right)\to \left(Y,B\right)$ is an \textit{$\mathcal{R}$-equivalence}, meaning  there exists a morphism of pairs $g:\left(Y,B\right)\to \left(X,A\right)$ such that $(f\circ g)\mathcal{R}1_{\left(Y,B\right)}$ and $(g\circ f)\mathcal{R} 1_{\left(X,A\right)}$ (and, of course, $[f]$ is an isomorphism in $\mathcal{R}\mathcal{C}^2$ whose inverse is $[g]$), then $f^\ast:h^n\left(Y,B\right)\to h^n\left(X,A\right)$ is an isomorphism for each $n\in\mathbb{Z}$. 
\end{remark}

We say that a distinguished morphism $A\stackrel{i}{\to} X$ is a \textit{$\mathcal{R}$-deformation retract} of $X$ if there exists a morphism $r:X\to A$ such that $r\circ i=1_A$ and $(i\circ r)\mathcal{R} 1_X$. 

\begin{proposition}[Cohomology for deformation retracts]\label{cohomology-retratopordefor} 
Let $h^\ast=\{h^n,\delta^n\}_{n\in\mathbb{Z}}:\mathcal{R}\mathcal{C}^2\to\mathcal{A}$ be a family of contravariant functors $h^n$ and natural transformations $\delta^n$ satisfying Axiom (C1). If $A\stackrel{i}{\to} X$ is a $\mathcal{R}$-deformation retract of $X$, then \[h^n\left(X,A\stackrel{i}{\to} X\right)=0 \text{ for each } n\in\mathbb{Z}.\] In particular, for any distinguished isomorphism $A\stackrel{i}{\to} X$, we have  $h^n\left(X,A\stackrel{i}{\to} X\right)=0$ for each $n\in\mathbb{Z}$.
\end{proposition}
\begin{proof}
Let $r:X\to A$ be a $\mathcal{R}$-deformation retraction of $X$, meaning that $r\circ i=1_A$ and $(i\circ r)\mathcal{R} 1_X$. Consider the long exact sequence of the pair $\left(X,A\stackrel{i}{\to} X\right)$: \[\cdots\to h^{n-1}\left(A\right)\stackrel{\delta}{\to}
    h^n\left(X,A\stackrel{i}{\to} X\right)\stackrel{j^{\ast}}{\to}h^n\left(X\right)
    \stackrel{i^{\ast}}{\to} h^n\left(A\right)
    \stackrel{\delta}{\to}\cdots\] Since $i^\ast\circ r^\ast=1_{h^n(A)}$ and $r^\ast\circ i^\ast=1_{h^\ast(X)}$, we find that $i^\ast$ is an isomorphism. Consequently, we have  $\delta=0$ and $j^\ast=0$. 
    
    By the exactness of the sequence, it follows that $h^n\left(X, A\stackrel{i}{\to} X\right)=\text{Ker}~j^\ast=\text{im}~\delta=0$  for each $n\in\mathbb{Z}$.  
\end{proof}

Proposition~\ref{cohomology-retratopordefor} implies the following statement.    

\begin{corollary}\label{cor:lim-cilim-iso}
  Let $h^\ast=\{h^n,\delta^n\}_{n\in\mathbb{Z}}:\mathcal{R}\mathcal{C}^2\to\mathcal{A}$ be a family of contravariant functors $h^n$ and natural transformations $\delta^n$ satisfying the Axiom (C1). Let $\{A_\ell\stackrel{i_\ell}{\to} X\}_{\ell=1}^k$ be a covering of  
$X$. If the morphism $\displaystyle{\widetilde{\lim_{{\overset{\to}{1\leq \ell\leq k}}}} i_\ell}\to X$ (as given in Example~\ref{exam:limit-colimit}(1)) is an isomorphism, then \[h^n\left(X,\displaystyle{\widetilde{\lim_{{\overset{\to}{1\leq \ell\leq k}}}} i_\ell}\right)=0 \text{ for any $n\in\mathbb{Z}$}.\]   \end{corollary} 

Motivated by Corollary~\ref{cor:lim-cilim-iso}, we present the following property concerning covers.

\begin{example}[Lim-colim property]\label{exam:lim-colim-topology} 
 Let $\mathcal{C}$ be a category with finite limits, finite colimits, and covers. A finite covering $\{A_\ell\stackrel{i_\ell}{\to} C\}_{\ell=1}^k$ of $C$ satisfies the \textit{lim-colim property} if  the morphism $\displaystyle{\widetilde{\lim_{{\overset{\to}{1\leq \ell\leq k}}}} i_\ell}\to C$ (as given in Example~\ref{exam:limit-colimit}(1)) is an isomorphism. If every finite covering satisfies the lim-colim property, we say that the category $\mathcal{C}$ with covers satisfies the \textit{lim-colim property}.      
\end{example}

Any category with covers that satisfies the lim-colim property holds the following remark.

\begin{remark}
 Let $\mathcal{C}$ be a category with covers that satisfies the lim-colim property (see Example~\ref{exam:lim-colim-topology}), equipped with an initial object, finite limits, finite colimits, and  distinguished morphisms. Let $\mathcal{A}$ be the category of abelian groups or modules.  
 
 Let $h^\ast=\{h^n,\delta^n\}_{n\in\mathbb{Z}}:\mathcal{R}\mathcal{C}^2\to\mathcal{A}$ be a family of contravariant functors $h^n$ and natural transformations $\delta^n$ satisfying the Axiom (C1). By Corollary~\ref{cor:lim-cilim-iso}, the collection $h^\ast=\{h^n,\delta^n\}_{n\in\mathbb{Z}}:\mathcal{R}\mathcal{C}^2\to\mathcal{A}$ constitutes a $\mathcal{A}$-cohomology theory on $\mathcal{C}$.
\end{remark}

\subsection{Multiplicative cohomology theory} For a cohomology theory $h^\ast=\{h^n,\delta^n\}_{n\in\mathbb{Z}}:\mathcal{R}\mathcal{C}^2\to \mathcal{A}$, we will define the notion of a product over $h^\ast$. The cohomology group $h^\ast\left(X,A\stackrel{i}{\rightarrowtail} X\right)=\bigoplus_{n\in\mathbb{Z}} h^n\left(X,A\stackrel{i}{\rightarrowtail} X\right)$ (equivalently, $\displaystyle{h^\ast\left(X,A\stackrel{i}{\rightarrowtail} X\right)=\{h^n\left(X,A\stackrel{i}{\rightarrowtail} X\right)\}_{n\in\mathbb{Z}}}$) forms a graduated ring for each pair $\left(X,A\stackrel{i}{\rightarrowtail} X\right)\in \mathcal{R}\mathcal{C}^2$. To achieve this, we  will follow the ideas presented in \cite[Chapter 2, Section 6]{kono2006}.

\begin{definition}[Multiplicative cohomology theory]\label{multiplicative-cohomology} 
A cohomology theory $h^\ast=\{h^n,\delta^n\}_{n\in\mathbb{Z}}:\mathcal{R}\mathcal{C}^2\to \mathcal{A}$ is called \textit{multiplicative} if for each $m_1,\ldots,m_k\in\mathbb{Z}$ with $k\geq 2$, and any pairs $\left(X_1, A_1\stackrel{i_1}{\to} X_1\right),\ldots,\left(X_k, A_k\stackrel{i_k}{\to} X_k\right)$, there exists a homomorphism \[\mu_{m_1,\ldots,m_k}:h^{m_1}(X_1,A_1)\otimes\cdots\otimes h^{m_k}(X_k,A_k)\to h^{m_1+\cdots+m_k}\left(\prod_{\ell=1}^{k}X_\ell,\displaystyle{\widetilde{\lim_{{\overset{\to}{1\leq \ell\leq k}}}} i^\ell}\right),\] such that the collection $\mu=\{\mu_{m_1,\ldots,m_k}\}_{m_1,\ldots,m_k\in\mathbb{Z}}$ satisfies the following conditions:
\begin{itemize}
    \item[(1)] Each $\mu_{m_1,\ldots,m_k}$ is natural, meaning that for any morphisms of pairs $f_1:(X_1,A_1)\to (X^\prime_1,A^\prime_1),\ldots,f_k:(X_k,A_k)\to (X^\prime_k,A^\prime_k)$, we have \[(f_1\times\cdots\times f_k)^\ast\circ \mu_{m_1,\ldots,m_k}=\mu_{m_1,\ldots,m_k}\circ (f^\ast_1\otimes\cdots\otimes f^\ast_k).\] This can be illustrated with the following commutative diagram: 
    \begin{eqnarray*}
\xymatrix@C=3cm{ h^{m_1}(X_1,A_1)\otimes\cdots\otimes h^{m_k}(X_k,A_k) \ar[r]^{\mu_{m_1,\ldots,m_k}} & h^{m_1+\cdots+m_k}\left(\prod_{\ell=1}^{k}X_\ell,\displaystyle{\widetilde{\lim_{{\overset{\to}{1\leq \ell\leq k}}}} i^\ell}\right)  \\
       h^{m_1}(X'_1,A'_1)\otimes\cdots\otimes h^{m_k}(X'_k,B') \ar[r]_{\mu_{m_1,\ldots,m_k}}  \ar[u]^{f_1^\ast\otimes\cdots\otimes f^\ast_k} &  h^{m_1+\cdots+m_k}\left(\prod_{\ell=1}^{k}X'_\ell,\displaystyle{\widetilde{\lim_{{\overset{\to}{1\leq \ell\leq k}}}} (i')^\ell}\right)\ar[u]_{\left(f_1\times\cdots\times f_k\right)^\ast}}
\end{eqnarray*} where $f_1\times\cdots\times f_k$ denotes the morphism of pairs $\left(f_1\times\cdots\times f_k,\xi\right)$ (see Example~\ref{exam:product-}(2)).
    
    \item[(2)] (Commutativity) Consider the $i,j$-switching isomorphism for $i\neq j$: $$t:h^{m_1}(X_1,A_1)\otimes\cdots\otimes h^{m_k}(X_k,A_k)\to h^{m_1}(X_1,A_1)\otimes\cdots\otimes h^{m_k}(X_k,A_k)$$ defined by $t(u_1\otimes\cdots\otimes u_i\otimes\cdots\otimes u_j\otimes\cdots\otimes u_k)=(-1)^{m_im_j}u_1\otimes\cdots\otimes u_j\otimes\cdots\otimes u_i\otimes\cdots\otimes u_k.$ Also, consider the isomorphism of pairs  $$T:\left(\prod_{\ell=1}^{k}X_\ell,\displaystyle{\widetilde{\lim_{{\overset{\to}{1\leq \ell\leq k}}}} i^\ell}\right)\to \left(X_1\times\cdots\times X_j\times\cdots\times X_i\times\cdots\times X_k,\displaystyle{\widetilde{\lim_{\to}}(i^1,\ldots,i^j,\ldots,i^i,\ldots,i^k)}\right),$$ where \[\displaystyle{\widetilde{\lim_{\to}}(i^1,\ldots,i^k)}=\displaystyle{\widetilde{\lim_{{\overset{\to}{1\leq \ell\leq k}}}} i^\ell}.\] We have that $\mu$ satisfies the following condition:
    \[\mu_{m_1,\ldots,m_i,\ldots,m_j,\ldots,m_k}=T^\ast\circ\mu_{m_1,\ldots,m_j,\ldots,m_i,\ldots,m_k}\circ t.\] 
\end{itemize}
\end{definition}

In the case that $h^\ast$ is multiplicative, the following composition: 
\[\cup:h^{m_1}(X,A_1)\otimes\cdots\otimes h^{m_k}(X,A_k)\stackrel{\mu}{\to} h^{m_1+\cdots+m_k}\left(\prod_{\ell=1}^{k}X,\displaystyle{\widetilde{\lim_{{\overset{\to}{1\leq \ell\leq k}}}} i^\ell}\right)\stackrel{\Delta^\ast}{\to} h^{m_1+\cdots+m_k}\left(X,\displaystyle{\widetilde{\lim_{{\overset{\to}{1\leq \ell\leq k}}}} i_\ell}\right),\] defines a natural homomorphism for each $m_1,\ldots,m_k\in\mathbb{Z}$, where $\Delta$ is the diagonal morphism (see Example~\ref{exam:diagonal-}(3)). This homomorphism $\cup$ is called the \textit{intern product} or \textit{cup product}, while the homomorphism $\mu$ is referred to as the \textit{exterior product}. 

For each $x_1\in h^{m_1}(X,A_1),\ldots,x_k\in h^{m_k}(X,A_k)$, we denote $\cup(x_1\otimes\cdots\otimes x_k)$ as $x_1\cup\cdots\cup x_k$, and $\mu(x_1\otimes\cdots\otimes x_k)$ as  $x_1\times\cdots\times x_k$.

\begin{remark}
  \noindent\begin{enumerate}
       \item[(1)] One has \[\mu=\cup\circ (p_1^\ast\otimes\cdots\otimes p_k^\ast),\] which means that  the following diagram commutes:
       \begin{eqnarray*}
\xymatrix{h^{m_1}(X_1,A_1)\otimes\cdots\otimes h^{m_k}(X_k,A_k) \ar[rr]^{p_1^\ast\otimes\cdots\otimes p_k^\ast} \ar[drr]_{\mu}&  &h^{m_1}\left(\prod_{\ell=1}^{k}X_\ell,A^1\right)\otimes\cdots\otimes h^{m_k}\left(\prod_{\ell=1}^{k}X_\ell,A^k\right) \ar[d]^{\cup}  \\
        & &  h^{m_1+\cdots+m_k}\left(\prod_{\ell=1}^{k}X_\ell,\displaystyle{\widetilde{\lim_{{\overset{\to}{1\leq \ell\leq k}}}} i^\ell}\right)}
\end{eqnarray*} where each $p_i:\left(\prod_{\ell=1}^{k}X_\ell,A^i\right)\to (X_i,A_i)$ is the $i$-th projection. 
\item[(2)] Suppose that $A\stackrel{i}{\rightarrowtail} X$ is a monomorphism. Note that the cup product $\cup:h^m(X,A)\otimes h^n(X,A)\to h^{m+n}(X,A)$ is associative. This means that the following diagram commutes:
\begin{eqnarray*}
\xymatrix@C=3cm{ h^{m}(X,A)\otimes h^{n}(X,A)\otimes h^{k}(X,A) \ar[r]^{\cup\otimes 1_{h^{k}(X,A)}} \ar[d]^{1_{h^{m}(X,A)}\otimes\cup}&  h^{m+n}(X,A)\otimes h^{k}(X,A)\ar[d]^{\cup}  \\
       h^{m}(X,A)\otimes h^{n+k}(X,A)\ar[r]^{\cup} &  h^{m+n+k}(X,A)}
\end{eqnarray*}
As a consequence, the cohomology group
    \[h^\ast(X,A)=\bigoplus_{n\in\mathbb{Z}}h^n(X,A)\] together with the cup product $\cup$ forms a graduated ring. Furthermore, since $\mu$ is commutative, we have \[u\cup v=(-1)^{mn}v\cup u\] for any $u\in h^m(X,A)$ and $v\in h^n(X,A)$. Thus $h^\ast(X,A)$ is a graded commutative ring. 
\item[(3)] In the case that $I\stackrel{}{\rightarrowtail} X$ and $A\stackrel{i}{\rightarrowtail} X$ are monomorphisms, we have \[\cup:h^{m}\left(X\right)\otimes h^{n}\left(X,A\stackrel{i}{\rightarrowtail} X\right)\to h^{m+n}\left(X,A\stackrel{i}{\rightarrowtail} X\right)\] for any $m,n\in\mathbb{Z}$. Consequently, the cohomology group \[h^\ast\left(X,A\stackrel{i}{\rightarrowtail} X\right)=\bigoplus_{n\in\mathbb{Z}}h^n\left(X,A\stackrel{i}{\rightarrowtail} X\right)\] becomes a graduated module over the graduated ring $h^\ast(X)$. 
    \end{enumerate}
\end{remark} 

\medskip The following is a key property of the cup product. 

\begin{proposition}[Cup product]\label{diagonal-cup} Let $\left(X,A_1\stackrel{i_1}{\to} X\right),\ldots,\left(X,A_k\stackrel{i_k}{\to} X\right)$ be pairs, with $X\stackrel{\iota_1}{\to}(X,A_1),\ldots,X\stackrel{\iota_k}{\to}(X,A_k)$ and $X\stackrel{\iota}{\to }\left(X,\displaystyle{\widetilde{\lim_{{\overset{\to}{1\leq \ell\leq k}}}} i_\ell}\right)$ representing the usual morphisms of pairs. Then, for any $u_1\in h^{n_1}(X,A_1),\ldots,u_k\in h^{n_k}(X,A_k)$, we have: \[\iota^\ast(u_1\cup\cdots\cup u_k)=\iota_1^\ast(u_1)\cup\cdots\cup \iota_k^\ast(u_k).\] 
\end{proposition}
\begin{proof}
Let $I$ be the initial object of $\mathcal{C}$, and $i_I:I\to X$ be the unique morphism from $I$ to $X$. We have the following commutative diagram:
 \begin{eqnarray*}
\xymatrix@C=3cm{ X \ar[r]^{\,\,\iota\,\,} \ar[d]_{\Delta} & \left(X,\displaystyle{\widetilde{\lim_{{\overset{\to}{1\leq \ell\leq k}}}} i_\ell}\right)\ar[d]^{\Delta}  \\
       \left(\prod_{\ell=1}^{k}X,\displaystyle{\widetilde{\lim_{{\overset{\to}{1\leq \ell\leq k}}}} (i_I)^\ell}\right)  \ar[r]^{\iota_1\times\cdots\times \iota_k} &  \left(\prod_{\ell=1}^{k}X,\displaystyle{\widetilde{\lim_{{\overset{\to}{1\leq \ell\leq k}}}} i^\ell}\right) }
\end{eqnarray*} Here, $\iota_1\times\cdots\times \iota_k$ denotes the morphism of pairs given in Example~\ref{exam:product-}(2). This leads us to the following commutative diagram: 
\begin{eqnarray*}
\xymatrix@C=3cm{ h^{n_1+\cdots+n_k}(X)  & h^{n_1+\cdots+n_k}\left(X,\displaystyle{\widetilde{\lim_{{\overset{\to}{1\leq \ell\leq k}}}} i_\ell}\right)\ar[l]_{\,\,\iota^\ast\,\,}  \\
       h^{n_1+\cdots+n_k}\left(\prod_{\ell=1}^{k}X,\displaystyle{\widetilde{\lim_{{\overset{\to}{1\leq \ell\leq k}}}} (i_I)^\ell}\right)\ar[u]_{\Delta^\ast} &  h^{n_1+\cdots+n_k}\left(\prod_{\ell=1}^{k}X,\displaystyle{\widetilde{\lim_{{\overset{\to}{1\leq \ell\leq k}}}} i^\ell}\right)\ar[u]_{\Delta^\ast}\ar[l]_{\quad\quad(\iota_1\times\cdots\times \iota_k)^\ast\quad\quad\quad}  \\
      h^{n_1}(X)\otimes\cdots\otimes h^{n_k}(X)\ar[u]_{\mu} & h^{n_1}(X,A_1)\otimes\cdots\otimes h^{n_k}(X,A_k)\ar[u]_{\mu} \ar[l]_{\,\,\iota_1^\ast\otimes\cdots\otimes \iota_k^\ast\,\,}
       }
\end{eqnarray*}
 This implies the equality: \[\iota^\ast(u_1\cup\cdots\cup u_k)=\iota_1^\ast(u_1)\cup\cdots\cup \iota_k^\ast(u_k)\] for any $u_1\in h^{n_1}(X,A_1),\ldots,u_k\in h^{n_k}(X,A_k)$.   
\end{proof}

We have the following example. 

\begin{example}[Multiplicative generalized cohomology theory]\label{exam:multiplicative-generalized-cohomology-theory}
   One can observe that the category $\text{Top}$ has the empty space $\varnothing$ as an initial object, along with finite limits and finite colimits. We consider the open quasi Grothendieck topology,  Open, as presented in Example~\ref{exam:open-cover-space}. Additionally, we take any inclusion map as a distinguished morphism. Let $\mathcal{R}$ (which we also denote as $H$) be the homotopy relation $\simeq$ between maps of pairs. 
   
   Any multiplicative generalized cohomology theory $h^\ast=\{h^n,\delta^n\}_{n\in\mathbb{Z}}:\mathrm{H}\mathrm{Top}^2\to\mathcal{A}$ (as detailed in \cite[Chapter 2, Section 6]{kono2006}) constitutes a $\mathcal{A}$-multiplicative cohomology theory on $(\text{Top},\text{Open})$. Furthermore, it is noteworthy that the category with covers $(\text{Top},\text{Open})$ satisfies the lim-colim property. Specifically, for any finite covering $\{A_\ell\stackrel{i_\ell}{\hookrightarrow} X\}_{\ell=1}^k$ of $X$, we have: \[\displaystyle{\widetilde{\lim_{{\overset{\to}{1\leq \ell\leq k}}}} i_\ell}=\bigcup_{\ell=1}^{k} A_\ell=X,\] and the morphism $\displaystyle{\widetilde{\lim_{{\overset{\to}{1\leq \ell\leq k}}}} i_\ell}\to X$ is the identity map $1_X$ .
\end{example}

\subsection{Cohomological lower bound} As a principal result of this section, we present a cohomological lower bound for the $\mathcal{R}$-sectional number of a morphism, which is a tool widely used in computations. This arises as follows, following the notation established previously. 

\begin{theorem}[Cohomological lower bound]\label{thm:cohomo-ic} 
Let $\mathcal{C}$ be a category with an initial object, finite limits, finite colimits, covers, and distinguished morphisms. Let $h^\ast=\{h^n,\delta^n\}_{n\in\mathbb{Z}}:\mathcal{R}\mathcal{C}^2\to\mathcal{A}$ be a $\mathcal{A}$-multiplicative cohomology theory on $\mathcal{C}$. For any morphism $f:X\to Y$ in $\mathcal{C}$, we have
\[\mathrm{nil}\left(\mathrm{Ker}\left(f^\ast:h^\ast(Y)\to h^\ast(X)\right)\right)\leq\mathcal{R}\text{-}\mathrm{sec}(f).\]
\end{theorem}
\begin{proof}
  Let $\mathcal{R}\text{-}\mathrm{sec}(f)=m<\infty$ and consider a covering $\{j_\ell:B_\ell\to Y\}_{\ell=1}^{m}$ of $Y$ such that for each $\ell=1,\ldots,m$, the morphism $f$ admits a $\mathcal{R}$-$j_\ell$-section:
  \begin{eqnarray*}
\xymatrix{ X \ar[rr]^{f} & &Y   \\
        &  B_\ell\ar@{-->}[lu]^{}\ar[ru]_{j_\ell} & } 
\end{eqnarray*}
Suppose that $\mathrm{nil}\left(\mathrm{Ker}\left(f^\ast:h^\ast(Y)\to h^\ast(X)\right)\right)>m$. Then, there exist $\alpha_\ell\in h^{n_\ell}(Y)$ such that $f^\ast(\alpha_\ell)=0$ for each $\ell=1,\ldots,m$ and  $\alpha_1\cup\cdots\cup\alpha_m\neq 0$. 

For each $\ell=1,\ldots,m$, we have the following commutative diagram:
    \begin{eqnarray*} 
\xymatrix{ h^{n_\ell}(X)\ar[rd]_{}  & & h^{n_\ell}(Y)\ar[ll]_{f^\ast}  \ar[dl]^{\left(j_\ell\right)^\ast}  \\
        &  h^{n_\ell}(B_\ell) & } 
\end{eqnarray*} Thus, $\left(j_\ell\right)^\ast(\alpha_\ell)=0$.

Now, for each $\ell=1,\ldots,m$, we consider the long exact sequence associated with the pair $\left(Y,B_\ell\stackrel{j_\ell}{\to} Y\right)$:
    \[\cdots\to h^{n_\ell}\left(Y,B_\ell\right)\stackrel{\iota_\ell^\ast}{\to}
    h^{n_\ell}(Y)\stackrel{\left(j_\ell\right)^\ast}{\to}h^{n_\ell}\left(B_\ell\right)
    \stackrel{\delta}{\to} h^{n_\ell+1}\left(Y,B_\ell\right)\to\cdots\]
 where $\iota_\ell:Y\to \left(Y,B_\ell\right)$ is the usual morphism. Therefore, $\alpha_\ell\in \mathrm{Ker}\left(\left(j_\ell\right)^\ast\right)=\mathrm{im}\left(\iota_\ell^\ast\right)$, and thus there exists a cohomology class $\widetilde{\alpha_\ell}\in h^{n_\ell}\left(Y,B_\ell\right)$ such that $\iota_\ell^\ast(\widetilde{\alpha_\ell})=\alpha_\ell$.
 
Consequently, we have $\widetilde{\alpha_1}\cup\cdots\cup \widetilde{\alpha_m}=0$ because $\widetilde{\alpha_1}\cup\cdots\cup \widetilde{\alpha_m}\in h^{n_1+\cdots+n_m}\left(Y,\displaystyle{\widetilde{\lim_{{\overset{\to}{1\leq \ell\leq m}}}} j_\ell}\right)=0$ (here we use the fact that $\{B_\ell\stackrel{j_\ell}{\to} Y\}_{\ell=1}^m$ is a finite covering of $Y$ together with Axiom (C2)). Hence, we have: \[0=\iota^\ast\left(\widetilde{\alpha_1}\cup\cdots\cup\widetilde{\alpha_m}\right)=\alpha_1\cup\cdots\cup\alpha_m.\] The last equality follows from Proposition~\ref{diagonal-cup}, leading to a contradiction. Therefore, we conclude that: \[\mathrm{nil}\left(\mathrm{Ker}\left(f^\ast:h^\ast(Y)\to h^\ast(X)\right)\right)\leq m=\mathcal{R}\text{-}\mathrm{sec}(f).\] 
\end{proof}   

In concrete cases, such as those explored in standard sectional theory, we do not aim to compute the entire kernel of the homomorphism
$f^\ast:h^\ast(Y)\to h^\ast(X)$. Instead, we focus on identifying 
 specific elements within the kernel and seek to establish long non-trivial products. 

\medskip Example~\ref{exam:multiplicative-generalized-cohomology-theory} and Theorem~\ref{thm:cohomo-ic} together imply the property (6) mentioned in the introduction.

\begin{corollary}[Cohomological lower bound for the usual sectional theory]
 Let $h^\ast=\{h^n,\delta^n\}_{n\in\mathbb{Z}}:\mathrm{H}\mathrm{Top}^2\to\mathcal{A}$ be any multiplicative generalized cohomology theory. For a map  $f:X\to Y$, we have  \[\mathrm{nil}\left(\mathrm{Ker}\left(f^\ast:h^\ast(Y)\to h^\ast(X)\right)\right)\leq\text{secat}(f).\]
\end{corollary}

Theorem~\ref{thm:cohomo-ic} together with Example~\ref{exam:sectional-number-Set} leads to the following example.

\begin{example}
Consider the topology on Set as presented in Example~\ref{exam:cover-set}, along with a fixed collection of distinguished morphisms. If $h^\ast=\{h^n,\delta^n\}_{n\in\mathbb{Z}}:\text{Set}^2\to\mathcal{A}$ is any $\mathcal{A}$-multiplicative cohomology theory on Set, then the kernel of the homomorphism \[\left(\pi_1^P\right)^\ast:h^\ast\left(P\right)\to h^\ast\left(F_{<}(P)\right)\] is trivial. Here, $P$ denotes the set of prime numbers, $F_{<}(P)=\{(p,p')\in P\times P:~p<p'\}$, and $\pi_1^P:F_{<}(P)\to P$ is defined by $\pi_1^P(p,p')=p$ (see Example~\ref{exam:sectional-number-Set}).    
\end{example}


    

\bibliographystyle{plain}

\end{document}